\documentclass[12pt]{amsart}
\usepackage{amsmath,amsfonts,amssymb}
\usepackage{enumerate}
\usepackage{xcolor}
\usepackage[margin=0.83in]{geometry}
\usepackage{setspace}
\usepackage{esint}

\allowdisplaybreaks

\newtheoremstyle{newline}
  {3pt}
  {3pt}
  {\itshape}
  {}
  {\bfseries}
  {.}
  {.5em}
  {}
\theoremstyle{newline}\newtheorem{theorem}{Theorem}[section]
\theoremstyle{newline}\newtheorem{corollary}{Corollary}[section]
\theoremstyle{newline}\newtheorem{proposition}{Proposition}[section]
\theoremstyle{newline}\newtheorem{lemma}{Lemma}[section]
\theoremstyle{newline}
\theoremstyle{newline}\newtheorem*{remark}{Remark}
\newtheoremstyle{newlinedefn}
  {3pt}
  {3pt}
  {}
  {}
  {\bfseries}
  {.}
  {.5em}
  {}
\theoremstyle{newlinedefn}

 \numberwithin{equation}{section} 
 
\newcommand{\bu}{\boldsymbol{u}}
\newcommand{\bv}{\boldsymbol{v}}
\newcommand{\bw}{\boldsymbol{w}}

\newcommand{\boldf}{\boldsymbol{f}}
\newcommand{\bg}{\boldsymbol{g}}

\newcommand{\beps}{\boldsymbol{\varepsilon}}

\newcommand{\phib}{\boldsymbol{\phi}}

\newcommand{\bbT}{\mathbf{T}}
\newcommand{\bbS}{\mathbf{S}}
\newcommand{\bbU}{\mathbf{U}}

\newcommand{\diver}{\mathrm{div}}

\begin{document}
\onehalfspacing
\title[Nonlinear dynamic fracture problems]{Nonlinear dynamic fracture problems with polynomial and strain-limiting constitutive relations}
\author{Victoria Patel}
\email{victoria.patel@maths.ox.ac.uk}

\thanks{V. Patel is supported by the UK Engineering and Physical Sciences Research Council [EP/L015811/1].}
\maketitle
{\footnotesize
 \centerline{Mathematical Institute, University of Oxford}
   \centerline{Andrew Wiles Building, Woodstock Road }
   \centerline{Oxford OX2 6GG, UK}
   }

\begin{abstract}
We extend the framework of dynamic fracture problems with a phase-field approximation to the case of a nonlinear constitutive relation between the Cauchy stress tensor \( \bbT \), linearised strain \( \beps(\bu) \) and strain rate \( \beps(\bu_t ) \). The relationship takes the form \( \beps(\bu_t) + \alpha\beps(\bu) = F(\bbT) \) where \( F \) satisfies certain \(p\)-growth conditions. We take particular care to study the case \(p=1\) of a `strain-limiting' solid, that is, one in which the strain is bounded {\it a priori}. We prove the existence of long-time, large-data weak solutions of a balance law coupled with a minimisation problem for the phase-field function and an energy-dissipation inequality, in any number \( d\) of spatial dimensions. In the case of Dirichlet boundary conditions, we also prove the satisfaction of an energy-dissipation equality. 
\end{abstract}

\section{Introduction}\label{sec:intro}

In this work, we are interested in combining the theory of dynamic fracture with the theory of nonlinear, implicit constitutive relations. At present, mathematical analysis of dynamic fracture problems has only been performed in the setting of linear elasticity to the author's knowledge. In this work, we aim to extend such fracture studies and examine time-dependent damage problems with a constitutive relation from nonlinear viscoelasticity. Of particular interest is the strain-limiting case, by which we mean that the linearised strain \( \beps(\bu) = \frac{1}{2}(\nabla \bu + (\nabla \bu)^{\mathrm{T}})\) is bounded {\it a priori}. A key motivation for studying such models is that such a defining relation ensures that no singularity is experienced in the strain. Indeed, we cannot contradict the small strain assumption, which is essential to the construction of the model here. 
This contrasts the known results for damage models with a linear constitutive relation where a singularity occurs at the crack tip \cite{onthenonlinearelasticresponse}.  
In this section, we discuss the theory of implicit constitutive relations that was established by Rajagopal in his pioneering work \cite{onICT} and a series of later papers, followed by an introduction to the theory of dynamic fracture based on the principles in \cite{mott1948brittle, larsen2010models}. Then we combine the two theories to obtain a fracture problem with a nonlinear constitutive relation between the Cauchy stress tensor, linearised strain tensor and strain rate. 
We particularly emphasise the work in Section \ref{sec:a} where strain-limiting models are studied in the context of dynamic fracture problems, the state of the art in this work. 

A  framework for nonlinear, implicit constitutive relations in continuum mechanics was introduced by Rajagopal in his pioneering work \cite{onICT} and a series of later papers. Although such relations  had been studied previously, this was only in an {\it ad hoc} way. The novelty of \cite{onICT} is the systematic and thermodynamically consistent structure that is available for the family of constitutive relations. These ideas are extended in \cite{RN20}  where nonlinear relationships between the Cauchy stress tensor \( \bbT\), strain tensor \( \beps\) and strain rate \( \beps_t\) are justified in the small strain setting. Namely, we may consider the defining relation
\begin{equation}\label{intro:equ14}
\beps(\bu_t + \alpha\bu)  = F(\bbT),
\end{equation}
where \( F\) is a nonlinear, monotonic function from \( \mathbb{R}^{d\times d}_{sym}\) to itself and \( \alpha\) is a positive constant. We let \( \beps = \frac{1}{2}((\nabla \cdot) + (\nabla \cdot)^{\mathrm{T}}) \) denote the symmetric gradient operator, \( \bu \) is the displacement and \( \bbT \) is the Cauchy stress tensor. The relation (\ref{intro:equ14}) is a generalisation of the classical linear Kelvin--Voigt model for viscoelastic bodies. 

A subclass of these models that we take particular care to study here is that of a strain-limiting solid. A strain-limiting constitutive relation ensures that the strain is bounded {\it a priori}, irrespective of the magnitude of the stress tensor. 
The motivation to study such models is as follows. 
It is well known that brittle materials can fracture in the small strain range under large stresses. 
Furthermore, under the assumption that the strain is small, the mathematical model that results from the balance laws of momentum and mass for the deformation of the body under investigation can be linearised. As we are considering the motion of solids, it is sensible to assume that we are working in the small strain setting. Indeed, solids damage and break under sufficiently large strains. 
Moreover, the linearisation procedure as a result of the small strain assumption simplifies the problem and makes it more amenable to mathematical analysis. 
Hence, from now on, it is assumed that the small displacement gradient assumption holds, that is, 
\begin{equation}\label{intro:smalldisp}
\sup_{(t,x) \in Q} |\nabla \bu(t,x) | = O(\delta),
\end{equation}
for a \( \delta\ll 1\), where \( Q\) is the space-time domain. It follows from (\ref{intro:smalldisp}) that the linearised strain \( \beps(\bu) \) as defined above must  be at most of order \( \delta\). 

Now, consider a body with a crack which is possibly growing due to an external force acting on the material. 
In the case of a linear constitutive relation,   the strain has magnitude of order \(r^{-\frac{1}{2}} \) where \( r\) is the distance to the crack tip \cite{onthenonlinearelasticresponse}. The strain experiences a singularity at the crack tip, contradicting the standing assumption that the strain is small. A more reliable model could be one in which the magnitude of the strain is bounded {\it a priori} so such contradictions cannot  arise. 

This motivates studying the constitutive relation
\begin{equation}\label{intro:equ1}
\beps(\bu_t + \alpha\bu) = F(\bbT) := \frac{\bbT }{(1  + |\bbT|^a)^{\frac{1}{a}}},
\end{equation}
for fixed parameter \( a> 0 \). 
Other choices for the function \( F \) are available, but for simplicity of the presentation, we make this specific choice as suggested in \cite{RN21, RN19}. For further details of the possible generalisations, we refer to the remark in Section \ref{sec:a}. 

The defining relation (\ref{intro:equ1}) is coupled with the balance law
\begin{equation}\label{intro:equ9}
\bu_{tt} = \diver(\bbT) + \boldf,
\end{equation}
where \( \boldf\) represents the external force,  alongside suitable boundary and initial conditions.
The balance law comes from considering the balance of linear momentum under (\ref{intro:smalldisp}). The balance of angular momentum reduces to \( \bbT = \bbT^{\mathrm{T}}\), which is naturally included in (\ref{intro:equ1}). Due to the assumption (\ref{intro:smalldisp}), we assume that the density is constant in time and so we do not include a balance of mass. 
For further details on the relevant kinematics and the derivation of this problem, we refer to \cite{mypaperpreprint} and the references therein. We note that the strain-limiting problem (\ref{intro:equ1}), (\ref{intro:equ9}) has been recently studied on domains without a crack. 
For the analysis of the spatially periodic problem, we refer to \cite{mypaperpreprint} and for a general domain with Dirichlet boundary conditions, we mention \cite{preprint2}. 

In this paper, we consider  (\ref{intro:equ1}), (\ref{intro:equ9}) in the context of  dynamic fracture. 
Hence, before proceeding further, we  give a brief introduction to the modern theory of fracture. 
The  principles that drive the mathematical analysis of fracture in brittle bodies can be traced back to the seminal work of Griffith \cite{griffith1921vi}. In the quasi-static setting\footnote{By quasi-static setting, we mean that the intertial term \( \bu_{tt}\) is negligible in the balance of momentum equation (\ref{intro:equ9}).}, he stipulated that a crack can only grow if the increase in surface energy is exactly equal to the resulting decrease in elastic potential energy. 
In other words, the newly created area is proportional to the loss of stored energy. 
In the anti-plane setting\footnote{By anti-plane setting, we mean that, in three spatial dimensions, the displacement \( \bu \) takes the particular form \( \bu(x_1,x_2,x_3) = (0, 0, u(x_1,x_2))\). } with a linear relationship between the stress and strain, we write this mathematically as
\begin{equation}\label{intro:equ2}
\frac{\mu }{2} \int_{\Omega\setminus C(t)} |\nabla u(t) |^2 \,\mathrm{d}x + G_c \mathcal{H}^{d-1}(C(t)) = \frac{\mu }{2} \int_{\Omega\setminus C(0)} |\nabla u(0) |^2 \,\mathrm{d}x + G_c \mathcal{H}^{d-1} ( C(0)) ,
\end{equation} 
where  \( C(t) \subset \Omega\) is the crack set, \( \mu \) is the elasticity constant and \( G_c \) is the fracture toughness. 
We use \( \mathcal{H}^{d-1}\) to denote the \((d-1)\)-dimensional Hausdorff measure. 
The first term in (\ref{intro:equ2}) corresponds to the elastic energy and the second term to the surface energy. This formulation was first introduced by Francfort and Marigo \cite{MR1633984}, based on the Mumford--Shah functional that is used in image segmentation \cite{MR997568}.

Alongside this energy balance,  the  balance of linear momentum in the quasi-static setting must hold. We formulate this as a variational problem. Indeed, the couple \((u(t), C(t)) \) must satisfy
\begin{equation}\label{intro:equ11}
\mathcal{A}(u(t), C(t)) := \frac{\mu }{2} \int_{\Omega\setminus C(t)} |\nabla u(t) |^2 \,\mathrm{d}x + G_c \mathcal{H}^{d-1} (C(t)) \leq \frac{\mu }{2} \int_{\Omega\setminus \tilde{C}} |\nabla \tilde{u}|^2 \,\mathrm{d}x + G_c \mathcal{H}^{d-1} (\tilde{C}) ,
\end{equation}
for every suitable comparison couple \((\tilde{u}, \tilde{C}) \) with \( C(t) \subset \tilde{C}\). 
The principle of irreversibility is also required to hold, that is, the crack is non-decreasing in time, so \(C(s) \subset C(t) \) for every \( s< t\). Problems of the form (\ref{intro:equ2}), (\ref{intro:equ11}), where we work directly with the crack set, are called sharp interface problems. 
Mathematical analysis of this anti-plane shear problem took place in \cite{MR1633984}, with higher dimensional settings treated in \cite{MR1897378} and \cite{MR1988896}. Both the crack set and displacement are treated as unknowns in these works. This causes significant analytical difficulties. Indeed, the analysis takes places on the set \(\cup_{t\geq 0} \{t\}\times ( \Omega\setminus C(t))\), which is not only a domain that changes in time, but is {\it a priori} unknown because the crack set \((C(t))_{t\geq 0}\) is unknown. 

An approach to   avoid working on time dependent cracking domains is to approximate the crack set by a phase-field function. The Ambrosio--Tortorelli  functional \cite{MR1113814, MR1164940} takes the following form in the context of quasi-static fracture for linear elasticity:
\begin{equation}\label{intro:equ10}
\mathcal{A}_\epsilon (u, v) = \frac{\mu }{2} \int_\Omega (v^2 + \eta_\epsilon) |\nabla u |^2 \,\mathrm{d}x + G_c \int_\Omega\epsilon|\nabla v|^2  + \frac{1}{4\epsilon} (1 -v )^2\,\mathrm{d}x,
\end{equation}
where \( \epsilon>0\) is an approximation parameter, \( \eta_\epsilon\) is a numerical stability parameter such that \( 0 < \eta_\epsilon\ll \epsilon\),  and \( v\) is the unknown phase-field function, also referred to in the literature as the material parameter. The function \( v \) approximates the crack set in the sense that \( v \) takes values close to \( 1\) away from the crack and  values close to \( 0 \) near  the crack. 
Although \( \epsilon\) acts as a length scale for the approximation, 
we must be careful not to assign a physical meaning to this parameter \cite{BourdinBlaise2011Atmf}. For the mathematical analysis of this approximation in the quasi-static setting, we refer to  \cite{KUHN20103625} and \cite{MR2106765}, for example. 

The approximation (\ref{intro:equ10}) is a sensible choice because \( \mathcal{A}_\epsilon\) converges to the sharp-interface functional \( \mathcal{A}\) in the sense of \( \Gamma\)-convergence. We refer to \cite{MR1968440} for the details but the essence is as follows. Let \(((u^\epsilon, v^\epsilon) )_\epsilon\) be a sequence such that \(v^\epsilon\) minimises \( v\mapsto\mathcal{A}_\epsilon( u^\epsilon, v) \) over suitable test functions \( v\leq v_\epsilon\) at every time. Then there exists a subsequence in \( \epsilon\) and a  solution couple \( (u, C)\) to  (\ref{intro:equ2}), (\ref{intro:equ11}) such that, for a.e. time \( t\), we have
\begin{align*}
&v^\epsilon\nabla u^\epsilon(t) \rightarrow\nabla u(t) \quad\text{ strongly in }L^2(\Omega)^d,
\\
&\lim_{\epsilon\rightarrow 0+} \int_\Omega((v^\epsilon)^2 + \eta_\epsilon) |\nabla  u^\epsilon|^2 \,\mathrm{d}x = \int_\Omega|\nabla u(t) |^2\,\mathrm{d}x
\\
&\lim_{\epsilon\rightarrow 0+ }  \int_\Omega\epsilon|\nabla v|^2  + \frac{1}{4\epsilon} (1 -v )^2\,\mathrm{d}x = \mathcal{H}^{d-1}( C(t)).
\end{align*}
 Hence, the sequence of solutions for the approximate problem represents an approximation to the solution of the sharp-interface functional. 

We remark that there are other choices of approximation available. For example, the  fourth order functional 
\begin{align*}
\tilde{\mathcal{A}}_\epsilon(u, v) = \frac{\mu }{2} \int_\Omega(v^2 + \eta_\epsilon) |\nabla u|^2 \,\mathrm{d}x + \frac{G_c}{4}\int_\Omega\frac{1}{\epsilon} (1-v)^2 + 2\epsilon|\nabla v|^2 + \epsilon^3 |\Delta v|^2 \,\mathrm{d}x,
\end{align*}
is investigated in \cite{BORDEN2014100} and \cite{MR4060061}. The functional  \( \Gamma\)-converges to \( \mathcal{A}\). An advantage of this choice is that the higher order spatial derivatives improve the convergence rate of numerical solutions. In particular, when studying these fracture problems from the point of view of computational analysis, we note that the Ambrosio--Tortorelli functional is not necessarily the best choice of approximation for the sharp-interface functional. 

The aforementioned literature all concern the quasi-static problem. Here, we are interested in the adaptation of Griffith's theory to the dynamic problem, that is,  we do not assume that the  term \( \bu_{tt}\) is negligible. An extension was first introduced by Mott \cite{mott1948brittle} and relies on the following three principles:
\begin{itemize}
\item[(i)] elastodynamics,  the balance of linear momentum should hold away from the crack set;
\item[(ii)] energy balance, an energy-dissipation balance that includes the kinetic energy should hold;
\item[(iii)] irreversibility, the crack cannot heal itself in time. 
\end{itemize}
However, this is not sufficient for a well-posed problem. The case of a stationary crack should be ruled out. The following additional principle was introduced by Larsen \cite{larsen2010models}:
\begin{itemize}
\item[(iv)] maximal dissipation, that is, if the crack can grow while satisfying (i)--(iii), then it should.
\end{itemize}
As with the quasi-static problem, the study of dynamic fracture problems broadly follows one of two approaches: working directly in the sharp-interface setting, or an approximation by means of a phase-field function, generally based on the Ambrosio--Tortorelli functional. The focus in this paper is on the use of a phase-field approximation, thus avoiding the technicalities of working in a time dependent domain. However, for completeness, we first discuss the known results in the sharp-interface case.

For sharp-interface problems, the analysis is significantly more complicated compared to the quasi-static model due to the presence of the intertial term. Hence, a  first step in the analysis is to solve the elastodynamic equation for an unknown displacement when the crack set \((C(t))_{t>0 }\) is known {\it a priori}. We look for a function \( \bu \) such that
\begin{equation}\label{intro:equ12}
\bu_{tt} = \diver\big( \mathbb{E}(t,x) \nabla \bu \big) + \boldf (t,x) \quad \text{ in }\Omega\setminus C(t), \quad t\in [0, T],
\end{equation}
where \( \mathbb{E}\) is the elasticity tensor, given a sequence of subsets \((C(t))_{t\geq 0 }\) of \( \Omega\) satisfying some regularity requirements. In an attempt to better understand (\ref{intro:equ12}), in \cite{MR2847479} for the scalar-valued case and \cite{MR4117509} for the vector-valued case, the wave equation
\begin{align*}
\bu_{tt} - \Delta \bu - \gamma \Delta \bu_t = \boldf
\end{align*}
is studied, where  \( \gamma \in \{0, 1\}\). The system is damped if \( \gamma = 1\) and undamped if \( \gamma = 0 \). The viscous term from the damping provides a certain smoothing effect. A result of this is that, if \( \gamma =1\),  uniqueness and an energy-dissipation balance can be shown. However, uniqueness and satisfaction of an appropriate energy-dissipation balance is an open problem in the undamped setting. 
We note that in  both  \cite{MR2847479}  and \cite{MR4117509}, the existence proofs use a discretisation in the time variable. This is the standard approach to tackle such problems, and we use a similar approach here.

However,  different approaches to prove the existence of solutions are available under  stronger regularity requirements on the crack set. In \cite{MR2306785}, \cite{MR3748493} and, for the vector-valued case, \cite{MR3654907}, a co-ordinate change is used to transform the problem from one on \( \cup_{t\geq 0} \{ t\}\times (\Omega\setminus C(t)) \),  a cracking domain,  to one on the reference domain \( [0, \infty) \times (\Omega\setminus C(0))  \). A solution is found on this reference domain and it is shown  that the existence of such a solution to the transformed problem is equivalent to the existence of a solution to the original problem on the cracking domain. 

The next step in the analysis is to deal with the case where both the crack set and displacement are {\it a priori} unknown. As far as the author is aware, the only known result in the multi-dimensional setting is in \cite{MR3541897}. Under some assumptions on the regularity and geometry of the crack set, the authors show that among all possible couples \((u(t), C(t))_{t\geq 0 }\), there exists one which satisfies the  maximal dissipation condition. There is no damping term present, which is an advantage of the analysis compared to the aforementioned works.  Although a significant amount remains unknown, this work suggests that the maximal dissipation condition is the correct principle to impose on the dynamic problem. 

An important aspect of all of the aforementioned literature is that only linear elasticity is considered, that is, a linear relationship between the Cauchy stress tensor and linearised strain. This leads to a significant gap in the literature. Indeed, to the author's knowledge, there has been no study of dynamic fracture problems in the context of nonlinear constitutive relations. 
In this work, we aim to extend some of the known results to such models, considering both a general case and that of a strain-limiting solid. To avoid the technicalities that come with the sharp-interface model, we  consider a dynamic problem with a phase-field approximation based on the Ambrosio--Tortorelli functional. A regularised model of the type that we study in this work was first introduced in \cite{BourdinBlaise2011Atmf, larsen2010models, MR2673410}. The formulation is based on the principles of elastodynamics, energy-balance and irreversibility, but not maximal dissipation as in \cite{MR3541897}. 
This is due to the fact that maximal dissipation is naturally guaranteed by the formulation. Indeed, the minimisation problem (\ref{intro:equ4}) introduced below ensures that the `crack' grows whenever it is possible for it to do so. 

In \cite{MR2673410}, the authors solve the equation
\begin{equation}\label{intro:equ3}
\bu_{tt} = \diver\big( (v^2 + \eta_\epsilon) \mathbb{E}\beps(\bu + \bu_t)\big) + \boldf \quad \text{ in } (0, T) \times \Omega,
\end{equation}
subject to the couple \((\bu, v) \) solving the  minimisation problem
\begin{align}\label{intro:equ4}
\mathcal{A}_\epsilon(\bu(t), v(t)) &:= \frac{1}{2}\int_\Omega (v(t)^2 + \eta_\epsilon)\mathbb{E}\beps(\bu) : \beps(\bu) \,\mathrm{d}x
+ \int_\Omega\frac{1}{4\epsilon} (1 - v(t))^2 + \epsilon|\nabla v(t)|^2 \,\mathrm{d}x 
\\&= \inf_{v\leq v(t)}\mathcal{A}_\epsilon(\bu(t), v). \nonumber
\end{align}
The first integral on the right-hand side of (\ref{intro:equ4}) represents an approximation of the elastic energy and the second  is an approximation of the surface energy, with fixed approximation parameter \( \epsilon\). 
The minimisation problem is identical to that  for the quasi-static model. In the literature, it is sometimes referred to as unilateral minimality.

We note that an issue occurs as a result of the minimisation problem when taking the limit in the approximation parameter \( \epsilon\). 
As stated in \cite{MR2673410},  ``the unilateral minimality of the \( v\) variable has no obvious counterpart in a sharp-interface model.''
This means that there is nothing ensuring maximal dissipation in the limiting model. 
Another deficiency with the model (\ref{intro:equ3}), (\ref{intro:equ4}) is the presence of the viscous  term \( \bu_t\). When taking the limit in the approximation parameter, the viscoelastic paradox can occur \cite{MR4117498}, that is,  the only possible solution of the corresponding sharp-interface problem is constant in time. However, in such dynamic fracture models, our interest lies in non-stationary solutions. Hence this paradox is a potential deficiency of the damping term in (\ref{intro:equ3}). In the problems that we study here, we include a damping term. It is out of the scope of the paper to remove this term but will be the focus of future work. 

Such a shortcoming is overcome in \cite{MR4064375} where the author considers a variation of (\ref{intro:equ3}), (\ref{intro:equ4}) that has no such damping term in (\ref{intro:equ3}). 
Rather than include a dissipative term in (\ref{intro:equ3}), a rate-dependent term is included in the minimisation problem (\ref{intro:equ4}). The author of \cite{MR4064375} considers the balance law
\begin{equation}\label{intro:equ5}
\bu_{tt} = \diver\big( (v^2 + \eta_\epsilon) \mathbb{E}\beps(\bu)\big) + \boldf 
\end{equation}
coupled with the rate-dependent minimisation problem
\begin{equation}\label{intro:equ7}
\mathcal{A}_\epsilon(\bu(t), v(t)) + (v_t(t) , v(t) )_{k,2} = \inf_{v\leq v(t) } \Big\{ \mathcal{A}_\epsilon(\bu(t), v) + (v_t(t) , v )_{k,2}\Big\} ,
\end{equation}
where \( \mathcal{A}_\epsilon\) is as in (\ref{intro:equ4}) and \((\cdot, \cdot)_{k,2}\) is the standard inner product in the Sobolev space \( H^k(\Omega) \). 
For every \( k \in \mathbb{N}_0\), it is shown that a weak solution to (\ref{intro:equ5}), (\ref{intro:equ7}) exists such that it satisfies an energy-dissipation inequality. However, if \( k > \frac{d}{2}\), it is possible to  prove that the energy-dissipation inequality is in fact an equality at every point in time. This is essentially a consequence of the Sobolev embedding theorem, from which we know that \( H^k(\Omega) \) is continuously embedded into \( C(\overline{\Omega}) \) for \( k > \frac{d}{2}\). 

The proof of the existence of solutions to the  dynamic problem in both \cite{MR2673410} and \cite{MR4064375} involves introducing a discretisation in the time variable. The time-discrete formulation involves an alternating procedure. The elastodynamic equation is solved for the displacement given the phase-field function from the previous time step. The minimisation problem for an unknown phase-field function is then solved, given the displacement function from the aforementioned step. A similar procedure is used here. 

A key novelty of the presentation in Section \ref{sec:p} is the formulation of the problem and the extension of the analysis of the linear case to a model of nonlinear viscoelasticity with a Kelvin--Voigt rheology, based on the implicit constitutive framework discussed previously. 
In particular, we note the derivation of bounds on the sequence of solutions to the time discrete problem. This demonstrates the highly nonlinear nature of the problem. 
However, the case of the strain-limiting material is the true highlight of this work. We make a significant step towards the use of strain-limiting constitutive relations in sharp-interface fracture models, which is the ultimate goal. 

The nonlinear nature of the problem adds significant analytical challenges to problem, compared with the equivalent problem in the linear setting. Not only does the problem require careful formulation, but we work hard to obtain various {\it a priori} bounds and to improve the regularity of approximate solution sequences. In the strain-limiting case, we have the added complication that the stress tensor \( \bbT \) is only bounded {\it a priori} in \( L^1(Q)^{d\times d}\). It is well known that this space has poor compactness properties. Indeed, at best we can obtain weak-* convergence in the space of Radon measures on \( \overline{Q}\). However, we want to stay out of the space of measure-valued solutions.  We use careful manipulations to improve the regularity of the approximations of the stress which allow the deduction of a pointwise convergence result, despite the poor integrability. This is later improved to strong convergence in \(L^1(0, T; L^1_{loc}( \Omega)^{d\times d}) \) and we prove that the limit is exactly the solution that we are looking for. 

There are two main results that we present in this work.  The first, presented in Section \ref{sec:p},  concerns a constitutive relation of the form
\begin{equation}\label{intro:equ8}
\beps(\bu_t + \alpha\bu) = F(\bbT) :=  |\bbT|^{p-2}\bbT,
\end{equation}
where \( p > 1\). When  \( p \in (1,2]\), we prove the existence of a weak solution \((\bu, v) \) to the  elastodynamic equation (\ref{intro:equ9}) coupled with a minimisation problem analogous to (\ref{intro:equ4})  and an energy-dissipation balance at every point in time. For \( p > 2\), we also obtain an existence result. However, the energy-dissipation balance only holds at almost every point in time. This results from a lack of regularity of the phase-field function. 

The second result, presented in Section \ref{sec:a},  concerns strain-limiting solids with the constitutive relation (\ref{intro:equ1}).  The elastodynamic equation (\ref{intro:equ9}) is coupled with the defining relation (\ref{intro:equ1}), a rate-dependent minimisation problem of the type (\ref{intro:equ7}) and an energy-dissipation balance. 
We supplement this with mixed Dirichlet--Neumann boundary conditions. We restrict ourselves to \( k > \frac{d}{2} + 1\) to ensure good regularity of the phase-field function. This allows the derivation  of higher regularity estimates on approximations of the displacement and stress tensor. 

If the Neumann part of the boundary is non-empty, due to a lack of global estimates on the stress tensor, the elastodynamic equation is shown to hold weakly up to a penalisation term on the Neumann part of the boundary. 
Indeed, the only global convergence result for the approximations of the stress tensor is in the set of measures on \( \overline{Q}\) due to the lack of weak compactness in the Lebesgue space \( L^1(Q)\). 
An important consequence of this is that we are only able to show that an energy-dissipation inequality holds. However, if we have a fully Dirichlet boundary condition, we have no such penalisation term and are able to obtain a full existence result. Indeed, a weak form of (\ref{intro:equ9}) holds, the constitutive relation is satisfied pointwise a.e., the minimisation problem holds at every point in time and we have an energy-dissipation equality. If additionally \( a\in (0, \frac{2}{d}) \), we also have \( \bbT \in L^{1+\delta}(0, T; L^{1 + \delta}_{loc}( \Omega)^{d\times d}) \) for every positive \( \delta\) such that \( a + \delta< \frac{2}{d}\). 

The structure of the remainder of the paper is as follows. In Section \ref{sec:naux}, we state the relevant notation and background results. In Section \ref{sec:p}, we study the nonlinear fracture problem for a constitutive relation of the form (\ref{intro:equ8}). The case \( p \in (1,2]\) is dealt with first. The existence of a solution is shown by means of a discretisation in time and semi-discretisation in space. In particular, we consider a  Galerkin approximation in space with respect to the displacement variable \( \bu \). Next, we adapt this work to the case \( p \in (2,\infty) \). We must use the monotonicity of the approximations of the phase-field function to overcome the lack of uniform bounds on the time derivatives of these functions. In Section \ref{sec:a}, we study the strain-limiting problem with constitutive relation (\ref{intro:equ1}). The proof uses a discretisation in time with an elliptic regularisation.   We obtain suitable weighted regularity estimates on the sequence of approximations \( (\bbT^n)_n \) of the stress tensor \((\bbT^n)_n \) to overcome the poor integrability properties of the sequence. These estimates are used to prove the pointwise convergence of this sequence to a limit \( \bbT \), which we show to be a part of the required solution triple \((\bu, \bbT, v) \). The climax of this work in the proof of a full existence result under Dirichlet boundary conditions and a partial existence result when the Neumann part of the boundary is non-empty.

\subsection{Notation and auxiliary results}\label{sec:naux}
Throughout this paper,  \(\Omega\subset\mathbb{R}^d\) denotes a bounded domain with Lipschitz boundary \( \partial\Omega\). The boundary has a Dirichlet part \( \Gamma_D\) and a Neumann part \( \Gamma_N \). These are relatively open but disjoint subsets of \( \partial\Omega\) such that \( \overline{\Gamma_D\cup \Gamma_N} = \partial\Omega\). We always  assume that \( \Gamma_D\neq \emptyset\).  
We let \( \mathbf{n}\)  denote the outward unit normal to a point on the boundary \(\partial\Omega\).  
We fix the spatial dimension \( d\geq 2\) and  a final time horizon \( T\in (0, \infty ) \). We denote the space-time domain by \( Q = (0, T) \times \Omega\).

We use \( W^{k,p}(\Omega) \) to denote the standard Sobolev space for \( k \in \mathbb{N}\), \( p \in [1,\infty]\) with the usual norm \( \|\cdot\|_{k,p}\). 
When considering \( W^{k,p}(\Omega_0)\) for a subdomain \( \Omega_0\subset \Omega \), we write  \( \|\cdot\|_{W^{k,p}(\Omega_0)}\). 
We write \( C^\infty_D(\overline{\Omega}) \) for the set of functions \( h\in C^\infty(\overline{\Omega}) \) such that \( h = 0 \) on \( \Gamma_D\) and let \( W^{k,p}_D(\Omega) \) denote  the closure of \( C^\infty_D(\overline{\Omega}) \) with respect to \( \|\cdot \|_{k,p}\) for \( p \in [1,\infty)\). We define \( W^{k,\infty}_D(\Omega) := W^{k,\infty}(\Omega) \cap W^{k,1}_D(\Omega)\).

We define \( C^1( \partial\Omega)^d\) to be the set of continuously differentiable functions on \( \partial\Omega\) such that there exists an extension to the  space \( C^1(\overline{\Omega})^d\). The norm on \( C^1(\partial\Omega)^d\) is defined by
\[
\|\bw \|_{C^1(\partial\Omega)} = \inf\big\{ \|\tilde{\bw}\|_{1, \infty} \,:\, \tilde{\bw} \in C^1(\overline{\Omega})^d ,\, \tilde{\bw} = \bw \text{ on }\partial\Omega\big\}. 
\]
We define \( C^1_0(\Gamma_N)^d \) to be the subset of those functions \( \bw \in C^1(\partial\Omega)^d\) such that \( \bw \) vanishes on \( \overline{\Gamma_D}\). By Theorem 1 of \cite{MR2373373}, we know that there exists a linear map \( E: C^1(\partial\Omega) \rightarrow C^1(\overline{\Omega}) \) such that \( T\bw  = \bw \) on \( \partial\Omega\) for every \( \bw \in C^1(\partial\Omega)^d\). Furthermore, there exists a constant \( C\), depending only on \(\Omega\), such that
\begin{equation}\label{intro:equ15}
\|T \bw \|_{1,\infty} \leq C \|\bw \|_{C^1( \partial\Omega) }. 
\end{equation}
The restriction of \( T \) to the space \(C^1_D(\overline{\Omega})^d\) maps into \( C^1_0(\Gamma_N)^d\) by definition of the latter function space. Moreover, for every \( \bw \in C^1_0(\Gamma_N)^d\), there exists an extension \( \tilde{\bw}\in C^1_D(\overline{\Omega})^d\) such that (\ref{intro:equ15}) holds. This property is vital when dealing with the penalisation term in Section \ref{sec:a} on the Neumann part of the boundary. 


 If \( \mathcal{F}\) is a function space of real-valued functions with norm \( \|\cdot \|_{\mathcal{F}}\), we let \(\mathcal{F}^d\) denote the set of \( \mathbb{R}^d\)-valued functions such that each co-ordinate is an element of \( \mathcal{F}\). The norm is given by  \(   \||\cdot|\|_{\mathcal{F}}\), where \( |\cdot |\) is the Euclidean norm, but is denoted by \( \|\cdot \|_{\mathcal{F}}\) with an abuse of notation. The meaning is to be understood from the context.  We define \( \mathcal{F}^{d\times d}\) analogously, but with \( |\cdot |\)  meaning the Frobenius norm. 
 
 For a Banach space \( X\), we use \( L^q(0, T; X) \) to denote the standard Bochner space with norm denoted \( \|\cdot \|_{L^q(0, T; X)}\).  Furthermore, \( L^\infty_{w^*}(0, T; X^*) \) denotes the set of functions \( h:[0, T]\rightarrow X^*\) such that 
\( t\mapsto \langle h(t), x\rangle_{(X^*,X)}\) is measurable on \( [0, T]\) for every element \( x\in X\) and there exists a constant \( c\), independent of \( x\), such that
\[
|\langle h(t), x\rangle_{(X^*, X)}|\leq c\|x\|_X\quad\text{ for a.e. }t\in [0, T].
\]
By Theorem 10.1.16 of \cite{RN118}, for example, \( L^\infty_{w^*}(0, T; X^*) \) is the dual space of \( L^1(0, T; X) \). In particular, bounded subsets of \( L^\infty_{w^*}(0, T; X^*) \) are weakly-* compact by the Banach--Alaoglu theorem.

Let \( H^k_D(\Omega) \) be the set of functions \( h \in W^{k,2}(\Omega)\)  such that \( h = 0\) on \( \Gamma_D\). 
We use \( H^k_{D+1}(\Omega)\) to denote the set of  \( h \in  W^{k,2}(\Omega) \)  such that \( h = 1\) on \( \Gamma_D\). 
The duals of \( H^1_D(\Omega) \) and \( W^{1,p}_D(\Omega)^d\) are denoted by \( H^{-1}_D(\Omega)\) and \( W^{-1,p^\prime}_D(\Omega)^d\), respectively, where \( p^\prime  = \frac{p}{p-1}\) is the H\"{o}lder conjugate of \(  p\). 

We use lower case letters to denote scalar-valued functions, bold  lower case letters for vector-valued functions and bold upper case letters for tensor-valued functions. The differential operator \( \nabla \) denotes the gradient operator so, if \( \bu \) is an \(\mathbb{R}^d\)-valued function and \( \bbT \) is an \( \mathbb{R}^{d\times d}\)-valued function, we have
\[
\nabla \bu = (\partial_i \bu_j)_{i,j=1}^d, \quad \nabla \bbT = (\partial_k T_{ij})_{i,j,k=1}^d.
\]
The divergence operator \(\diver\) is defined by
\begin{align*}
\diver(\bu) = \frac{\partial u_i}{\partial x_i}, \quad \diver(\bbT) = \Big( \frac{\partial T_{ij}}{\partial x_j}\Big)_{j=1}^d,
\end{align*}
assuming Einstein's summation convention. 
Throughout this paper, we  fix positive problem parameters \( \alpha\), \( \epsilon\) and \(\eta \), where \( \alpha\) is the inverse of the viscosity coefficient, \( \epsilon\) is the approximation level for the phase-field function and \( \eta\) is a stability coefficient. In practice,  \(\eta\ll\epsilon\ll 1\) but this is not required for the results here.

 We require certain properties of convex conjugate functions. In this paper, we define the function \( \varphi:\mathbb{R}^{d\times d}\rightarrow [0, \infty) \) by 
 \begin{equation}\label{naux:equ1}
  \varphi(\bbT) = \frac{|\bbT|^p}{p},\quad\text{ and } \quad\varphi(\bbT)  = \int_0^{|\bbT|} \frac{t}{(1 + t^a)^{\frac{1}{a}}}\,\mathrm{d}t,
 \end{equation}
 in Section \ref{sec:p} and Section \ref{sec:a}, respectively. 
 Given one of these choices, we further define
 \begin{align*}
 F(\bbT) = \frac{\partial \varphi}{\partial\bbT}(\bbT) = \Big( \frac{\partial\varphi}{\partial T_{ij}}(\bbT)\Big)_{i,j=1}^d, \quad \varphi^*(\bbT) = \sup_{\bbS\in \mathbb{R}^{d\times d}}\Big\{ \bbT :\bbS - \varphi(\bbS)\Big\},
 \end{align*}
 the derivative and convex conjugate, respectively. On the image of \( F\) over \( \mathbb{R}^{d\times d}\), we can show that  
 \(\frac{\partial\varphi^*}{\partial\bbT} = F^{-1}\). For every \( \bbT \in \mathbb{R}^{d\times d}\) such that the statements make sense, we have
 \begin{equation}\label{intro:equ13}
  \bbT:F(\bbT ) = \varphi(\bbT) + \varphi^*(F(\bbT)) ,\quad \bbT:F^{-1}(\bbT) = \varphi(F^{-1}(\bbT)) + \varphi^*(\bbT) .
 \end{equation}
Given the choices  in (\ref{naux:equ1}),  we notice that \( F(\bbT) = |\bbT|^{p-2}\bbT\) and \( F(\bbT) = \frac{\bbT}{(1 + |\bbT|^a)^{\frac{1}{a}}}\), respectively.  

We  require the following result in the analysis of the strain-limiting problem.  The proof can be found in \cite{RN8}. 
 \begin{lemma}
 For every \( y\geq 0 \) and \( a > 0 \), we have
 \begin{align*}
 \min\{ 1 , 2^{-1 + \frac{1}{a}}\}(1 + y) \leq (1 + y^a)^{\frac{1}{a}}\leq \max\{ 1 , 2^{-1 + \frac{1}{a}}\}(1 + y) .
 \end{align*}
 \end{lemma}
 A key tool in the deduction of various {\it a priori} estimates and uniform bounds on the sequences of approximate solutions is the following variation of the Korn--Poincar\'{e} inequality for a partly homogeneous Dirichlet boundary condition. The result follows from Theorem 4 in \cite{MR971465}, combined with an application of the Poincar\'{e} inequality.
 
 \begin{theorem}[Korn--Poincar\'{e} inequality]\label{korn}
 Let \( \Omega\), \( \Gamma_D\) and \( \Gamma_N\) be as  above with \( \Gamma_D \neq \emptyset\). For every \( p \in (1,\infty) \), there exists a constant \( C\) depending only on \( \Omega\), \( \Gamma_D\) and \( p \) such that, for every \( \bu\in W^{1,p}_D(\Omega)^d\),
 \begin{align*}
 \|\bu\|_{1,p}\leq C\|\beps(\bu) \|_p.
 \end{align*}
 \end{theorem}

\section{Nonlinear fracture problems with growth}\label{sec:p}
In this section, we study nonlinear dynamic fracture problems with growth in the constitutive relation, by which we mean that the strain is not {\it a priori} bounded. 
This is the first step in the analysis of strain-limiting fracture problems because we want to understand nonlinear dynamic fracture problems. However, we start with this growth setting so that we gain some understanding of the problem formulation without the additional difficulty of the lack of integrability of the stress tensor that we have in the case that \( F\) is bounded. 

We consider the following problem: find a triple \((\bu, \bbT, v):Q \rightarrow\mathbb{R}^d \times \mathbb{R}^{d\times d}\times \mathbb{R}\) such that 
\begin{equation}\label{p:equ1}
\begin{aligned}
\bu_{tt} &= \diver\big(b(v)\bbT\big) + \boldf &\quad&\text{ in }Q,
\\
\beps(\bu_t + \alpha\bu) &= F(\bbT) := |\bbT|^{p-2}\bbT&\quad&\text{ in }Q,
\\
\bu\big|_{\Gamma_D} &= \mathbf{0}, \, v\big|_{\Gamma_D} = 1&\quad& \text{ on }[0, T]\times \Gamma_D,
\\
b(v) \bbT \mathbf{n} &= \bg&\quad& \text{ on }[0, T]\times \Gamma_N,
\\
\bu(0, \cdot)&= \bu_0, \bu_t(0, \cdot) = \bu_1,\, v(0, \cdot) = v_0&\quad&\text{ in }\Omega,
\end{aligned}
\end{equation}
subject to the minimisation problem
\begin{equation}\label{p:equ2}
\mathcal{E}(\bu(t), v(t)) + \mathcal{H}(v(t)) = \inf\Big\{ \mathcal{E}(\bu(t), v) + \mathcal{H}(v)\,:\, v\in H^1_{D+1}(\Omega),\, v\leq v(t) \Big\},
\end{equation}
and the energy-dissipation equality
\begin{equation}\label{p:equ3}
\begin{aligned}
&\mathcal{F}(t;\bu(t), \bu_t(t), v(t)) + \int_0^t \int_\Omega b(v) \big[ F^{-1}(\beps(\bu_t + \alpha\bu))  - F^{-1}(\beps(\alpha\bu))\big]:\beps(\bu_t)  + \boldf_t\cdot \bu \,\mathrm{d}x\,\mathrm{d}s
\\
&\quad 
+ \int_0^t \int_{\Gamma_N} \bg_t \cdot \bu\,\mathrm{d}S\,\mathrm{d}s
\\
&= \mathcal{F}(0; \bu_0, \bu_1, v_0) ,
\end{aligned}
\end{equation}
for every \( t\in [0, T]\), alongside the crack non-healing property \( v_t \leq 0 \). The functionals \( \mathcal{F}\), \( \mathcal{E}\), \( \mathcal{H}\) denote the total, elastic and surface energy, respectively, and are defined by
\begin{align*}
\mathcal{F}(t;\bu, \bw, v) &= \mathcal{K}(\bw) + \mathcal{E}(\bu, v)+ \mathcal{H}(v) - \langle l(t),\bu\rangle,
\\
\mathcal{E}(\bu, v) &= \int_\Omega\frac{b(v)}{\alpha}\varphi^*(\beps(\alpha\bu)) \,\mathrm{d}x,
\\
\mathcal{H}(v)  &= \frac{1}{4\epsilon}\| 1 - v\|_2^2 + \epsilon\|\nabla v\|_2^2 .
\end{align*}
We define the functionals \( \mathcal{K}\) and \( l \) by
\begin{align*}
\mathcal{K}(\bw) &= \frac{\|\bw\|_2^2}{2},
\\
\langle l(t), \bu\rangle &= \int_\Omega\boldf(t)\cdot \bu \,\mathrm{d}x +\int_{\Gamma_N} \bg(t) \cdot\bu\,\mathrm{d}S,
\end{align*}
representing the kinetic energy and external force, respectively. 
 To ensure that the minimisation problem holds at the initial time \( t = 0\), we demand the compatibility condition  
\begin{equation}\label{p:equ5}
\mathcal{E}(\bu_0, v_0) + \mathcal{H}(v_0) = \inf\Big\{ \mathcal{E}(\bu_0, v)+ \mathcal{H}(v)\,:\, v\in H^1_{D+1}(\Omega),\,v\leq v_0\Big\}.
\end{equation}
The function \( b :\mathbb{R}\rightarrow\mathbb{R}\) is defined   by
\begin{equation}\label{p:equ4}
b(v) = v^2 + \eta.
\end{equation}
The stability parameter \( \eta\) is required so that we do not experience degeneracy in the elastodynamic equation (\ref{p:equ1})$_1$ at points where the phase-field function \( v\) vanishes. 

We remark that 
the functional \( \mathcal{A}(v) := \mathcal{E}(\bu, v) + \mathcal{H}(v) \) is well-defined on \( H^1_{D+1}(\Omega) \), provided that we have \( \bu \in W^{1,p^\prime}(\Omega)^d\), which we later show to be the case for solutions of (\ref{p:equ1})--(\ref{p:equ3}). However, the functional \( \mathcal{A}\) may  take the value  positive infinity when testing against functions that are not also elements of \( L^\infty(\Omega) \). 

In this section we consider mixed Dirichlet--Neumann boundary conditions. We assume a homogeneous boundary condition on the Dirichlet part for simplicity of the presentation. However, it is not too difficult to adapt the arguments to allow for inhomogeneous boundary conditions. We refer to  \cite{preprint2}, for example,  for details on how to do this.  

Now we state the main results  for this section and  have a brief discussion regarding the methods used in the proofs. There are  two main results in this section concerning the existence of  weak energy solutions to   (\ref{p:equ1})--(\ref{p:equ3}), that is, a weak solution satisfying an energy-dissipation balance.  We differentiate between the cases \( p \in (1,2]\) and \( p \in (2,\infty) \). This is due to a technical argument that allows good time regularity of the phase-field function in the former case, but not necessarily the latter.  We remark that uniqueness remains an open problem for dynamic fracture models and  is out of the scope of this paper. Our focus is on the existence of solutions. The two main results are as follows. 

\begin{theorem}\label{p:thm1}
Suppose that \( p \in (1,2]\),  \( l \in C^1([0, T]; W^{-1,p}_D(\Omega)^d) \), \( \bu_0\), \( \bu_1 \in W^{1,p^\prime}(\Omega)^d\), and \( v_0 \in H^1_{D+1}(\Omega) \) such that \( v_0 \in [0, 1]\) a.e. in \(\Omega\) and  the compatibility condition (\ref{p:equ5}) holds. 
There exists a weak energy solution  \((\bu, \bbT, v) \) of problem  (\ref{p:equ1})--(\ref{p:equ3}) in the following sense with  
\begin{itemize}
\item \(\bu \in W^{2,2}(0, T; L^2(\Omega)^d) \cap W^{1,\infty}(0, T; W^{1,p^\prime}_D(\Omega)^d) \),
\item \( \bbT\in L^\infty(0, T; L^p(\Omega)^{d\times d}) \),
\item \( v\in W^{1,\infty}(0, T; H^1(\Omega) ) \) with \( v(t) \in H^1_{D+1}(\Omega) \) for every \( t\in [0, T]\). 
\end{itemize}
The weak elastodynamic equation holds,
\begin{equation}\label{p:equ6}
\int_\Omega\bu_{tt}(t) \cdot \bw + b(v(t))\bbT(t) :\beps(\bw) \,\mathrm{d}x = \langle l(t), \bw\rangle,
\end{equation}
for a.e. \( t\in (0, T) \) and every \( \bw\in W^{1,p^\prime}_D(\Omega)\), the constitutive relation \(\beps(\bu_t + \alpha\bu) = F(\bbT) \) is satisfied pointwise a.e. in \( Q\), the minimisation problem (\ref{p:equ2}) holds for every \( t\in [0, T]\) and the energy-dissipation equality (\ref{p:equ3}) is satisfied for every \( t\in [0, T]\). We have the non-healing property \( v_t \leq 0 \) pointwise a.e. in \( Q\). The initial conditions are satisfied  in the sense that
\begin{align*}
\lim_{t\rightarrow 0+}\big[ \|\bu(t) - \bu_0 \|_{1,p^\prime} + \|\bu_t(t) - \bu_1 \|_2 + \|v(t)- v_0\|_{1,2}\big] = 0.
\end{align*}
\end{theorem}

\begin{theorem}\label{p:thm2}
Suppose that \( p \in (2,\infty) \), \( l \in C^1([0, T]; W^{-1,p}_D(\Omega)^d)\), \( \bu_0 \), \( \bu_1\in W^{1,p^\prime}_D(\Omega)^d\cap L^2(\Omega)^d\), and \( v_0 \in H^1_{D+1}(\Omega ) \) such that  \( v_0 \in [0, 1]\) a.e. in \( \Omega\) and  the compatibility condition (\ref{p:equ5}) holds. There exists a weak energy solution  \((\bu, \bbT, v) \) of (\ref{p:equ1})--(\ref{p:equ3}) with regularity
\begin{itemize}
\item \(\bu \in W^{2,2}(0, T; L^2(\Omega)^d) \cap W^{1,\infty}(0, T; W^{1,p^\prime}_D(\Omega)^d) \),
\item \( \bbT \in L^\infty(0, T; L^p(\Omega)^{d\times d}) \),
\item \( v:[0, T]\rightarrow H^1_{D+1}(\Omega)\) with \( v\in L^\infty(0, T; H^1(\Omega))\) and \( v\) is continuous a.e. as a map from \( [0, T]\) to \(L^2(\Omega) \), 
\end{itemize}
solving the problem (\ref{p:equ1})--(\ref{p:equ3}) in the following sense. 
The weak elastodynamic equation (\ref{p:equ6}) holds for a.e. \( t\in (0, T) \) and every test function \( \bw \in W^{1,p^\prime}_D(\Omega)^d\cap L^2(\Omega)^d\), we have the constitutive relation \( \beps(\bu_t + \alpha\bu) = F(\bbT) \)  pointwise a.e. in \(Q\), the minimisation problem (\ref{p:equ2}) is satisfied for every \( t\in [0, T]\), and the energy-dissipation equality (\ref{p:equ3}) holds for a.e. \( t\in [0, T]\). Furthermore,  \( v(t) \leq v(s) \) a.e. in \( \Omega\), for every \( 0 \leq s\leq t\leq T\). The initial conditions hold in the sense that \( v(0) = v_0\) and 
\begin{align*}
\lim_{t\rightarrow 0+ }\big[  \|\bu(t) - \bu_0 \|_{1,p^\prime}  + \|\bu_t(t) - \bu_1 \|_2 \big] = 0.
\end{align*}
\end{theorem}

Both proofs follow the same steps initially. A two-level approximation is introduced with parameters \((N,M) \),  where \( N \) refers to a discretisation in space for the displacement variable and \(M \) is a time discretisation parameter.  
For notational simplicity, we  denote \(\beta = (N,M)\).
We  let \( M \rightarrow\infty \), followed by the limit as \( N \rightarrow\infty\). It is in the latter limit where we need to differentiate between the cases \( p \in (1,2]\) and \( p \in (2,\infty) \). 
Hence, we focus first on proving Theorem \ref{p:thm1}. In Section \ref{sec:plarge}, we  present the required changes for the proof of Theorem \ref{p:thm2}. 

 The two-level approximation suggests how the analysis could be adapted when considering the problem in a numerical setting and using, for example, finite-element approximations. However, we choose Galerkin approximations with very high regularity to simplify the analysis when taking the limit as \( M \rightarrow\infty\). 
 In particular, we are able to avoid technical arguments such as those in Section 3.6 of \cite{MR2673410}.
 
We structure the proof of Theorem \ref{p:thm1} in the following way. First, we introduce the two-level approximation and in Theorem \ref{p:thm3} prove the existence of unique solutions to the time-discrete Galerkin approximation. In Propositions \ref{p:prop2} and \ref{p:prop3}, we show that the solution sequence satisfies discrete counterparts of the minimisation problem (\ref{p:equ2}) and the energy-dissipation equality (\ref{p:equ3}). Then we derive various \(M\)-independent   bounds,  which allow us to take the limit as \( M \rightarrow\infty \). We obtain a weak-energy solution of a time-continuous Galerkin approximation of (\ref{p:equ1})--(\ref{p:equ3}). We derive various \( N \)-independent bounds on this solution. Due to the growth of the function \( F\), the solution of the time-continuous Galerkin approximation experiences good integrability and regularity. Hence, we can take the limit as \( N\rightarrow\infty \) and use standard compactness arguments to obtain a weak energy solution of (\ref{p:equ1})--(\ref{p:equ3}) in the sense of Theorem \ref{p:thm1}. We remark that, out of the time-discrete setting, we are unable to make any claims regarding the uniqueness of solutions. 

For the discretisation in space, let \((\phib_i)_{i=1}^\infty\) be an orthogonal basis of \( W^{k,2}_D(\Omega)^d\) such that it is orthonormal in \( L^2(\Omega)^d\) and \( k > \frac{d}{2}+1\). We make this choice so that \(W^{k,2}_D(\Omega)^d\) embeds continuously into \( C^1(\overline{\Omega})^d\),  by the Sobolev embedding theorem. For details of how to obtain such a basis, we refer to \cite{RN119} and the references therein. We define \( V_N = \mathrm{span}\{ \phib_1,\dots ,\phib_N\}\) and let \( P_N:L^2(\Omega)^d\rightarrow V_N \) denote the orthogonal projection operator. By construction, there exists a constant \( C> 0 \), independent of \( N \), such that
\[
\|P_N\bu\|_{k,2}\leq C \|\bu\|_{k,2} \quad\text{ for every }\bu\in W^{k,2}_D(\Omega)^d.
\]
In this finite-dimensional setting,  for every \( q\in [1,\infty]\),  \( \|\cdot\|_{q}\) and \( \|\cdot \|_{1,q}\) are norms on \( V_N \) that are equivalent to the usual norm \(\|\cdot \|_{V_N}:= \|\cdot\|_{k,2}\), with the constants of equivalence depending on \( N \). 

Now we  define the time discrete, finite-dimensional    problem. Let \( \beta = (N , M) \in \mathbb{N}^2\) with  \( V_N \) defined as above. We define the time step \( h = \frac{T}{M}\) and, for every \( 0 \leq m \leq M \), we set \( t^\beta_m = mh\). For problem data \(\bu_0\), \( \bu_1\), \( v_0 \), \( \boldf\) and \( \bg\) satisfying the assumptions of Theorem \ref{p:thm1}, we define an initialisation for the \( \beta\)-approximation of (\ref{p:equ1})--(\ref{p:equ3}) by
\begin{equation}\label{p:equ9}
\bu^\beta_0 = \bu_{0,N},\quad  \bu^\beta_{-1} =  \bu_{0,N} - h\bu_{1,N}  ,\quad  v^\beta_0 = v^{N}_0,
\end{equation}
where \( v^{N}_0\in H^1_{D+1}(\Omega) \) is the solution of the minimisation problem
\begin{equation}\label{p:equ10}
\mathcal{E}(\bu_{0,N}, v^{N}_0) + \mathcal{H}(v^{N}_0) = \inf\Big\{ \mathcal{E}(\bu_{0,N}, v) + \mathcal{H}(v): v\leq v_0 ,v\in H^1_{D+1}(\Omega)\Big\}.
\end{equation}
We impose (\ref{p:equ10}) so that, at the level of the approximate problem, the minimisation problem is satisfied at \( t = 0 \). The function \(v^N_0\) exists and is unique, applying the direct method in the calculus of variations and strict convexity. 
The functions \( \bu_{0,N}\), \( \bu_{1,N}\in W^{k,2}_D(\Omega)^d \) are approximations of the initial data \( \bu_0 \), \( \bu_1 \), respectively, and are chosen  depending on the value of \(p \).

In the case that \( p \in (1,2]\), we choose the approximations such that \( \bu_{0,N}\), \( \bu_{1,N}\in V_N \) and, for every \( N \), 
\begin{align*}
  0 = \lim_{\tilde{N}\rightarrow\infty} \|\bu_0 - \bu_{0,\tilde{N}}\|_{1,p^\prime} + \|\bu_1 - \bu_{1,\tilde{N}}\|_{1,p^\prime} \leq  \|\bu_0 - \bu_{0,N}\|_{1,p^\prime} + \|\bu_1 - \bu_{1,N}\|_{1,p^\prime}\leq C ,
\end{align*}
where \( C = C(\bu_0 , \bu_1) \) depends only on the initial data. 
Existence of such a construction is deduced as follows.  The space \(\cup_{N \geq 1}V_N \) is dense in \( W^{k,2}_D(\Omega)^d\) so, by the Sobolev embedding theorem, it must be dense in \( W^{1,p^\prime}_D(\Omega)^d \) for any \( p^\prime\in (1,\infty) \). 
Furthermore, \( V_N \) is a closed subset of \(W^{1,p^\prime}_D(\Omega)^d\). Hence for every \( \tilde{\bu}\in W^{1,p^\prime}_D(\Omega)^d\),  there exists a \( \tilde{\bu}_N \in V_N \) such that
\[
\|\tilde{\bu} - \tilde{\bu}_N \|_{1,p^\prime} = \inf_{\bv_N \in V_N} \|\tilde{\bu} - \bv_N\|_{1,p^\prime} \leq \|\tilde{\bu}\|_{1,p^\prime}.
\]
By density, we must have
\[
\lim_{N\rightarrow\infty}\inf_{\bv_N \in V_N} \|\tilde{\bu} - \bv_N\|_{1,p^\prime} = 0. 
\]
Thus we conclude the existence of such approximation sequences for \( \bu_0 \) and \( \bu_1\) as above. Furthermore, we note that because \( p^\prime\geq 2\), it follows that the sequences \((\bu_{0, N})_N \) and \((\bu_{1,N})_N \) are bounded in \( L^2(\Omega)^d\) and converge strongly in this space to \( \bu_0\), \( \bu_1\), respectively, whenever \( \bu_0 \), \( \bu_1\in W^{1,p^\prime}_D(\Omega)^d\). 

In the case that \( p > 2\), we must choose the sequences in such a way so that they approximate the data in both \( L^2(\Omega)^d\) and \( W^{1,p^\prime}_D(\Omega)^d\).  
This is due to the fact that \( W^{1,p^\prime}_D(\Omega) \) is not necessarily embedded in \( L^2(\Omega)\), unlike the case that \( p \leq 2\). 
We replace \(\|\cdot \|_{1,p^\prime}\) in the above by \( \|\tilde{\bu}\|_* := \|\tilde{\bu}\|_2 + \|\tilde{\bu}\|_{1,p^\prime}\) and repeat the same arguments. Hence, given \( \bu_0\), \( \bu_1 \in L^2(\Omega)^d\cap W^{1,p^\prime}_D(\Omega)^d\), we can find \( \bu_{0, N}\), \( \bu_{1,N}\in V_N \) such that
\begin{align*}
0 = \lim_{N\rightarrow\infty} \|\bu_i - \bu_{i,N}\|_2 + \|\bu_i - \bu_{i,N}\|_{1,p^\prime} \ \leq \|\bu_i\|_2 + \|\bu_i\|_{1,p^\prime}\quad\text{ for }i \in \{1,2\}. 
\end{align*}


 
 The \(\beta\)-approximation problem is defined  recursively as follows. For every \( 1\leq m \leq M \), we look for a function \( \bu^\beta_m \in V_N \) such that
\begin{equation}\label{p:equ47}
\int_\Omega\delta^2\bu^\beta_m\cdot\bw + b(v^\beta_{m-1}) F^{-1}(\beps(\delta\bu^\beta_m + \alpha\bu^\beta_m)) :\beps(\bw) \,\mathrm{d}x = \langle l^\beta_m, \bw\rangle,
\end{equation}
for every \( \bw \in V_N \), where \( l^\beta_m := l(t^\beta_m)\). Next we look for a function \( v^\beta_m \in H^1_{D+1}(\Omega) \) that solves the minimisation problem
\begin{equation}\label{p:equ7}
\mathcal{E}(\bu^\beta_m, v^\beta_m ) + \mathcal{H}(v^\beta_m) = \inf\Big\{ \mathcal{E}(\bu^\beta_m, v) + \mathcal{H}(v)\,:\, v\in H^1_{D+1}(\Omega),\, v\leq v^\beta_{m-1}\Big\}. 
\end{equation}
We use \( \delta\) and \( \delta^2\) to denote the first and second order difference quotients, respectively, by which we mean
\begin{equation}\label{p:equ8}
\delta\bu^\beta_m := \frac{\bu^\beta_m - \bu^\beta_{m-1}}{h},\quad \delta^2\bu^\beta_m := \frac{\delta\bu^\beta_m - \delta\bu^\beta_{m-1}}{h} = \frac{\bu^\beta_m - 2\bu^\beta_{m-1} + \bu^\beta_{m-2}}{h^2}.
\end{equation}
We  define \( \delta v^\beta_m\) and \( \delta l^\beta_m \) analogously.

We remark that the function \( F\) is a bijection from \( \mathbb{R}^{d\times d}\) to itself so we can formulate the problem in terms of only the displacement and the phase-field function. We reintroduce the stress function later when taking the limit as \( N \rightarrow\infty \), to simplify the notation.

\begin{theorem}\label{p:thm3}
Suppose that \( l \in C^1([0, T]; W^{-1,p}_D(\Omega)^d) \) with \( \bu_0\), \( \bu_1 \in W^{1,p^\prime}_D(\Omega)^d\cap L^2(\Omega)^d\) and \( v_0 \in H^1_{D+1}(\Omega) \) such that  \( v_0 \in [0, 1]\) a.e. in \(\Omega\). 
For every \( \beta\in \mathbb{N}^2\), defining the initial data by (\ref{p:equ9}), there exists a unique sequence of solutions \((\bu^\beta_m, v^\beta_m)_{m=1}^M\) solving the time discrete problem (\ref{p:equ7}), (\ref{p:equ8}).  
\end{theorem}

Existence follows in a standard way from the Browder--Minty theorem for (\ref{p:equ7}) and the direct method in the calculus of variations for (\ref{p:equ8}) so we do not include the details. Uniqueness is a result of strict monotonicity and strict convexity, respectively. 

Before proceeding any further, we note the following convergence result  for the sequence \((v^N_0)_N\). This is vital knowledge when we take the limit as \( N \rightarrow\infty \). It enables us to prove the satisfaction of the energy-dissipation equality when taking the limit in \( N \). 

\begin{proposition}\label{p:prop1}
Let \( v_0 \in H^1_{D+1}(\Omega)\) with \( v_0 \in [0,1]\) a.e. in \( \Omega\) and \( \bu_0 \in W^{1,p^\prime}_D(\Omega)^d\cap L^2(\Omega)^d\) such that the compatibility condition (\ref{p:equ5}) holds. Let \( v^{N}_0\in H^1_{D+1}(\Omega)\) be the solution of the minimisation problem (\ref{p:equ10}).
Then \( v^N_0\in [0, 1]\) a.e. in \( \Omega\) and we have  \( v^N_0 \rightarrow v_0 \) strongly in \( H^1(\Omega) \) as \( N \rightarrow\infty \). 
\end{proposition}

\begin{proof}
First, we notice that, comparing \( v\) with \( \max\{v,0\}\), any minimiser of (\ref{p:equ10}) must be non-negative. By construction, we have \( v_0^N \leq v_0 \leq 1\). Hence \((v^N_0)_N \) is uniformly bounded in \( L^\infty(\Omega) \). Next, by construction, we have
\begin{align*}
C(\epsilon) \|v^N_0\|_{1,2}^2 &\leq \mathcal{E}(\bu_{0,N}, v_0^N) + \mathcal{H}(v^N_0)
\\
& \leq \mathcal{E}(\bu_{0,N}, v_0) + \mathcal{H}(v_0)
\\
& \leq \frac{b(1)}{\alpha} \int_\Omega \varphi^*(\beps(\bu_{0, N})) \,\mathrm{d}x + C(\epsilon) \|v_0 \|_{1,2}^2
\\
&\leq C\big[ \|\beps( \bu_{0,N}) \|_{p^\prime}^{p^\prime} + \|v_0 \|_{1,2}^2\big]
\\
& \leq C\big[ \|\bu_0\|_2^{p^\prime} +  \|\bu_0 \|_{1,p^\prime}^{p^\prime} + \|v_0 \|_{1,2}^2 \big].
\end{align*}
It follows that \( (v^N_0)_N \) converges weakly in \( H^1(\Omega) \) and strongly in \( L^2(\Omega)\) to a limit \( \tilde{v}_0\in H^1_D(\Omega)\) with \( \tilde{v}_0 \leq v_0\), up to a subsequence that we do not relabel. It follows from Fatou's lemma and weak lower semi-continuity that
\begin{equation}\label{p:equ36}
\begin{aligned}
\mathcal{E}(\bu_0, \tilde{v}_0) + \mathcal{H}( \tilde{v}_0) & \leq \lim_{N\rightarrow\infty} \Big[ 
\mathcal{E}(\bu_{0,N}, v_0^N) + \mathcal{H}(v^N_0)\Big]
\\
& \leq \lim_{N\rightarrow\infty } \Big[ \mathcal{E}(\bu_{0,N}, v_0) + \mathcal{H}(v_0)\Big]
\\
&= \mathcal{E}(\bu_0, v_0) + \mathcal{H}(v_0) 
\\
& \leq \mathcal{E}(\bu_0, \tilde{v}_0) + \mathcal{H}(\tilde{v_0}) .
\end{aligned}
\end{equation}
The final inequality follows from the compatibility condition (\ref{p:equ5}). Hence \( \tilde{v}_0 \) is also a minimiser of the problem (\ref{p:equ5}). By strict convexity, we deduce that \( v_0 = \tilde{v}_0\). 

Next, we notice that \( (\mathcal{E}(\bu_{0,N}, v^N_0) )_N\) converges to \( \mathcal{E}(\bu_0, v_0) \) as a result of the dominated convergence theorem and the uniform boundedness of \((b(v^N_0))_N\). Inserting this into (\ref{p:equ36}), it follows that
\begin{align*}
\lim_{N\rightarrow\infty}\mathcal{H}(v^N_0) = \mathcal{H}(v_0).
\end{align*}
The functional \( \mathcal{H}\) is equivalent to the usual norm on \( H^1(\Omega) \). Thus \((v^N_0)_N \) converges weakly and in norm to \( v_0 \) and so we have strong convergence. 
\end{proof}

Now we prove that the sequence of solutions satisfies  properties related to the minimisation problem, and an inequality analogous to the energy-dissipation equality  for the time continuous solution. The first result is an immediate consequence of the minimisation problem (\ref{p:equ8}). However, the formulation presented in Proposition \ref{p:prop2} is more amenable when it comes to taking the limit as \( M \rightarrow\infty \). Corollary \ref{p:cor1} follows from Proposition \ref{p:prop2}, which is vital in proving uniform bounds on the sequence \((\delta v^\beta_m)_{m=1}^M \). In Proposition \ref{p:prop3}, we prove an energy-dissipation inequality in the spirit of (\ref{p:equ3}). We are not able to obtain an equality in this setting due to the presence of numerical dissipation terms resulting from the time discretisation. We later use this inequality to obtain one direction of an energy-dissipation equality for the semi-discrete couple that is constructed by taking the limit in \(M\). 

\begin{proposition}\label{p:prop2}
Let the assumptions of Theorem \ref{p:thm3} hold. Then \( v^\beta_m \in [0, 1]\) and \( \delta v^\beta_m \leq 0 \) a.e. in \( \Omega\), for every \( 1\leq m \leq M\). Furthermore, for every \( \tilde{\chi}\in H^1_{D+1}(\Omega) \) such that \( \tilde{\chi}\leq v^\beta_{m-1}\), we have
\begin{equation}\label{p:equ11}
\begin{aligned}
0 &\leq \big[ \partial_v\mathcal{E}(\bu^\beta_m , v^\beta_m) + \mathcal{H}^\prime(v^\beta_m) \big](\tilde{\chi} - v^\beta_m)
\\
&= \int_\Omega \frac{b^\prime(v^\beta_m)}{\alpha}(\tilde{\chi} - v^\beta_m) \varphi^*(\beps(\alpha\bu^\beta_m)) + \frac{1}{2\epsilon} (v^\beta_m - 1) (\tilde{\chi} - v^\beta_m) + 2\epsilon\nabla v^\beta_m \cdot \nabla (\tilde{\chi} - v^\beta_m) \,\mathrm{d}x.
\end{aligned}
\end{equation}
In particular, for every \( \chi \in H^1_D(\Omega) \) such that \( \chi \leq 0 \),
\begin{equation}\label{p:equ17}
0 \leq \big[ \partial_v\mathcal{E}(\bu^\beta_m , v^\beta_m ) + \mathcal{H}^\prime(v^\beta_m) \big](\chi).
\end{equation}
\end{proposition}

\begin{corollary}\label{p:cor1}
Let the assumptions of Theorem \ref{p:thm3} hold. For every \( 1\leq m \leq M \), we have
\begin{align*}
0 = \big[ \partial_v\mathcal{E}(\bu^\beta_m , v^\beta_m ) + \mathcal{H}^\prime(v^\beta_m) \big](\delta v^\beta_m).
\end{align*}
\end{corollary}

\begin{proof}
Using (\ref{p:equ10}), we see that (\ref{p:equ17}) also holds for \( m= 0 \). 
Reasoning as in the proof of Lemma 3.2 of \cite{MR2673410}, we immediately obtain the following. Indeed, we test in (\ref{p:equ11}) against \( \tilde{\chi} = v^\beta_{m-1}\) for ``\(\geq \)'' and against \( \tilde{\chi} = 2v^\beta_m - v^\beta_{m-1}\) for the opposite direction ``\(\leq\)''.
\end{proof}

\begin{proposition}[Discrete energy-dissipation inequality]\label{p:prop3}
Let the assumptions of Theorem \ref{p:thm3} hold. The sequence of solutions to (\ref{p:equ7}), (\ref{p:equ8}) satisfies, for every \( 1\leq m \leq M \),
\begin{align*}
&\mathcal{F}(t^\beta_m;\bu^\beta_m, \delta\bu^\beta_m , v^\beta_m) + h\sum_{j=1}^m \langle \delta l^\beta_j, \bu^\beta_j \rangle
\\
&\quad
+ h \sum_{j=1}^m \int_\Omega b(v^\beta_{j-1}) \big[ F^{-1}(\beps(\delta\bu^\beta_j + \alpha\bu^\beta_j)) - F^{-1}(\beps(\alpha\bu^\beta_m)) \big] :\beps(\delta\bu^\beta_m) \,\mathrm{d}x
\\
&\leq \mathcal{F}(0; \bu_{0,N},  \bu_{1,N}, v^{N}_0).
\end{align*}
\end{proposition}

\begin{proof}
We test in the elastodynamic equation (\ref{p:equ7}) against \( h\delta\bu^\beta_m \). The interial term, that is, the one involving \( \delta^2\bu^\beta_m\), can be rewritten as
\begin{equation}\label{p:equ13}
\begin{aligned}
\int_\Omega \delta^2\bu^\beta_m \cdot h \delta\bu^\beta_m \,\mathrm{d}x &= \int_\Omega \big( \delta\bu^\beta_m - \delta\bu^\beta_{m-1}\big) \cdot \delta\bu^\beta_m \,\mathrm{d}x
\\
&= \frac{1}{2}\Big( \|\delta\bu^\beta_m\|_2^2 - \|\delta\bu^\beta_{m-1}\|_2^2 + \|\delta\bu^\beta_m - \delta\bu^\beta_{m-1}\|_2^2 \Big),
\end{aligned}
\end{equation}
as a result of the well-known relation
\begin{equation}\label{p:equ35}
\mathbf{a}\cdot (\mathbf{a} - \mathbf{b}) = \frac{1}{2}\big( |\mathbf{a}|^2 - |\mathbf{b}|^2 + |\mathbf{a} - \mathbf{b}|^2\big) \quad \text{ for all }\mathbf{a},\,\mathbf{b}\in \mathbb{R}^d.
\end{equation}
For the term involving the external force, we write
\begin{equation}\label{p:equ14}
\langle l^\beta_m, h\delta\bu^\beta_m 
\rangle 
= \langle l^\beta_m , \bu^\beta_m \rangle - \langle l^\beta_{m-1}, \bu^\beta_{m-1}\rangle + h\langle \delta l^\beta_m, \bu^\beta_{m-1}\rangle.
\end{equation}
For the term involving the nonlinearity, we must be more careful. First, we write
\begin{equation}\label{p:equ15}
\begin{aligned}
&\int_\Omega b(v^\beta_{m-1}) F^{-1}(\beps(\delta\bu^\beta_m + \alpha\bu^\beta_m)):\beps(h\delta\bu^\beta_m)\,\mathrm{d}x
\\&= h \int_\Omega b(v^\beta_{m-1}) \big[ F^{-1}(\beps(\delta\bu^\beta_m + \alpha\bu^\beta_m)) - F^{-1}(\beps( \alpha\bu^\beta_m))\big]:\beps(\delta\bu^\beta_m) \,\mathrm{d}x 
\\&\quad
+ \int_\Omega  b(v^\beta_{m-1}) F^{-1}(\beps( \alpha\bu^\beta_m)):\beps(h\delta\bu^\beta_m)\,\mathrm{d}x.
\end{aligned}
\end{equation}
The first term on the right-hand side of (\ref{p:equ15}) is in the form that we want. 
For the second term, using that \(F^{-1}\) is the derivative of \(\varphi^*\), we get
\begin{equation}\label{p:equ16}
\begin{aligned}
&\int_\Omega b(v^\beta_{m-1}) F^{-1}(\beps( \alpha\bu^\beta_m)):\beps(h\delta\bu^\beta_m)\,\mathrm{d}x 
\\
&= \int_\Omega\int_0^1 b(v^\beta_{m-1}) F^{-1}(\alpha\beps(s\bu^\beta_m + (1-s) \bu^\beta_{m-1})) : \beps(\bu^\beta_m - \bu^\beta_{m-1})\,\mathrm{d}s \,\mathrm{d}x
\\
&\quad + \int_\Omega \int_0^1 b(v^\beta_{m-1}) \big[ F^{-1}(\beps( \alpha\bu^\beta_m)) - F^{-1}(\alpha\beps(s\bu^\beta_m + (1-s) \bu^\beta_{m-1}))\big] : \beps(\bu^\beta_m - \bu^\beta_{m-1})\,\mathrm{d}s \,\mathrm{d}x
\\
&\geq \int_\Omega\int_0^1 b(v^\beta_{m-1}) F^{-1}(\alpha\beps(s\bu^\beta_m + (1-s) \bu^\beta_{m-1})) : \beps(\bu^\beta_m - \bu^\beta_{m-1})\,\mathrm{d}s \,\mathrm{d}x
\\
&= \int_\Omega \frac{b(v^\beta_m)}{\alpha}\varphi^*(\beps(\alpha\bu^\beta_m)) - \frac{b(v^\beta_{m-1}) }{\alpha}\varphi^*(\beps(\alpha\bu^\beta_{m-1})) + \frac{b(v^\beta_{m-1}) - b(v^\beta_m) }{\alpha}\varphi^*(\beps(\alpha\bu^\beta_m)) \,\mathrm{d}x
\\
&\geq \mathcal{E}(\bu^\beta_m, v^\beta_m) - \mathcal{E}(\bu^\beta_{m-1}, v^\beta_{m-1}) + \mathcal{H}(v^\beta_m) - \mathcal{H}(v^\beta_{m-1}),
\end{aligned}
\end{equation}
where the transition to the final line follows from the minimisation problem (\ref{p:equ7}). 
Summing (\ref{p:equ13}), (\ref{p:equ14}), (\ref{p:equ15}) and (\ref{p:equ16}), we deduce that
\begin{align*}
&\mathcal{F}(t^\beta_m;\bu^\beta_m, \delta\bu^\beta_m, v^\beta_m) + h\langle \delta l^\beta_m, \bu^\beta_{m-1}\rangle
\\
&\quad
+ h\int_\Omega b(v^\beta_{m-1}) \big[ F^{-1}(\beps(\delta\bu^\beta_m + \alpha\bu^\beta_m)) - F^{-1}(\beps(\alpha\bu^\beta_m)) \big] : \beps(\delta\bu^\beta_m) \,\mathrm{d}x
\\
&\leq \mathcal{F}(t^\beta_{m-1};\bu^\beta_{m-1},\delta\bu^\beta_{m-1}, v^\beta_{m-1}).
\end{align*}
Using this inequality recursively, we obtain the required result.
\end{proof}

\subsection{\(M\)-independent {\it a priori} bounds}
Now we focus on finding estimates on the sequence of solutions to the time discrete problem that are independent of  \( M \). This is so that we can take the limit as \( M \rightarrow\infty \) and obtain sufficiently strong convergence results. All constants may depend on the data  and  parameters of the original problem, but we specify the dependence on \( N \) if it is present. The first bound is  an immediate corollary of the energy-dissipation inequality.

\begin{lemma}\label{p:lem1}
Let the assumptions of Theorem \ref{p:thm3} hold and suppose that (\ref{p:equ5}) holds.  There exists a constant \( C = C(N) \), independent of \( M \), such that
\begin{align*}
\max_{1\leq m \leq M}\|\bu^\beta_m\|_{V_N} + \max_{1\leq m \leq M} \|\delta\bu^\beta_m \|_{V_N} + \max_{1\leq m \leq M} \|v^\beta_m \|_{1,2}\leq C.
\end{align*}
\end{lemma}
\begin{proof}
Rearranging the energy dissipation inequality and using monotonicity of \( F \) to eliminate terms, we see that
\begin{equation}\label{p:equ37}
\begin{aligned}
&\mathcal{K}(\delta\bu^\beta_m) + \mathcal{E}(\bu^\beta_m , v^\beta_m)+ \mathcal{H}(v^\beta_m) 
\\&\leq
\mathcal{K}(\delta\bu^\beta_0) + \mathcal{E}(\bu^\beta_0, v^\beta_0) + \mathcal{H}(v^\beta_0) - \langle l^\beta_0, \bu^\beta_0\rangle - h\sum_{j=1}^m \langle \delta l^\beta_j , \bu^\beta_{j-1}\rangle + \langle l^\beta_m, \bu^\beta_m \rangle
\\
&\leq \frac{\|\bu_{1,N}\|_2^2}{2} + \frac{b(1)}{\alpha}\int_\Omega \varphi^*(\beps(\alpha \bu_{0,N})) \,\mathrm{d}x + \mathcal{H}(v^{N}_0)
\\&\quad
+ C_T\|l(0) \|_{-1,2}\|\bu_{0,N}\|_{1,2}
+ \frac{h}{2} \sum_{j=1}^m\|\delta l^\beta_j \|_{-1,2}^2 + \frac{C_T h}{2}\sum_{j=0}^{m-1} \|\bu^\beta_{j} \|_{1,2}^2 
\\&\quad
+ C_T\|l\|_{L^\infty(0, T; W^{-1,2}_D(\Omega))} \|\bu^\beta_m \|_{1,2}.
\end{aligned}
\end{equation}
Using  that \( \bu^\beta_m = \bu^\beta_0 + \sum_{j=1}^m\delta\bu^\beta_j \) and adding \( \|\bu^\beta_m \|_2^2 \) to both sides of (\ref{p:equ37}), we apply Young's inequality and the equivalence of norms in \( V_N \) to deduce that
\begin{align*}
&{\|\bu^\beta_m \|_2^2} +  {\|\delta\bu^\beta_m \|_2^2} +\mathcal{E}(\bu^\beta_m, v^\beta_m)  +  \mathcal{H}(v^\beta_m)
\\
&\leq C(N)\Big[ \|\bu_{1,N}\|_{1,2}^2 + \| \bu_{0,N}\|_{1,2}^2 + \|  \bu_{0,N}\|_{1,p^\prime}^{p^\prime}+ \sup_{N\in \mathbb{N}}\|v^{N}_0\|_{1,2}^2 
 +\|l\|_{L^\infty(0, T; W^{-1,2}_D(\Omega))}^2 
  \\&\quad
  + \|l_t\|_{L^2(0, T; W^{-1,2}_D(\Omega))} 
 + h \sum_{j=1}^{m-1}\Big( \|\bu^\beta_j\|_2^2 + \|\delta\bu^\beta_j \|_2^2\Big) \Big]. 
\end{align*}
Applying the discrete Gronwall inequality and equivalence of norms again, we maximise over the indices \( 1\leq m \leq M\) to obtain the stated  result.
\end{proof}

Due to the regularity of the functions in \( V_N \),   Lemma \ref{p:lem1} gives a very strong bound on the solution sequence. Taking the limit in the time discretisation first allows us to   maintain these nice regularity properties as long as possible.  Without the compatibility condition (\ref{p:equ5}),  an \(M \)-independent bound is still possible as the above result shows. However,   later \( N \)-independent bounds are not possible without the assumption, hence why it is included in later results concerning \( N \)-independent bounds. 

The phase-field function experiences a nonlinearity in the elastodynamic equation, by which we mean the presence of the term \( b(v^\beta_m) \). Hence, to identify the resulting limit, we require  a strong convergence result for   \((v^\beta_m )_{m=1}^M \). To this end, we look for a bound on the discrete time derivatives \((\delta v^\beta_m)_{m=1}^M\). Due to the nonlinearity in the constitutive relation, we cannot argue as in the proof of Lemma 3.4 from \cite{MR2673410}. In particular, we cannot `transfer' the effects of the nonlinearity between \(\beps(\bu) \) and its derivative in time. 
Instead we  make use of  the embedding \( V_N \subset W^{k,2}_D(\Omega) \subset C^1(\overline{\Omega})\), a consequence of the Sobolev embedding theorem for \( k > \frac{d}{2}  + 1\), and using Corollary \ref{p:cor1}. 

\begin{lemma}\label{p:lem2}
Let the assumptions of Lemma \ref{p:lem1} hold. There exists a constant \( C = C(N) \) independent of \( M \) such that
\begin{align*}
\max_{1\leq m \leq M}\|\delta v^\beta_m\|_{1,2}\leq C.
\end{align*}
\end{lemma}
\begin{proof}
Using Proposition \ref{p:prop2} and Corollary \ref{p:cor1}, for every \( 1\leq m \leq M\), we have
\begin{equation}\label{p:equ38}
0  = \big[ \partial_v\mathcal{E}(\bu^\beta_m , v^\beta_m ) + \mathcal{H}^\prime(v^\beta_m) \big](\delta v^\beta_m),
\end{equation}
and 
\begin{equation}\label{p:equ39}
0 \leq \big[ \partial_v\mathcal{E}(\bu^\beta_{m-1} , v^\beta_{m-1} ) + \mathcal{H}^\prime(v^\beta_{m-1}) \big](\delta v^\beta_m).
\end{equation}
Subtracting (\ref{p:equ38}) from (\ref{p:equ39}), we get
\begin{align*}
\frac{1}{2\epsilon}\|\delta v^\beta_m\|_2^2 + 2\epsilon\|\nabla \delta v^\beta_m \|_2^2 \leq \frac{1}{h}\int_\Omega \Big[ \frac{b^\prime(v^\beta_{m-1})}{\alpha}\varphi^*(\beps(\alpha\bu^\beta_{m-1})) - \frac{b^\prime(v^\beta_m)}{\alpha}\varphi^*(\beps(\alpha\bu^\beta_m))\Big]\cdot \delta v^\beta_m\,\mathrm{d}x.
\end{align*}
We add and subtract  \( \alpha^{-1}b(v^{\beta}_{m-1}) \varphi^*(\beps(\alpha\bu^\beta_m)) \) to the right-hand side of the above and use that \( b^\prime(v) = 2v\)  to deduce that
\begin{align*}
&\frac{1}{2\epsilon}\|\delta v^\beta_m\|_2^2 + 2\epsilon\|\nabla \delta v^\beta_m \|_2^2  + \frac{2}{\alpha} \int_\Omega |\delta v^\beta_m|^2 \varphi^*(\beps(\alpha\bu^\beta_m)) \,\mathrm{d}x
\\&\leq \int_\Omega\int_0^1 b^\prime(v^\beta_{m-1}) \delta v^\beta_m F^{-1}(\alpha\beps(s\bu^\beta_{m-1} + (1-s)\bu^\beta_m)) :\beps(\bu^\beta_{m-1} - \bu^\beta_m) \,\mathrm{d}s\,\mathrm{d}x
\\
&\leq  2\|\delta v^\beta_m\|_2\big[ \|\beps(\bu^\beta_{m-1}) \|_2 + \|\beps(\bu^\beta_m) \|_2 \big]\cdot \big[ \|\beps(\bu^\beta_{m-1}) \|_\infty^{p^\prime - 1} + \|\beps(\bu^\beta_{m}) \|_\infty^{p^\prime - 1} \big]
\\
&\leq \frac{\|\delta v^\beta_m\|_2^2}{4\epsilon} + C(N) \Big[\max_{0\leq j \leq M } \|\bu^\beta_j \|_{V_N}^{2 p^\prime}  + \max_{1\leq j \leq M}\|\delta\bu^\beta_j \|_{V_N}^{2p^\prime} \Big],
\end{align*}
using H\"{o}lder's inequality, followed by Young's inequality. Applying Lemma \ref{p:lem1}, the  result follows. 
\end{proof}

In the elastodynamic equation for weak energy solutions of (\ref{p:equ1})--(\ref{p:equ3}), we have the interial term \( \bu_{tt}\). 
Hence, we should look for a bound on the sequence \((\delta^2\bu^\beta_m)_{m=1}^M \). The finite-dimensional setting of \( V_N \) simplifies this matter significantly. Indeed, we need only use the weak elastodynamic equation (\ref{p:equ47}), the equivalence of norms in finite-dimensional spaces and the previously determined bounds. 

\begin{lemma}\label{p:lem3}
Let the assumptions of Lemma \ref{p:lem1} hold. There exists a constant \( C = C(N)\) independent of \( M \) such that 
\begin{align*}
\max_{1\leq m \leq M} \|\delta^2\bu^\beta_m \|_{V_N} \leq C.
\end{align*}
\end{lemma}

\begin{proof}
By the uniform bound on \((v^\beta_m)_{m=0}^M \), the elastodynamic equation (\ref{p:equ7}) and Lemma \ref{p:lem1}, it  follows that 
\begin{align*}
\max_{1\leq m \leq M}\|\delta^2\bu^\beta_m\|_{V_N^*}\leq C(N ),
\end{align*}
where \( V_N^* \) is the dual space.  By equivalence of norms, the  result follows. 
\end{proof}

\subsection{The limit $ M \rightarrow\infty$}
Now we define   interpolants of the  solution sequence for the time discrete problem. We  derive bounds  on these functions  using the results of Lemmas \ref{p:lem1}, \ref{p:lem2} and \ref{p:lem3}, and then take the limit as \( M\rightarrow\infty \). 
The limiting functions define a solution of a time continuous Galerkin approximation  of (\ref{p:equ1})--(\ref{p:equ3}) with respect to \( V_N \). 
We define the piecewise affine, backwards and forwards interpolants, respectively, of the sequence \((\bu^{\beta}_m)_{m=0}^M\) by
\begin{align*}
\overline{\bu}^\beta (t) &= \frac{t-t^\beta_{m-1}}{h}\bu^\beta_m + \frac{t^\beta_m - t}{h}\bu^\beta_{m-1} &\quad& \text{ for }t\in [t^\beta_{m-1}, t^\beta_m],\, 1\leq m \leq M,
\\
\bu^{\beta,+}(t) & = \bu^\beta_m &\quad& \text{ for }t\in (t^\beta_{m-1}, t^\beta_m],\, 1\leq m \leq M,
\\
\bu^{\beta,-}(t) &= \bu^\beta_{m-1}&\quad& \text{ for }t\in [t^\beta_{m-1}, t^\beta_m),\, 1\leq m \leq M,
\end{align*}
where we extend \( \bu^{\beta,+}\) and \( \bu^{\beta,-}\) continuously to \( t = 0\) and \( t = T \), respectively. Similarly, we denote by \(\overline{\bu}^{\beta,\prime}\), \( \bu^{\beta, +,\prime}\), \( \bu^{\beta,-, \prime}\) the interpolants of \((\delta\bu^\beta_m)_{m=0}^M\), \( \bu^{\beta,+,\prime\prime}\) the backwards interpolant of \((\delta^2\bu^\beta_m)_{m=1}^M \), \( \overline{v}^\beta\), \( v^{\beta,+}\), \( v^{\beta,-}\) the interpolants of \((v^\beta_m)_{m=1}^M \), and \(v^{\beta,+,\prime}\) the backwards interpolant of \((\delta v^\beta_m)_{m=1}^M\). We  can identify
\begin{align*}
\overline{\bu}^\beta_t = \bu^{\beta,+,\prime},\quad \overline{\bu}^{\beta,\prime}_t = \bu^{\beta,+,\prime\prime}, \quad \overline{v}^\beta_t= v^{\beta,+,\prime}.
\end{align*}
By Lemmas \ref{p:lem1}, \ref{p:lem2} and \ref{p:lem3}, we deduce that there exists a constant \( C = C(N) \), independent of \( M \), such that
\begin{align*}
&\|\overline{\bu}^\beta\|_{W^{1,\infty}(V_N)} + \|\bu^{\beta,\pm}\|_{L^\infty(V_N) }
+ \|\overline{\bu}^{\beta,\prime}\|_{W^{1,\infty}(V_N)} + \|\bu^{\beta,\pm,\prime}\|_{L^\infty(V_N)} + \|\bu^{\beta,+,\prime\prime}\|_{L^\infty(V_N)} 
\\&\quad
+ \|\overline{v}^\beta\|_{W^{1,\infty}(H^1)} + \|v^{\beta,\pm}\|_{W^{1,\infty}(H^1)} + \|v^{\beta,+,\prime}\|_{L^\infty(H^1)}
\\&
\leq C.
\end{align*}
With this bound and standard weak compactness results for Bochner spaces, we obtain the  convergence result in Lemma \ref{p:lem4}. Various strong convergence results follow from application of the Aubin--Lions lemma and the compact embedding \( V_N \subset C^1(\overline{\Omega})^d\). 
We include only the ones that are used in identifying the limiting couple as a solution of the problem of interest. 
In particular, these strong convergence results are used for identification of the nonlinear terms in the limit. We collect the convergences in Lemma \ref{p:lem4}. 
We do not include the details of the proof as it follows standard procedures. 
Then we prove various properties of the limiting couple, denoted \(( \bu^N, v^N) \). In Proposition \ref{p:prop4}, we show that a weak elastodynamic equation holds with respect to the Galerkin approximation from \( V_N \), and in Proposition \ref{p:prop5} that an appropriate minimisation problem is satisfied at every point in \( [0, T]\). Finally, in Proposition \ref{p:prop6}, we prove that an energy-dissipation equality holds.  

\begin{lemma}\label{p:lem4}
Let the assumptions of Lemma \ref{p:lem1} hold and denote \( \beta = (N,M) \) where \( N \) is fixed.
There exists a subsequence in \( M\), independent of \( N \), not relabelled, and a limiting couple \((\bu^N, v^N) \) such that  \( \bu^N\in W^{2,\infty}(0, T; V_N) \) and  \( v^N\in W^{1,\infty}(0, T; H^1(\Omega)) \) with \( v^N(t) \in H^1_{D+1}(\Omega) \) for every \( t\in [0, T]\), and the following convergence results hold as \( M \rightarrow\infty \):
\begin{itemize}
\item \( \overline{\bu}^{(N,M)}, \, \overline{\bu}^{(N,M), \prime} \overset{\ast}{\rightharpoonup}\bu^N, \, \bu^N_t\) weakly-* in \( W^{1,\infty}(0, T; V_N)\), respectively;
\item  \( \overline{\bu}^{(N,M)}, \, \overline{\bu}^{(N,M), \prime} \rightarrow\bu^N, \, \bu^N_t\) strongly in \( C([0, T]; C^1(\overline{\Omega})^d) \), respectively; 
\item \( \bu^{(N,M),\pm}, \, \bu^{(N,M),\pm,\prime}\overset{\ast}{\rightharpoonup}\bu^N, \, \bu^N_t\) weakly-* in \( L^{\infty}(0, T; V_N)\), respectively;
\item \( \overline{v}^{(N,M)} \overset{\ast}{\rightharpoonup} v^N\)   weakly-* in \( W^{1,\infty}(0, T; H^1(\Omega))\);
\item \( v^{(N,M),\pm}\overset{\ast}{\rightharpoonup} v^N\) weakly-* in \( L^{\infty}(0, T; H^1(\Omega))\).
\end{itemize}
Furthermore, \( v^N\in [0,1]\)  and \( v^N_t\leq 0 \) a.e. in \( Q\).
\end{lemma}

\begin{proposition}\label{p:prop4}
Let the assumptions of Lemma \ref{p:lem1} hold and let \((\bu^N, v^N) \) be the limiting couple constructed in Lemma \ref{p:lem4}. For a.e. \( t\in (0, T) \) and every \( \bw \in V_N \), we have
\begin{equation}\label{p:equ18}
\int_\Omega\bu^N_{tt}(t) \cdot \bw + b(v^N(t)) F^{-1}(\beps(\bu^N_t + \alpha\bu^N)) :\beps(\bw) \,\mathrm{d}x = \langle l(t), \bw\rangle.
\end{equation}
The initial conditions hold in the sense that
\begin{align*}
\lim_{t\rightarrow 0+} \Big[ \|\bu^N(t) - \bu_{0,N} \|_{V_N} + \|\bu^N_t(t) - \bu_{1,N} \|_{V_N} + \|v^N(t) - v^{N}_0\|_{1,2}\Big] = 0.
\end{align*}
\end{proposition}

\begin{proof}
For (\ref{p:equ18}),  rewriting (\ref{p:equ7}) in terms of the interpolants, for every \( M \in\mathbb{N}\) and a.e. \( t\in (0, T) \), we first notice that we have
\begin{align*}
\int_\Omega \overline{\bu}^{\beta,+,\prime}_t(t) \cdot \bw + b(v^{\beta,-}(t))F^{-1}(\beps(\bu^{\beta,+,\prime}  + \alpha\bu^{\beta,+})(t)):\beps(\bw) \,\mathrm{d}x = \langle l^{\beta,+}(t), \bw\rangle,
\end{align*}
for every \( \bw \in V_N \). By the  convergence results in Lemma \ref{p:lem4}, we  deduce that, for every  \( t\in (0, T) \),
\begin{equation}\label{p:equ40}
\begin{aligned}
&\lim_{M\rightarrow\infty}\Big\{ \int_\Omega b(v^{\beta,-}(t))F^{-1}(\beps(\bu^{\beta,+,\prime}  + \alpha\bu^{\beta,+})(t))):\beps(\bw) \,\mathrm{d}x   - \langle l^{\beta,+}(t), \bw\rangle\Big\}
\\&= \int_\Omega b(v^N(t)) F^{-1}(\beps(\bu^N_t + \alpha\bu^N)(t)):\beps(\bw)\,\mathrm{d}x - \langle l(t), \bw\rangle,
\end{aligned}
\end{equation}
and,  for every \( \psi  \in C([0, T]) \), 
\begin{equation}\label{p:equ41}
\lim_{M\rightarrow\infty}\int_Q \overline{\bu}^{\beta,+,\prime}_t(t) \cdot \bw\psi  \,\mathrm{d}x\,\mathrm{d}t = \int_Q \bu^N_{tt}(t) \cdot \bw \psi  \,\mathrm{d}x\,\mathrm{d}t.
\end{equation}
Using (\ref{p:equ40}), (\ref{p:equ41}) and Lebesgue's differentiation theorem, we obtain (\ref{p:equ18}). Satisfaction of the initial conditions follows from standard arguments and regularity results for Bochner functions. (See \cite[Chapter 5]{evansPDE}, for example.)
\end{proof}

\begin{proposition}\label{p:prop5}
Let the assumptions of Lemma \ref{p:lem1} hold and let \((\bu^N, v^N) \) be the limiting couple constructed in Lemma \ref{p:lem4}. For every \( t\in [0, T]\), \( v^N(t) \) solves the variational problem
\begin{equation}\label{p:equ20}
\mathcal{E}(\bu^N(t), v^N(t)) + \mathcal{H}(v^N(t)) = \inf\Big\{ \mathcal{E}(\bu^N(t), v)+ \mathcal{H}(v)\,:\,v\in H^1_{D+1}(\Omega), \, v\leq v^N(t)\Big\}.
\end{equation}
\end{proposition}
\begin{proof}
By the construction of \((v^N_0)_N \), (\ref{p:equ20}) holds at \( t = 0\). 
For \( t\in (0, T]\), it suffices  to show that
\begin{equation}\label{p:equ22}
0 \leq \int_\Omega\frac{b^\prime(v^N(t))}{\alpha}\chi\varphi^*(\beps(\alpha\bu^N(t))) + \frac{1}{2\epsilon} \big( v^N(t)  -1 \big) \chi + 2\epsilon\nabla v^N(t) \cdot \nabla \chi \,\mathrm{d}x,
\end{equation}
for every \( \chi \in H^1_{D}(\Omega) \) with \( \chi \leq 0 \). However, by Proposition \ref{p:prop2}, for every \(\psi \in C([0, T]) \) with \( \psi \geq 0 \), we have
\begin{equation}\label{p:equ42}
0 \leq \int_Q \frac{b^\prime(v^{\beta,+})}{\alpha}\chi \psi  \varphi^*_n(\beps(\alpha\bu^{\beta,+})) + \frac{1}{2\epsilon}(v^{\beta,+}-1) \chi \psi  + 2\epsilon\nabla v^{\beta,+} \cdot \nabla (\chi \psi)  \,\mathrm{d}x\,\mathrm{d}t.
\end{equation}
Letting \( N \rightarrow\infty \) in (\ref{p:equ42}) and using the weak convergence results, an application of Lebesgue's differentiation theorem gives (\ref{p:equ22}) for a.e. \( t\in (0, T] \). Noticing that the right-hand side of (\ref{p:equ22}) is continuous as a function of \( t\), we conclude that (\ref{p:equ22}) holds for every \( t\in [0, T]\). 
\end{proof}


\begin{proposition}\label{p:prop6}
Let the assumptions of Lemma \ref{p:lem1} hold and let \((\bu^N, v^N) \)be  the limiting couple constructed in Lemma \ref{p:lem4}. For every \( t\in [0, T]\), the following energy-dissipation equality holds:
\begin{equation}\label{p:equ23}
\begin{aligned}
&\mathcal{F}(t;\bu^N(t), \bu^N_t(t), v^N(t)) + \int_0^t \langle l_t, \bu^N\rangle\,\mathrm{d}s
\\
&\quad
+ \int_0^t \int_\Omega b(v^N) \big[ F^{-1}(\beps(\bu^N_t + \alpha\bu^N)) - F^{-1}(\beps(\alpha\bu^N))\big]:\beps(\bu^N_t) \,\mathrm{d}x\,\mathrm{d}s
\\
&= \mathcal{F}(0;\bu_{0,N},  \bu_{1,N}, v^{N}_0).
\end{aligned}
\end{equation}
\end{proposition}

\begin{proof}
Trivially the result holds for  \( t = 0\).  Letting \( M\rightarrow\infty \) in the discrete energy-dissipation inequality and using weak lower semi-continuity, we deduce that (\ref{p:equ23}) holds with ``\(=\)'' replaced by ``\(\leq\)'', for a.e. \( t\in (0, T] \). 

For the opposite inequality, we first  notice that
\begin{align*}
&\mathcal{K}(\bu^N_t(t)) - \langle l(t), \bu^N(t) \rangle + \int_0^t \langle l_t,\bu^N\rangle\,\mathrm{d}s + \int_0^t \int_\Omega b(v^N) F^{-1}(\beps(\bu^N_t + \alpha\bu^N)):\beps(\bu^N_t)\,\mathrm{d}x\,\mathrm{d}s
\\
&\quad
- \mathcal{K}( \bu_{1,N}) + \langle l(0),\bu_{0,N}\rangle
\\
&= \int_0^t \int_\Omega \bu^N_{tt}\cdot \bu^N_t \,\mathrm{d}x\,\mathrm{d}s - \int_0^t \langle l,\bu^N_t\rangle \,\mathrm{d}s + \int_0^t \int_\Omega b(v^N)F^{-1}(\beps(\bu^N_t + \alpha\bu^N)):\beps(\bu^N_t)\,\mathrm{d}x\,\mathrm{d}s
\\
&=0,
\end{align*}
using (\ref{p:equ18}). Thus, it suffices to show that 
\begin{equation}\label{p:equ24}
\begin{aligned}
&\mathcal{E}(\bu^N(t), v^N(t)) + \mathcal{H}(v^N(t)) - \int_0^t \int_\Omega b(v^N)F^{-1}(\beps(\alpha\bu^N)):\beps(\bu^N_t) \,\mathrm{d}x\,\mathrm{d}s
\\
&\geq \mathcal{E}( \bu_{0,N}, v^{N}_0) + \mathcal{H}(v^{N}_0).
\end{aligned}
\end{equation}
To prove that (\ref{p:equ24}) holds, we mimic an argument from \cite{MR2673410}. Fixing  \( t\in (0, T]\), for each \( K\in \mathbb{N}_{\geq 2}\) we define a time step \( h_K = \frac{t}{K}\) and the corresponding time-discrete approximate solution sequence
\begin{align*}
\bu^{N,K}_k = \bu^N (hk) ,\quad v^{N,K}_k = v^N(hk)\quad\text{ for }0 \leq k \leq K.
\end{align*}
Then we can  write
\begin{equation}\label{p:equ25}
\mathcal{E}(\bu^N(t), v^N(t)) - \mathcal{E}( \bu_{0,N}, v^{N}_0) = \sum_{k=1}^K \Big[ \mathcal{E}(\bu^{N,K}_k, v^{N,K}_k)  - \mathcal{E}(\bu^{N,K}_{k-1}, v^{N,K}_{k-1}) \Big].
\end{equation}
Using the variational problem (\ref{p:equ20}), we deduce that
\begin{equation}\label{p:equ26}
\begin{aligned}
&\mathcal{E}(\bu^{N,K}_k, v^{N,K}_k)  - \mathcal{E}(\bu^{N,K}_{k-1}, v^{N,K}_{k-1})
\\
&\geq \int_\Omega \frac{b(v^{N,K}_k)}{\alpha} \Big( \varphi^*(\beps(\alpha\bu^{N,K}_k)) - \varphi^*(\beps(\alpha\bu^{N,K}_{k-1})) \Big) \,\mathrm{d}x + \mathcal{H}(v^{N,K}_{k-1}) - \mathcal{H}(v^{N,K}_k)
\\
&= \frac{t}{K}\int_\Omega\int_0^1 \frac{b(v^{N,K}_k)}{\alpha} F^{-1}(\alpha\beps(s\bu^{N,K}_k + (1-s)\bu^{N,K}_{k-1})):\beps(\delta\bu^{N,K}_{k}) \,\mathrm{d}s\,\mathrm{d}x
\\&\quad
 + \mathcal{H}(v^{N,K}_{k-1})  - \mathcal{H}(v^{N,K}_k).
\end{aligned}
\end{equation}
Using the previous notation for interpolants of discrete solution sequences, we insert (\ref{p:equ26}) into (\ref{p:equ25}) to obtain
\begin{equation}\label{p:equ49}
\begin{aligned}
&\mathcal{E}(\bu^N(t), v^N(t)) - \mathcal{E}( \bu_{0,N}, v^{N}_0) - \mathcal{H}(v^{N}_0) + \mathcal{H}(v^N(t))
\\
&\geq \int_0^t\int_\Omega\int_0^1 \frac{b(v^{N,K,+})}{\alpha}F^{-1}(\alpha\beps(s\bu^{N,K,+}+(1-s)\bu^{N,K,-})):\beps(\overline{\bu}^{N,K}_t)\,\mathrm{d}s\,\mathrm{d}x\,\mathrm{d}\tau.
\end{aligned}
\end{equation}
Using that  \(\bu^N\in W^{2,\infty}(0, T; V_N) \) and \( v^N\in W^{1, \infty}(0, T; H^1(\Omega)) \cap L^\infty(Q)\), we see that 
\begin{itemize}
\item \( \bu^{N,K,\pm}\rightarrow\bu^N\) strongly in \( C([0, T];V_N)\);
\item \( \overline{\bu}^{N,K}\rightarrow\bu^N\) strongly in \( W^{1,\infty}(0, T; V_N) \);
\item \( v^{N,K,+}\rightarrow v^N\) strongly in \( L^\infty(0, T; L^2(\Omega)^d) \).
\end{itemize}
Letting \( K \rightarrow\infty \) in (\ref{p:equ49}), it follows that
\begin{align*}
&\mathcal{E}(\bu^N(t), v^N(t)) - \mathcal{E}( \bu_{0,N}, v^{N}_0) - \mathcal{H}(v^{N}_0) + \mathcal{H}(v^N(t))
\\&\geq \int_0^t\int_\Omega \frac{b(v^N)}{\alpha}F^{-1}(\beps(\alpha\bu^N)) :\beps(\bu^N_t) \,\mathrm{d}x\,\mathrm{d}s.
\end{align*}
Hence the energy-dissipation equality holds for a.e. \( t\in [0, T]\). Using the continuity of the solution couple, we conclude that the equality must hold for every \( t\in [0, T]\).
\end{proof}

\subsection{\(N\)-independent bounds}\label{sec:Nbounds}
In this section, we derive \( N \)-independent bounds on the solution couple \(( \bu^N, v^N) \) constructed in Lemma \ref{p:lem4}. This allows us to take \( N \rightarrow\infty \) and obtain a weak energy solution of the original problem (\ref{p:equ1})--(\ref{p:equ3}). We reintroduce the stress tensor at this stage to simplify the representation. Also, we use the Korn--Poincar\'{e} inequality (Theorem \ref{korn}) throughout this section. 
The first two results provide   bounds on the displacement \( \bu \) and stress tensor \( \bbT \). In Lemma \ref{p:lem9}, we get a bound on the phase-field function \( v\). Following this, we differentiate between the cases \( p \in (1,2]\) and \( p \in (2,\infty) \). When \( p \in (1,2]\) we are able to improve the bounds on the phase-field function. Indeed, we obtain bounds on the time derivative in the spirit of of Lemma \ref{p:lem2}. We are unable to obtain such a result when \( p > 2\) for technical reasons. Instead, we make use of the monotonicity of the phase-field functions as a function of the time variable. These two cases are discussed in Section \ref{sec:psmall} and \ref{sec:plarge}, respectively. 


\begin{lemma}\label{p:lem7}
Let the assumptions of Lemma \ref{p:lem4} hold and let \((\bu^N, \bbT^N, v^N) \) be the solution triple constructed there, where we define \( \bbT^N := F^{-1}(\beps(\bu^N_t + \alpha\bu^N)) \). There exists a constant \( C\), independent of \( N \), such that
\begin{align*}
\sup_{t\in [0, T]}\|\bu^N(t) \|_2 +\sup_{t\in [0, T]}\|\bu^N(t)  \|_{1,p^\prime} + \sup_{t\in [0, T]}\|\bu^N_t(t) \|_2 + \int_0^T \|\bbT^N(t)\|_p^p + \|\bu^N_t(t) \|_{1,p^\prime}^{p^\prime} \,\mathrm{d}t  \leq C.
\end{align*}
\end{lemma}

\begin{proof}
For  the first, third and fourth bounds, we test against \(\bu^N_t + \alpha\bu^N\) in the elastodynamic equation (\ref{p:equ18}). Using the lower bound for \( b \), we can argue as in \cite{mypaperpreprint} and so do not include the details here. Using the constitutive relation and the resulting memory kernel property, for every \( t\in [0, T]\), we have
\begin{align*}
|\beps(\bu^N(t))| \leq \mathrm{e}^{-\alpha t}|\beps(\bu_{0,N})| + \int_0^t \mathrm{e}^{\alpha(s-t)} |\bbT^N|^{p-1}\,\mathrm{d}s.
\end{align*}
Taking the \( L^{p^\prime}(\Omega) \) norm of both sides, using the properties of the approximation of the initial data, and applying the Korn--Poincar\'{e} inequality, the \( L^\infty(0, T; W^{1,p^\prime}(\Omega)^{d}) \) bound on \((\bu^N)_N \) follows. Using the bound on \( \bbT^N \) with the constitutive relation, we see  that \( (\beps(\bu^N_t + \alpha\bu^N))_N \) is  bounded in \( L^{p^\prime}(Q)^{d\times d}\). Putting this together with the bound on \( \bu^N\) and applying the Korn--Poincar\'{e} inequality, the final bound on \( \bu^N_t\) follows.
\end{proof}

\begin{lemma}\label{p:lem8}
Let the assumptions of Lemma \ref{p:lem4} hold and let \((\bu^N, \bbT^N, v^N) \) be the solution triple constructed there. There exists a constant \( C\) independent of \( N \) such that
\begin{align*}
\int_Q |\bu^N_{tt}|^2 \,\mathrm{d}x \,\mathrm{d}t + \sup_{t\in [0, T]}\Big\{ \|\bbT^N(t) \|_p + \|\bu^N_t(t)\|_{1,p^\prime}\Big\}\leq C.
\end{align*}
\end{lemma}

\begin{proof}
It suffices to prove the claim for the first two terms, repeating the reasoning from the proof of Lemma \ref{p:lem7} in order to obtain the bound on the final term. We test in the elastodynamic equation (\ref{p:equ18}) against \(\bu^N_{tt} + \alpha\bu^N_t\) to get
\begin{align*}
&\int_\Omega |\bu^N_{tt}|^2 \,\mathrm{d}x + \frac{\mathrm{d}}{\mathrm{d}t }\Big( \frac{\alpha}{2}\|\bu^N_t\|_2^2\Big)  + \int_\Omega b(v^N) \bbT^N :F(\bbT^N)_t\,\mathrm{d}x 
\\&= \langle l, \bu^N_{tt} + \alpha\bu^N_t\rangle
\\
&= \frac{\mathrm{d}}{\mathrm{d}t }\Big( \langle l, \bu^N_{t}+\alpha\bu^N\rangle \Big) - \langle l_t,\bu^N_t + \alpha\bu^N\rangle.
\end{align*}
The term involving the nonlinearity can be rewritten as
\begin{align*}
\int_\Omega b(v^N) \bbT^N :F(\bbT^N)_t\,\mathrm{d}x  &= \frac{\mathrm{d}}{\mathrm{d}t }\Big( \int_\Omega b(v^N) \frac{p-1}{p}|\bbT^N|^p\,\mathrm{d}x\Big) - \int_\Omega b(v^N) v^N_t \Big[ \frac{p-1}{p}|\bbT^N|^p\Big]\,\mathrm{d}x
\\
&\geq \frac{\mathrm{d}}{\mathrm{d}t }\Big( \int_\Omega b(v^N) \frac{p-1}{p}|\bbT^N|^p\,\mathrm{d}x\Big),
\end{align*}
recalling the non-positivity of   \( v^N_t\). Integrating over \((0, t) \) for an arbitrary  \( t\in (0, T]\), we deduce that
\begin{equation}\label{p:equ48}
\begin{aligned}
&\int_0^t \int_\Omega |\bu^N_{tt}|^2 \,\mathrm{d}x\,\mathrm{d}s + \int_\Omega b(v^N(t)) \frac{p-1}{p}|\bbT^N(t)|^p\,\mathrm{d}x
\\
&\leq  \frac{\alpha}{2}\| \bu_{1,N}\|_2^2 + \int_\Omega b(v^N_0) \frac{p-1}{p}|\bbT^N(0)|^p\,\mathrm{d}x + \langle l(t), (\bu^N_t + \alpha\bu^N)(t) \rangle
\\
&\quad
- \langle l(0), \bu_{1, N} + \alpha\bu_{0, N} \rangle - \int_0^t \langle l_t , \bu^N_t + \alpha\bu^N\rangle \,\mathrm{d}s
\\
&\leq C\Big[ \|\bu_{1,N}\|_2^2 +  \int_\Omega|F^{-1}(\beps(\bu_{1,N} + \alpha\bu_{0,N})) |^{p}\,\mathrm{d}x  + \|l(t) \|_{-1,p}\|\beps(\bu^N_t+\alpha\bu^N)(t) \|_{p^\prime}
\\&\quad
+ \|l(0) \|_{-1,p}\|\beps(\bu_{1,N} + \alpha\bu_{0,N})\|_{p^\prime} + \int_0^t \|l_t\|_{-1,p}\|\beps(\bu^N_t+ \alpha\bu^N) \|_{p^\prime}\,\mathrm{d}s
\Big]
\\
&\leq C\Big[ \|\bu_{1,N}\|_2^2 +  \|\beps(\bu_{1,N} + \alpha\bu_{0,N}) \|_{p^\prime}^{p^\prime} + \|l\|_{W^{1,\infty}(W^{-1,p})}\|\beps(\bu^N_t + \alpha\bu^N)(t) \|_{p^\prime}
\\&\quad
+ \|l\|_{W^{1,\infty}(W^{-1,p})}\|\beps(\bu_{1,N}+ \alpha\bu_{0,N}) \|_{p^\prime} + \int_0^t \|l_t\|_{-1,p}^p + \|\beps(\bu^N_t + \alpha\bu^N) \|_{p^\prime}^{p^\prime}\,\mathrm{d}s\Big].
\end{aligned}
\end{equation}
However, we have that
\begin{align*}
|\bbT^N(t) |^p = |F^{-1}(\beps(\bu^N_t + \alpha\bu^N)(t))|^{p} = |\beps(\bu^N_t + \alpha\bu^N) |^{p^\prime}.
\end{align*}
Using this on the right-hand side of (\ref{p:equ48}) and applying Young's inequality, we deduce that
\begin{align*}
&\int_0^t \int_\Omega |\bu^N_{tt}|^2 \,\mathrm{d}x \,\mathrm{d}s + \int_\Omega |\bbT^N(t) |^p + |\beps(\bu^N_t + \alpha\bu^N)(t) |^{p^\prime}\,\mathrm{d}x
\\
&\leq 
C\Big[ \|\bu_{1,N}\|_2^2 +  \|\beps(\bu_{1,N} + \alpha\bu_{0,N}) \|_{p^\prime}^{p^\prime} + \|l\|_{W^{1,\infty}(W^{-1,p})}^p + \int_0^t \|l_t\|_{-1,p}^p + \|\beps(\bu^N_t + \alpha\bu^n) \|_{p^\prime}^{p^\prime} \,\mathrm{d}s\Big]
\end{align*}
Using Lemma \ref{p:lem7} and the properties of the approximations of the initial data (in particular, the boundedness in both \( L^2(\Omega)^d \) and \( W^{1,p^\prime}_D(\Omega)^d\) irrespective of the value of \(p\)), the right-hand side can be bounded above independent of \( N\) so the required result follows. 
\end{proof}


\begin{lemma}\label{p:lem9}
Let the assumptions of Lemma \ref{p:lem4} hold. Let \((\bu^N, \bbT^N, v^N) \) be the solution triple constructed there. Suppose additionally that the compatibility condition (\ref{p:equ5}) holds. There exists a constant \( C\), independent of \( N \), such that
\begin{align*}
\sup_{t\in [0, T]}\|v^N(t) \|_{1,2}\leq C.
\end{align*}
\end{lemma}

\begin{proof}
By the energy-dissipation equality, we have
\begin{align*}
\mathcal{H}(v^N(t)) 
&\leq \mathcal{F} (t; \bu^N(t), \bu^N_t(t), v^N(t)) + \langle l(t),\bu^N(t) \rangle
\\&\quad
+\int_0^t \int_\Omega b(v^N) \big[ F^{-1}( \beps(\bu^N_t + \alpha\bu^N)) - F^{-1}(\beps(\alpha\bu^N)) \big] :\beps(\bu^N_t) \,\mathrm{d}x\,\mathrm{d}s
\\
&=  \langle l(t),\bu^N(t) \rangle - \int_0^t \langle l_t,\bu^N\rangle\,\mathrm{d}s   + \mathcal{F}(0; \bu_{0,N}, \bu_{1,N}, v^N_0)
\\
&\leq C\Big[ \|l(t) \|_{-1,p}\|\beps(\bu^N(t) ) \|_{p^\prime} + \|\bu_{1,N}\|_2^2 + \|l(0) \|_{-1,p}\|\beps(\bu_{0,N}) \|_{p^\prime} 
\\&\quad
+ \int_\Omega b(v^N_0) \varphi^*(\beps(\alpha\bu_{0,N}))) \,\mathrm{d}x + \|v^N_0\|_{1,2}^2 + \int_0^t \|l_t(s)\|_{-1,p}\|\beps(\bu^N(s))\|_{p^\prime}\,\mathrm{d}s\Big].
\end{align*}
Using the boundedness of \((v^N_0)_N \) in \( H^1_D(\Omega)\) and \( (\bu_{0,N})_N \), \((\bu_{1,N})_N \) in \( L^2(\Omega)^d\cap W^{1,p^\prime}_D(\Omega)^d\), alongside the bounds in Lemma \ref{p:lem7}, the right-hand side can be bounded above independent of \( N \) and \( t\in [0, T]\). Taking the supremum with respect to the time variable, we use the definition of the functional \( \mathcal{H}\) to yield the stated result. 
\end{proof}

\subsection{Limit in the  case \( p \in (1,2]\)}\label{sec:psmall}
To replicate the ideas of Lemma \ref{p:lem2}, we first need to show that testing in the variational form of the minimisation problem against \( v^N_t\) causes the right-hand side of (\ref{p:equ22}) to vanish. Namely, we need to prove the time continuous analogue of Corollary \ref{p:cor1}. 

\begin{lemma}\label{p:lem10}
Let the assumptions of Lemma \ref{p:lem4} hold and let  \((\bu^N, \bbT^N, v^N) \) be the solution triple constructed there. For a.e. \( t\in (0, T) \), we have
\begin{equation}\label{p:equ43}
\big[ \partial_v\mathcal{E}(\bu^N(t), v^N(t)) + \mathcal{H}^\prime(v^N(t)) \big](v^N_t(t)) = 0. 
\end{equation}
\end{lemma}
\begin{proof}
Differentiating the energy-dissipation equality from Proposition \ref{p:prop6} with respect to the time variable, we obtain
\begin{align*}
0 &= \int_\Omega \bu_{tt}^N\cdot \bu_t^N \,\mathrm{d}x + \int_\Omega\frac{1}{2\epsilon} (v^N(t) - 1) v^N_t (t) + 2\epsilon \nabla v^N(t) \cdot \nabla v^N_t(t) \,\mathrm{d}x
\\
& \quad 
+ \int_\Omega \frac{b^\prime(v^N) }{\alpha}v^N_t \varphi^* (\beps(\alpha\bu^N)) + b(v^N) F^{-1} ( \beps(\alpha\bu^N)) : \beps(\bu^N_t ) \,\mathrm{d}x - \langle l_t, \bu^N\rangle - \langle l, \bu^N_t \rangle 
\\&\quad
+ \int_\Omega b(v^N) \big[ F^{-1}(\beps(\bu^N_t + \alpha\bu^N)) - F^{-1}(\beps(\alpha\bu^N)) \big] : \beps(\bu^N_t) \,\mathrm{d}x + \langle l_t, \bu^N \rangle
\\
&= \Big[ \int_\Omega \bu_{tt}^N\cdot \bu_t^N + F^{-1}(\beps(\bu^N_t + \alpha\bu^N)) : \beps(\bu^N_t ) \,\mathrm{d}x - \langle l, \bu^N_t\rangle\Big] 
\\
& \quad  + \Big[ \int_\Omega \frac{b^\prime(v^N) }{\alpha}v^N_t \varphi^* (\beps(\alpha\bu^N))  + \frac{1}{2\epsilon} (v^N(t) - 1) v^N_t (t) + 2\epsilon \nabla v^N(t) \cdot \nabla v^N_t(t) \,\mathrm{d}x\Big].
\end{align*}
The terms in the first set of square brackets on the right-hand side vanishes for a.e. \( t\in (0, T) \),  considering the elastodynamic equation with the test function \(\bu^N_t\). Hence the terms in the second set of square brackets must also vanish. However, this is exactly the left-hand side of (\ref{p:equ43}) and so we conclude the required result. 
\end{proof}

\begin{lemma}\label{p:lem11}
Let the assumptions of Lemma \ref{p:lem4} hold and let \((\bu^N, \bbT^N, v^N) \) be the solution triple constructed there. Suppose additionally that the compatibility condition (\ref{p:equ5}) holds.  There exists a constant \( C\), independent of \( N \), such that
\begin{align*}
\sup_{t\in [0, T]}\|v^N_t(t) \|_{1,2} \leq C. 
\end{align*}
\end{lemma}

\begin{proof}
Let \( t\in (0, T] \) and fix \( h > 0 \) sufficiently small such that \( t-h > 0 \). 
Subtracting (\ref{p:equ43}), evaluated at time \( t\), from the inequality
\begin{align*}
0 \leq \big[ \partial_v\mathcal{E}(\bu^N(t-h), v^N(t-h)) + \mathcal{H}^\prime(v^N(t-h)) \big](v^N_t(t)) ,
\end{align*}
which is a consequence of the proof of Proposition \ref{p:prop5}, we deduce that
\begin{align*}
0 & \leq \int_\Omega \Big[ \frac{b^\prime(v^N(t-h))}{\alpha}\varphi^*(\beps(\alpha\bu^N(t-h)) ) - \frac{b^\prime(v^N(t))}{\alpha}\varphi^*(\beps(\alpha\bu^N(t))) \Big] v^N_t(t) \,\mathrm{d}x
\\
&\quad
+ \frac{1}{2\epsilon} \int_\Omega \big[ (v^N(t-h) - 1) - (v^N(t) - 1) \big] v^N_t(t) \,\mathrm{d}x 
\\&\quad
+ 2\epsilon \int_\Omega \nabla \big( v^N(t-h) - v^N(t) \big) \cdot \nabla v^N(t) \,\mathrm{d}x.
\end{align*}
Using the fundamental theorem of calculus in the first integral on the right-hand side and dividing through by \( h \), it follows that
\begin{equation}\label{p:equ44}
\begin{aligned}
&\frac{1}{2\epsilon} \int_\Omega \frac{v^N(t) - v^N(t-h) }{h} v^N_t(t) \,\mathrm{d}x + 2\epsilon\int_\Omega \nabla \Big( \frac{v^N(t) - v^N(t-h)}{h}\Big) \cdot \nabla v^N_t(t) \,\mathrm{d}x
\\
&\quad
+ \frac{2}{\alpha}\int_\Omega \frac{v^N(t) - v^N(t-h)}{h}v^N_t(t) \varphi^*(\beps(\alpha\bu^N(t-h))) \,\mathrm{d}x
\\
& \leq 
\int_\Omega \int_0^1 \Big\{ \frac{b^\prime(v^N(t))}{\alpha}v^N_t(t) F^{-1}(\alpha \beps(s\bu^N(t-h) + (1-s) \bu^N(t))) 
\\&\quad
: \beps\Big( \frac{\bu^N(t-h) - \bu^N(t)}{h}\Big)\Big\}  \,\mathrm{d}s\,\mathrm{d}x.
\end{aligned}
\end{equation}
Recalling that \(\bu^N \in C^1([0, T]; V_N) \) and \( v^N \in W^{1,2}(0, T; H^1(\Omega)) \), taking the limit as \( h \rightarrow 0+ \) in (\ref{p:equ44}) yields the following, for a.e. \( t\in (0, T) \):
\begin{equation}\label{p:equ34}
\begin{aligned}
&\frac{1}{2\epsilon}\|v^N(t) \|_2^2 + 2\epsilon\|\nabla v^N_t(t) \|_2^2  + \frac{2}{\alpha}\int_\Omega | v^N_t(t) |^2 \varphi^*(\beps(\alpha\bu^N(t))) \,\mathrm{d}x
\\
&\leq \int_\Omega \frac{b^\prime(1)}{\alpha}|v^N_t(t) ||F^{-1}(\beps(\alpha\bu^N(t)))||\beps(\bu^N_t(t)) |\,\mathrm{d}x
\\
&= \int_\Omega \frac{b^\prime(1)}{\alpha}|v^N_t(t) ||\beps(\alpha\bu^N(t))|^{p^\prime-1}|\beps(\bu^N_t(t)) |\,\mathrm{d}x
\\
&= \int_\Omega \frac{b^\prime(1)}{\alpha}\Big( |v^N_t(t) ||\beps(\alpha\bu^N(t))|^{\frac{p^\prime}{2}} \Big) \cdot \Big( |\beps(\alpha\bu^N(t))|^{\frac{p^\prime-2}{2}} |\beps(\bu^N_t(t)) |\Big) \,\mathrm{d}x
\end{aligned}
\end{equation}
Applying H\"{o}lder's inequality to the final line of (\ref{p:equ34}) and absorbing one of the resulting terms into the left-hand side,  we deduce that 
\begin{align*}
&\frac{1}{2\epsilon}\|v^N(t) \|_2^2 + 2\epsilon\|\nabla v^N_t(t) \|_2^2 + \frac{1}{\alpha} \int_\Omega |v^N_t(t) |^2 \varphi^*(\beps(\alpha\bu^N(t))) \,\mathrm{d}x
\\
&\leq C\int_\Omega |\beps(\bu^N(t))|^{p^\prime-2} |\beps(\bu^N_t(t))|^{2}\,\mathrm{d}x.
\end{align*}
If \( p = 2\), we are done by the bound in Lemma \ref{p:lem8}. Otherwise, applying H\"{o}lder's inequality again but  with parameters \( q = \frac{p^\prime}{p^\prime -2}\) and \( q^\prime  = \frac{p^\prime}{2}\), valid choices because  \( p^\prime > 2\),  we obtain
\begin{equation}\label{p:equ45}
\begin{aligned}
&\frac{1}{2\epsilon}\|v^N(t) \|_2^2 + 2\epsilon\|\nabla v^N_t(t) \|_2^2 + \frac{1}{\alpha} \int_\Omega |v^N_t(t) |^2 \varphi^*(\beps(\alpha\bu^N(t))) \,\mathrm{d}x
\\ &\leq C \|\beps(\bu^N(t)) \|_{p^\prime}^{p^\prime - 2} \|\beps(\bu^N_t(t)) \|_{p^\prime}^2.
\end{aligned}
\end{equation}
By Lemma \ref{p:lem8}, the right-hand side of (\ref{p:equ45}) is uniformly bounded above, independent of \( t\) and \( N \), for a.e. \( t\in (0, T) \). Hence the stated result follows.
\end{proof}

\begin{theorem}\label{p:thm4}
Let \( p \in (1,2]\). Suppose  that \( l \in C^1([0, T]; W^{-1,p}_D(\Omega)^d) \), \( \bu_0\), \( \bu_1 \in W^{1,p^\prime}(\Omega)^d\) and \( v_0 \in H^1_{D+1}(\Omega) \) with \( v_0 \in [0, 1]\) a.e. in \(\Omega\), such that the compatibility condition (\ref{p:equ5}) holds. 
There exists a triple \((\bu, \bbT, v) \) that is a weak energy solution of   (\ref{p:equ1})--(\ref{p:equ3}) in the sense of Theorem \ref{p:thm1}. Furthermore, there exists a subsequence in \( N \), not relabelled, such that if \((\bu^N,\bbT^N, v^N) \) is the solution triple from Lemma \ref{p:lem4}, we have 
\begin{itemize}
\item \(\bu^N\rightharpoonup\bu\) weakly in \( W^{2,2}(0, T; L^2(\Omega)^d) \) and weakly-* in \( W^{1,\infty}(0, T; W^{1,p^\prime}(\Omega)^d) \),
\item \( \bbT^N \overset{\ast}{\rightharpoonup} \bbT\) weakly-* in \( L^\infty(0, T; L^p(\Omega)^{d\times d}) \), and
\item \( v^N\overset{\ast}{\rightharpoonup} v\) weakly-* in \( W^{1,\infty}(0, T; H^1(\Omega))\).
\end{itemize}
\end{theorem}
\begin{proof}
The convergence results follow easily from Lemma \ref{p:lem7}, Lemma \ref{p:lem8}, Lemma \ref{p:lem9}, Lemma \ref{p:lem11} and standard compactness results for Bochner spaces. By the Aubin--Lions lemma, we see that  \( \bu^N\rightarrow\bu\) strongly in \(C^1([0, T]; L^{p^\prime}(\Omega)^d) \), and \(v^N\rightarrow v\) strongly in \( C([0, T]; L^2(\Omega)) \), up to a possible further subsequence that we do not relabel. 

From the elastodynamic equation (\ref{p:equ18}), for every \( \bw \in C^\infty_D(\overline{\Omega})^d\) and every \( \psi  \in C([0, T]) \), we have
\begin{align*}
\int_Q \bu_{tt}^N\cdot \big( \psi P_N\bw  \big) + b(v^N) \bbT^N:\beps(\psi P_N\bw ) \,\mathrm{d}x\,\mathrm{d}t = \int_0^T  \langle l, \psi  P_N\bw\rangle\,\mathrm{d}t.
\end{align*}
Using the weak convergence results and the strong convergence of \((P_N\bw )_N \) in   in \( W^{k,2}_D(\Omega)^d\) to \( \bw \) as \( N\rightarrow\infty \), it follows that
\begin{align*}
\int_Q \bu_{tt}\cdot\big( \psi \bw  \big) + b(v) \bbT:\beps(\psi \bw) \,\mathrm{d}x\,\mathrm{d}t = \int_0^T \psi \langle l, \bw \rangle\,\mathrm{d}t.
\end{align*}
By Lebesgue's differentiation theorem and a standard density argument, the required elastodynamic equation follows for any test function \( \bw \in W^{1,p^\prime}_D(\Omega)^d\). 
The satisfaction of the  initial conditions  follows from the usual arguments, using the approximation properties  of the data \((\bu_{0, N})_N \), \((\bu_{1,N})_N\) and \( (v^N_0)_N\). 

To show that the constitutive relationship holds, we use a variant of Minty's method. First, we notice that
\begin{align*}
&\lim_{N\rightarrow\infty}\int_Q b(v^N) \bbT^N:\beps(\bu^N_t + \alpha\bu^N) \,\mathrm{d}x\,\mathrm{d}t 
\\& = \lim_{N\rightarrow\infty}\int_0^T \Big\{ \langle l,\bu^N_t + \alpha\bu^N\rangle -\int_\Omega \bu_{tt}^N\cdot (\bu^N_t +\alpha\bu^N) \,\mathrm{d}x\Big\}\,\mathrm{d}t
\\
&= \int_0^T \Big\{ \langle l,\bu_t + \alpha\bu\rangle - \int_\Omega \bu_{tt}\cdot (\bu_t + \alpha\bu) \,\mathrm{d}x\Big\} \,\mathrm{d}t
\\
&= \int_Q b(v) \bbT:\beps(\bu_t + \alpha\bu) \,\mathrm{d}x\,\mathrm{d}t,
\end{align*}
using the strong convergence results in the transition to the second line. For an arbitrary test function \( \bbS\in L^p(Q)^{d\times d}\), it follows that
\begin{align*}
0 &\leq \lim_{N\rightarrow\infty} \int_Q b(v^N) (\bbT^N - \bbS) :(F(\bbT^N) - F(\bbS))\,\mathrm{d}x\,\mathrm{d}t
\\
&= \int_Q b(v) (\bbT - \bbS) :(\beps(\bu_t + \alpha\bu) - F(\bbS) )\,\mathrm{d}x\,\mathrm{d}t.
\end{align*}
Replacing \( \bbS \) with \(  \bbT \pm \gamma \bbU \) where \( \gamma>0 \) and \( \bbU \in L^\infty(Q)^{d\times d}\), we get
\begin{align*}
0 \leq \mp \int_Q b(v) \bbU :(\beps(\bu_t + \alpha\bu) - F(\bbT))\,\mathrm{d}x\,\mathrm{d}t, 
\end{align*}
taking the limit as \( \gamma\rightarrow 0+ \). 
Choosing \( \bbU = \frac{\beps(\bu_t + \alpha\bu) - F(\bbT)}{1 + |\beps(\bu_t + \alpha\bu) - F(\bbT)|}\) and recalling that \( b \) is strictly positive, it follows that \( \beps(\bu_t + \alpha\bu) = F(\bbT) \) pointwise a.e. in \( Q\). Also, we deduce that 
\begin{equation}\label{p:equ46}
\lim_{N\rightarrow\infty}\int_Q b(v^N) |\bbT^N|^p \,\mathrm{d}x\,\mathrm{d}t = \int_Q b(v) |\bbT|^p\,\mathrm{d}x\,\mathrm{d}t.
\end{equation}
However, the sequence \( (b(v^N)^{\frac{1}{p}}\bbT^N)_N\) converges weakly in \( L^p(Q)^{d\times d}\). The weak convergence with convergence of norms (\ref{p:equ46}) implies that \( (b(v^N)^{\frac{1}{p}}\bbT^N)_N\) converges strongly in \( L^p(Q)^{d\times d}\). Thus \( \bbT^N \rightarrow\bbT \) strongly in \( L^p(Q)^{d\times d}\), which in turn implies that \( (\beps(\bu^N_t + \alpha\bu^N))_N \)  converges strongly in \( L^{p^\prime}(Q)^{d\times d}\) to \( \beps(\bu_t + \alpha\bu) \). Reasoning as before, it follows that \((\beps(\bu^N))_N\) converges strongly to \( \beps(\bu ) \) in \( L^{\infty}(0, T; L^{p^\prime}(\Omega)^d) \), and so, we deduce the pointwise convergence of \((\beps(\bu^N))_N \) a.e. in \( Q\), up to a further subsequence in \( N \).  Pointwise convergence and strong convergence in \( L^{p^\prime}(Q)\) of \((\beps(\bu^N_t))_N\) to \( \beps(\bu_t) \) then follows. Applying Fatou's lemma with the weak lower semi-continuity of norms to the energy-dissipation equality for \((\bu^N, \bbT^N, v^N) \), we get
\begin{align*}
&\mathcal{F}(t;\bu(t), \bu_t(t) , v(t)) + \int_0^t \int_\Omega b(v) \big[ \bbT - F^{-1}(\beps(\alpha\bu)) \big]:\beps(\bu_t) \,\mathrm{d}x\,\mathrm{d}s + \int_0^t \langle l_t(s), \bu(s) \rangle\,\mathrm{d}s
\\
&\leq \liminf_{N\rightarrow\infty} \Big\{ \mathcal{F}(t;\bu^N(t), \bu^N_t(t), v^N(t)) + \int_0^t \int_\Omega b(v^N) \big[ \bbT^N - F^{-1}(\beps(\alpha\bu^N))\big]:\beps(\bu_t^N) \,\mathrm{d}x\,\mathrm{d}s
\\
&\quad
+ \int_0^t \langle l_t(s),\bu^N(s) \rangle\,\mathrm{d}s\Big\}
\\
&= \liminf_{N\rightarrow\infty}\mathcal{F}(0;  \bu_{0,N},  \bu_{1,N}, v^N_0)
\\
&= \mathcal{F}(0;\bu_0,\bu_1, v_0).
\end{align*}
The final line follows from the choice of the approximations \((\bu_{0,N})_N \) and \((\bu_{1,N})_N \), as well as Proposition \ref{p:prop1}. 
For the opposite inequality, we  reason as in the proof of Proposition \ref{p:prop6}. To do so, we only need to prove that the limiting triple satisfies the minimisation problem (\ref{p:equ2}). For this, it suffices to show that
\begin{equation}\label{p:equ33}
0 \leq \big[ \partial_v \mathcal{E}(\bu(t), v(t)) + \mathcal{H}^\prime(v(t)) \big] (\chi),
\end{equation}
for every \( \chi \in H^1_D(\Omega) \). Using    the fact that \( \varphi^*(\beps(\alpha\bu^N)) \rightarrow\varphi^*(\beps(\alpha\bu)) \) strongly in \( L^1(Q)\), by the strong convergence result for \((\beps(\alpha\bu^N))_N \),  we obtain  (\ref{p:equ33}) easily from taking the limit in (\ref{p:equ22}). Hence we conclude that \((\bu,\bbT, v) \) is a solution in the required sense because we have the weak elastodynamic equation, the minimisation problem, the energy-dissipation equality and the satisfaction of the initial data. 
\end{proof}

\subsection{Limit in the case \( p \in (2,\infty) \)}\label{sec:plarge}
In this case, we have \( p^\prime\in (1,2) \). We cannot obtain a bound on \( (v^N_t)_N\) by repeating the reasoning from Section \ref{sec:Nbounds}. In particular, the proof of Lemma \ref{p:lem11} fails. Hence, at this stage \((v^N)_N \) is at best bounded in \( L^\infty(0, T; H^1(\Omega)) \) which implies that \((v^N)_N \) converges weakly-* in \( L^\infty(0, T; H^1(\Omega))\). However, we must find the limit of the nonlinear term \((b(v^N))_N \)   to take the limit in the elastodynamic equation (\ref{p:equ18}). Hence we need at least a pointwise convergence result for the sequence \((v^N)_N \). To obtain such a result, we make use of the monotonicity of these functions with respect to the time variable. Indeed the key is the following   result   from \cite{MR4064375}. 
We combine this result with the Rellich--Kondrachov compactness theorem in order to obtain a suitable strong convergence result as stated in Corollary \ref{p:cor2}. 
Next we prove a continuity result for the limiting function constructed in  Corollary \ref{p:cor2}. This is vital to prove one direction of the energy-dissipation inequality. 
Then we  prove the  existence result for parameter values \( p > 2\).


\begin{theorem}
Let \([a,b]\subset\mathbb{R}\) and \( v^m :[a,b]\rightarrow L^2(\Omega) \), \( m \in \mathbb{N}\), a sequence of functions such that
\( 
v^m(s) \leq v^m(t)\)  in \( \Omega\)  for every \( a\leq s\leq t\leq b\) and \( m\in \mathbb{N}\). 
Suppose also that there exists a constant \( C\), independent of \(m \), such that
\( 
\|v^m(t) \|_2 \leq C\)  for every \( t\in [a,b]\) and \( m\in \mathbb{N}\).
There exists a subsequence in \(m \), not relabelled, and a function \( v:[a,b]\rightarrow L^2(\Omega) \) such that, for every \( t\in [a,b]\), we have 
\(
v^m(t) \rightharpoonup v(t)\)   in \( L^2(\Omega) \) as \( m\rightarrow\infty\).
Moreover, \( \|v(t) \|_2 \leq C\) for every \( t\in [0, T]\) and 
\( 
v(s) \leq v(t)\)   in \( \Omega\),  for every \( a\leq s\leq t\leq b\).
\end{theorem}

\begin{corollary}\label{p:cor2}
Let \((v^m)_{m\geq 1}\) be a sequence of functions defined on  \([a,b]\) with values in \( H^1(\Omega) \) such that
\( 
v^m(s) \leq v^m(t) \)  in \( \Omega\) for every \( a\leq s\leq t\leq b\)  and \( m\in \mathbb{N}\). 
Suppose also  that there exists a constant \( C\) independent of \(m \) such that
\( 
\|v^m(t) \|_{1,2} \leq C\) for every \( t\in [a,b]\) and \( m\in \mathbb{N}\).
There exists a subsequence in \( m \), not relabelled, and a function \( v:[a,b]\rightarrow H^1(\Omega) \) such that, for every \( t\in [a,b]\), 
\( 
v^m(t) \rightharpoonup v(t) \) weakly in \(H^1(\Omega)\)  and strongly in \( L^2(\Omega) \). 
Furthermore, \( \|v(t) \|_{1,2}\leq C \) for every \( t\in [a,b]\) and 
\(
v(s) \leq v(t)\)   in \( \Omega\) for every \( a\leq s\leq t\leq b\). 
\end{corollary}

\begin{proposition}\label{p:prop7}
Let \((v^m)_m \) be the subsequence constructed in Corollary \ref{p:cor2}, with limiting function \( v\). Then \( v\) is almost everywhere continuous as a map from \([a,b]\) into \( L^2(\Omega) \). 
\end{proposition}

\begin{proof}
Fix a time \( t\in (a,b) \) and, for \( i \in \{1,2\}\), let \((\delta_{n,i})_{n\geq 1}\) be a sequence in \((0, \infty) \) such that \( t + \delta_{n,i}\in (a,b]\) for every \( n \) and \( \delta_{n,i}\rightarrow 0 + \) as \( n \rightarrow\infty\). The sequence \((v(t + \delta_{n,i}))_{n}\) is bounded in \( H^1(\Omega) \). Hence, up to a subsequence that we do not relabel, we have that \( (v(t + \delta_{n,i}) )_n \) converges weakly in \( H^1(\Omega) \) and strongly in \( L^2(\Omega) \) to a limit that we denote by \( \tilde{w}_i \). Using the strong convergence result and the monotonicity assumption, for every \( n \), we have
\begin{align*}
\|v(t + \delta_{n,1}) \|_2 \geq \lim_{m\rightarrow\infty} \|v(t + \delta_{m,2})\|_2 = \|\tilde{w}_2\|_2. 
\end{align*}
Letting \( n \rightarrow\infty\), we deduce that \( \|\tilde{w}_1\|_2 \geq \|\tilde{w}_2\|_2\). By symmetry, we deduce that \( \|\tilde{w}_1\|_2 = \|\tilde{w}_2\|_2\). We conclude that every subsequence of \((v(t + \delta))_{\delta> 0}\) converges to the same limit as \( \delta\rightarrow 0+\). Hence the upper limit \(v(t+) \) exists for every \( t\in (a,b) \). By analogous reasoning, the lower limit \(v(t-) \) also exists. 

By monotonicity, the function \( v\) is discontinuous at the point \( t\) in the case that the upper and lower limits do not coincide. Indeed, we may write
\begin{align*}
\{t\in (a,b) \,:\, v\text{ is not continuous at }t\} &= \cup_{\epsilon>0} \{ t\in (a,b) \,:\, \|v(t+) - v(t-) \|_2 > \epsilon\}
\\
&= \cup_{n\geq 1}\Big\{ t\in (a,b) \,:\, \|v(t+ ) -v(t-)\|_2 > \frac{1}{n}\Big\}. 
\end{align*}
By monotonicity, we know that
\begin{align*}
\|v(a) \|_2 \leq \|v(t-) \|_2 \leq \|v(t+) \|_2 \leq \|v(b) \|_2. 
\end{align*}
Hence, for every \( n\), the set \( \{ t\in (a,b) \,:\, \|v(t+ ) -v(t-)\|_2 > \frac{1}{n}\}\) must be finite. Hence the set of discontinuity points is countable and consequently \( v\) is continuous almost everywhere on \((a,b) \) as a map into \( L^2(\Omega) \). 
\end{proof}


\begin{theorem}\label{p:thm6}
Let \( p \in (2,\infty)\). Suppose  \( l \in C^1([0, T]; W^{-1,p}_D(\Omega)^d) \), \( \bu_0\), \( \bu_1 \in W^{1,p^\prime}_D(\Omega)^d \cap L^2(\Omega)^d \) and \( v_0 \in H^1_{D+1}(\Omega ) \) with \( v_0 \in [0, 1]\) a.e. in \( \Omega\),  such that the compatibility condition (\ref{p:equ5}) holds.  
There exists a weak energy solution  \((\bu, \bbT, v) \) of (\ref{p:equ1})--(\ref{p:equ3}) in the sense of  Theorem \ref{p:thm2}. Furthermore, there exists a subsequence in \( N \), not relabelled, such that if \((\bu^N, \bbT^N , v^N) \) is the solution triple from Lemma \ref{p:lem4},  the following convergence results hold:
\begin{itemize}
\item \( \bu^N \rightharpoonup \bu\) weakly in \(W^{2,2}(0,T; L^2(\Omega)^d) \) and weakly-* in \( W^{1,\infty}(0, T; W^{1,p^\prime}(\Omega)^d) \);
\item \( \bbT^N \overset{\ast}{\rightharpoonup}\bbT \) weakly-* in \( L^\infty(0, T; L^p(\Omega)^{d\times d}) \);
\item \( v^N(t) \rightarrow v(t) \) strongly in \( L^2(\Omega) \) and weakly in \( H^1(\Omega) \) for every \( t\in [0, T]\).
\end{itemize}
\end{theorem}

\begin{proof}
Using the bounds in Section \ref{sec:Nbounds} with Corollary \ref{p:cor2}, the convergence results are immediate. 
Considering the initial data for the approximate solution and the corresponding convergence results, the claim for the initial conditions follows. 
Combining  the strong convergence result from Corollary \ref{p:cor2} with the uniform bound on \( (v^N)_N \) and  Lebesgue's dominated convergence theorem,  we deduce that \( v^N\rightarrow v\) strongly in \( L^q(Q)^{d\times d}\), for every \( q\in [1,\infty) \). 
Hence, reasoning as before,   the elastodynamic equation must be satisfied by the limiting triple. 
Using Minty's method, it follows that the constitutive relation holds. The strong convergence of \(( \beps(\bu^N))_N \) in \( L^\infty(0, T; L^{p^\prime}(\Omega)^{d\times d}) \) and of \( (\beps(\bu^N_t))_N\) in \( L^{p^\prime}(Q)^{d\times d}\) follows. From this, we get the ``\(\leq \)'' direction of the energy-dissipation equality. Satisfaction of the minimisation problem can be deduced exactly as in the proof of Theorem \ref{p:thm4}. 

For the opposite  direction ``\(\geq \)'' of the energy-dissipation equality, we can repeat the reasoning from the proof of Proposition \ref{p:prop6} and use the continuity that follows from Proposition \ref{p:prop7}. Indeed, fix a point \( t\in (0, T]\) such that it is a point of continuity of \( v\). Let \( \mathcal{N}\) denote a set of null measure such that \( v\) is continuous on \([0, T]\setminus \mathcal{N}\). 

 Fix \( K \in \mathbb{N}_{\geq 2} \) and define a time step \(h_K = \frac{t}{K} \). For each \( 0 \leq k \leq K-1\), we fix a time point close to \( hk \) at which \( v\) is continuous. Namely, we pick  a \( t^K_k \in (hk-\frac{h_K}{4}, hk+\frac{h_K}{4})\cap \big((0, T)\setminus \mathcal{N}\big) \) such that \( v\) is continuous at \( t^K_k \) as a map into \( L^2(\Omega) \). Next, we define a sequence of approximations
 \begin{align*}
\bu^{K}_k := \bu(t^K_k), \quad \bbT^K_k := \bbT(t^K_k) ,\quad v^K_k := v(t^K_k)\quad\text{ for }0 \leq k \leq K. 
 \end{align*}
 Then we have
 \begin{align*}
 \mathcal{E}(\bu(t), v(t)) - \mathcal{E}(\bu_0, v_0) = \sum_{k=1}^K \mathcal{E}(\bu^K_k, v^K_k ) - \mathcal{E}( \bu^K_{k-1}, v^K_{k-1}) . 
 \end{align*}
 Using the minimisation problem, we see that
 \begin{align*}
& \mathcal{E}(\bu^K_k, v^K_k ) - \mathcal{E}( \bu^K_{k-1}, v^K_{k-1}) 
 \\&=
  \mathcal{E}(\bu^K_k, v^K_k ) - \mathcal{E}( \bu^K_{k-1}, v_k^K) + \mathcal{E}( \bu^K_{k-1}, v_k^K)  - \mathcal{E}( \bu^K_{k-1}, v^K_{k-1}) 
  \\
  &\geq \int_\Omega\frac{b(v^K_k) }{\alpha} \Big( \varphi^*(\beps(\alpha\bu^K_k)) - \varphi^*_n(\beps(\alpha\bu^K_{k-1})) \Big)\,\mathrm{d}x + \mathcal{H}(v^K_{k-1}) - \mathcal{H}(v^K_{k})
  \\
  &= 
  h_K\int_\Omega\int_0^1 \frac{b(v^K_k)}{\alpha} F^{-1}(\alpha\beps(s\bu^K_k + (1-s) \bu^K_{k-1})):\beps(\delta\bu^{K}_k)\,\mathrm{d}s\,\mathrm{d}x+ \mathcal{H}(v^K_{k-1}) - \mathcal{H}(v^K_{k}).
 \end{align*}
 As a result, we obtain the inequality
 \begin{align*}
 & \mathcal{E}(\bu(t), v(t)) - \mathcal{E}(\bu_0, v_0)  
  \\&
  \geq \mathcal{H}(v_0) - \mathcal{H}(v(t)) +  h_K\sum_{k=1}^K \int_\Omega\int_0^1 \frac{b(v^K_k)}{\alpha} F^{-1}(\alpha\beps(s\bu^K_k + (1-s) \bu^K_{k-1})):\beps(\delta\bu^{K}_k)\,\mathrm{d}s\,\mathrm{d}x. 
 \end{align*}
  For the sum on the right-hand side, we rewrite it in the following form:
 \begin{align*}
 & h_K\sum_{k=1}^K \int_\Omega\int_0^1 \frac{b(v^K_k)}{\alpha} F^{-1}(\alpha\beps(s\bu^K_k + (1-s) \bu^K_{k-1})):\beps(\delta\bu^{K}_k)\,\mathrm{d}s\,\mathrm{d}x
  \\
  &= \int_0^t \int_\Omega \int_0^1 \frac{b(v^{K,+})}{\alpha}F^{-1}( \alpha\beps(s\bu^{K,+} + (1-s) \bu^{K,-})) : \beps(\overline{\bu}^K_t) \,\mathrm{d}s\,\mathrm{d}x\,\mathrm{d}\tau
  \\
  &=  \int_{(0, t)\setminus \mathcal{N}} \int_\Omega \int_0^1 \frac{b(v^{K,+})}{\alpha}F^{-1}( \alpha\beps(s\bu^{K,+} + (1-s) \bu^{K,-})) : \beps(\overline{\bu}^K_t) \,\mathrm{d}s\,\mathrm{d}x\,\mathrm{d}\tau. 
 \end{align*}
 Considering the limit as \( K \rightarrow\infty\), we have
 \begin{align*}
 &b(v^K_+) \rightarrow b(v) \quad\text{ strongly in }L^q\big( \big( (0,t)\setminus \mathcal{N}\big) \times\Omega\big) \text{ for every }q\in [1,\infty),
 \\
 &\int_0^1 F^{-1}( \alpha\beps(s\bu^{K,+} + (1-s) \bu^{K,-}))\,\mathrm{d}s \rightarrow  F^{-1}(\alpha\beps(\bu)) \quad\text{  strongly in } L^\infty(0, t;  L^{p}(\Omega)^{d\times d}).
 \end{align*}
 By the uniform boundedness of \( v\) and the fact that \( \mathcal{N}\) is null, combining these convergences, we deduce that
 \begin{align*}
 b(v^K_+)\int_0^1 F^{-1}( \alpha\beps(s\bu^{K,+} + (1-s) \bu^{K,-}))\,\mathrm{d}s \rightarrow b(v) F^{-1}(\alpha\beps(\bu)) \,\text{  strongly in } L^\infty(0, t;  L^{p}(\Omega)^{d\times d}).
 \end{align*}
Suppose that  \( (\beps(\overline{\bu}^K_t))_K\) converges weakly in \( L^p((0, t) \times \Omega)^{d\times d}\) to \( \beps(\bu_t) \). Then we conclude that
 \begin{align*}
 &\lim_{K\rightarrow\infty} \int_{(0, t)\setminus \mathcal{N}} \int_\Omega \int_0^1 \frac{b(v^{K,+})}{\alpha}F^{-1}( \alpha\beps(s\bu^{K,+} + (1-s) \bu^{K,-})) : \beps(\overline{\bu}^K_t) \,\mathrm{d}s\,\mathrm{d}x\,\mathrm{d}\tau
 \\
 &= \int_0^t \int_\Omega b(v) F^{-1}(\beps(\alpha\bu)) :\beps(\bu_t) \,\mathrm{d}x\,\mathrm{d}s . 
 \end{align*}
 It follows that
 \begin{align*}
 \mathcal{E}(\bu(t), v(t)) + \mathcal{H}(v(t)) \geq \mathcal{E}(\bu_0, v_0) + \mathcal{H}(v_0) + \int_0^t \int_\Omega \frac{b(v) }{\alpha} F^{-1}(\beps(\alpha\bu)): \beps(\bu_t) \,\mathrm{d}x\,\mathrm{d}s.
 \end{align*}
 Reasoning as in Proposition \ref{p:prop6}, this is sufficient to show that the energy-dissipation equality holds. 
 It remains to prove the weak convergence result for \( (\beps(\overline{\bu}^K_t))_K\). We have
 \begin{align*}
 \beps(\delta\bu^K_k) &= \frac{1}{h_K} \beps(\bu^K_k - \bu^K_{k-1})
 \\
 &= \frac{1}{h_K}\beps(\bu(t_k^K) - \bu(t^K_{k-1}))
 \\
 &= \frac{t_k^K - t^K_{k-1}}{h_K}\fint_{t^K_{k-1}}^{t^{K}_k} \beps(\bu_t(s)) \,\mathrm{d}s,
 \end{align*}
where \(\fint\) denotes the usual average integral operator. Noting that the difference \(\frac{t_k^K - t^K_{k-1}}{h_K}\) is uniformly bounded with respect to \(K \) and \( k \), as a result of the construction of the points \((t_k^K)_k\), we deduce that
\begin{align*}
\|\beps(\delta\bu^K_k ) \|_{p^\prime} \leq C \|\beps(\bu_t) \|_{L^\infty(L^{p^\prime})}.
\end{align*}
This implies that
\[
\|\beps(\overline{\bu}^K_t) \|_{L^\infty(L^{p^\prime}) }\leq C \|\beps(\bu_t) \|_{L^\infty(L^{p^\prime})}.
\]
By the weak-* compactness of bounded sets in \( L^\infty(0, T; L^{p^\prime}(\Omega)^{d\times d})\), the required convergence result follows. 
\end{proof}

\section{Strain-limiting fracture problems}\label{sec:a}  
In this section, we consider a strain-limiting dynamic fracture problem, by which we mean that the strain is bounded {\it a priori}. Here, we present the state of the art of this paper.  Taking \( p \rightarrow 1+ \) in the constitutive relation from Section \ref{sec:p}, we obtain the relationship
\begin{align*}
\beps(\bu_t + \alpha\bu) = \frac{\bbT}{|\bbT|} =: F_1(\bbT).
\end{align*}
However, the function \( F_1\) is not an injection, nor is it continuous at the origin. To avoid the issue at the origin, and to also have an injective function in the constitutive relation, we regularise the denominator. This motivates the study of the relationship (\ref{intro:equ1}) for a fixed parameter \( a>0 \), that is,
\begin{align*}
\beps(\bu_t + \alpha\bu) = \frac{\bbT}{(1 + |\bbT|^a)^{\frac{1}{a}}} =:F(\bbT).
\end{align*}
The parameter \( a\) is fixed throughout this section so we do not indicate the dependence of \( F\) on \( a\). We prove an existence result for a  dynamic fracture problem with a phase-field approximation and a strain-limiting relation, independent of the value of \( a\). 

As discussed previously, the main issue that we need to overcome in the analysis is the lack of integrability of the stress tensor \( \bbT\). Indeed, the approximations of the stress tensor are {\it a priori} bounded  in the non-reflexive space \( L^1(Q)^{d\times d}\) and so are weakly-* compact only in the space of Radon measures on \( \overline{Q}\), but not in any Lebesgue function space. We need to use careful arguments to prove that the sequence of approximations converge in some sense to a Lebesgue function. Then, we need to prove that the convergence is sufficiently strong   so that we can show that the  elastodynamic equation and   the constitutive relation both hold in the limit. 

To do this, we require  additional regularity of the phase-field function \( v\).  
To get improved spatial regularity estimates on the approximations of \( \bbT \), we need \( \nabla v\) and its approximations to be uniformly bounded on \( Q\). 
From the construction of the problem, we expect the gradients to have absolute value to the order of \( \epsilon^{-1}\). Indeed,  the parameter \( \epsilon\) gives some sense of the `thickness' of the approximation of the crack set. However, 
proving a uniform bound in practice does not seem possible. Hence, in the spirit of \cite{MR4064375},  we introduce an extra term in the minimisation problem (\ref{p:equ2}) in order to guarantee this \( L^\infty(Q)^d\) regularity of \(\nabla v \). The new term imposes a dependence of the problem on the speed at which the crack grows. It is referred to as a rate-dependent term because of the presence of the time derivative \( v_t\). 

In the strain-limiting setting, we have an added issue in the case that the Neumann part of the boundary is non-empty. As in the steady problem without  fracture  \cite{RN6}, we experience a penalisation on this part of the boundary in the weak elastodynamic equation due to a lack of global integrability of approximations of the stress tensor. 
Indeed, we prove that the approximations of the stress tensor converge in \( L^\infty(0, T; \mathcal{M}( \overline{\Omega})^{d\times d})\) and pointwise a.e. on \( Q\) to limits \( \overline{\bbT}\) and \( \bbT\), respectively. Due to the choice of test function, any singularities on the Dirichlet part of the boundary do not appear, because the test functions vanish on the whole of \( \Gamma_D\). However, we cannot control any difference between \( \overline{\bbT}\) and \( \bbT \) on the Neumann part of the boundary, because the test functions do not vanish there. This is the reason for the presence of the penalty term on \( \Gamma_N \). 
This is also the reason for the existence result being stronger in the case of fully Dirichlet boundary conditions; it suffices to obtain local convergence with respect to the spatial domain \( \Omega\) in order to prove the required results. 
As a result of the error term in the mixed boundary condition case, we not only have a penalisation in the elastodynamic equation, but we are unable to deduce an energy-dissipation equality. Even if we assume {\it a priori} that the error term vanishes, the convergence results are not sufficiently strong  to show that equality holds.  

The problem of interest in this section is the following. For a fixed   \(k \in \mathbb{N}\), we look for  a triple \((\bu, \bbT, v) : Q \rightarrow\mathbb{R}^d \times \mathbb{R}^{d\times d}\times\mathbb{R}\) that solves
\begin{equation}\label{a:equ1}
\begin{aligned}
\bu_{tt}&= \diver\big(b(v) \bbT \big)  + \boldf &\quad& \text{ in }Q,
\\
\beps(\bu_t + \alpha\bu) &= F(\bbT) = \frac{\bbT}{(1 + |\bbT|^a)^{\frac{1}{a}}}&\quad& \text{ in }Q,
\\
\bu\big|_{\Gamma_D} &= 0, \, v\big|_{\Gamma_D} = 1&\quad& \text{ on }[0, T]\times \Gamma_D,
\\
b(v) \bbT \mathbf{n} &= \bg&\quad& \text{ on }[0, T]\times \Gamma_N,
\\
\bu(0, \cdot)&= \bu_0, \bu_t(0, \cdot) = \bu_1,\, v(0, \cdot) = v_0&\quad&\text{ in }\Omega,
\end{aligned}
\end{equation}
subject to the minimisation problem
\begin{equation}\label{a:equ2}
\mathcal{E}(\bu(t), v(t)) + \mathcal{H}(v(t)) + \mathcal{G}_k(v(t), v_t(t)) = \inf_{v\in H^k_{D+1}(\Omega),\, v\leq v(t)} \Big\{ \mathcal{E}(\bu(t), v) + \mathcal{H}(v) + \mathcal{G}_k(v, v_t(t)) \Big\},
\end{equation}
with energy-dissipation equality
\begin{equation}\label{a:equ3}
\begin{aligned}
&\mathcal{F}(t;\bu(t), \bu_t(t), v(t)) + \int_0^t \int_\Omega b(v) \big( \bbT - F^{-1}(\beps(\alpha\bu))\big]:\beps(\bu_t)  + \boldf_t\cdot \bu \,\mathrm{d}x\,\mathrm{d}s
\\
&\quad 
+ \int_0^t \int_{\Gamma_N} \bg_t \cdot \bu\,\mathrm{d}S\,\mathrm{d}s + \int_0^t \|v_t(s)\|_{k,2}\,\mathrm{d}s
\\
&= \mathcal{F}(0; \bu_0, \bu_1, v_0) ,
\end{aligned}
\end{equation}
for every \( t\in [0, T]\) and the crack non-healing property \( v_t\leq 0\). All  functionals are defined as in Section \ref{sec:p} but with  \( F\) given by (\ref{a:equ1}).  The functional \( \mathcal{G}_k \) is defined on \( H^k(\Omega) \times H^k(\Omega) \) by
\begin{align*}
\mathcal{G}_k(v, w)  = (v,w)_{k,2} := \sum_{|\alpha|\leq k} \int_\Omega\partial^\alpha v\cdot \partial^\alpha w\,\mathrm{d}x.
\end{align*}
Due to the presence of the rate-dependent term,  we cannot guarantee that \( v\geq 0\). 
For a discussion regarding the physical implications of the inclusion of this functional on the  model, we refer to \cite{MR4064375}. 

To ensure that the nonlinearity \(b(v) \) is still uniformly bounded from above, 
in this section we define \( b \) by
\begin{align*}
b(v) = \max\{ 0, v\}^2 + \eta.
\end{align*}
We later see  that  the phase-field function and its approximations are uniformly bounded in the function space \( L^\infty(0, T; H^k(\Omega)^d) \). Hence, if \( k > \frac{d}{2}\), by the Sobolev embedding theorem, we deduce that the functions are uniformly bounded on \( Q\). As a result, in this case, which is  a standing assumption of the   result proven here, the proof  is actually valid for the choice of \( b\) that is used in Section \ref{sec:p}. 

We do not require a compatibility condition of type (\ref{p:equ5}) because \( v_t(0) \) is not defined {\it a priori}. 
Indeed, the minimisation problem does not make sense at the initial time. 
However, we  require a compatibility condition between the Neumann boundary data and the initial data. We demand that
\begin{equation}\label{a:equ4}
\bg(0) = b(v_0) F^{-1}(\beps(\bu_1 + \alpha\bu_0)) \mathbf{n}\quad\text{ on }\Gamma_N.
\end{equation}
This is vital in order to bound approximations of \(\bu_{tt}(0) \) which is needed to obtain higher regularity estimates in the temporal variable of the approximations of the stress tensor.  For (\ref{a:equ4}) to make sense for the given initial data, we also demand the safety strain condition, that is, we have
\begin{equation}\label{a:equ5}
\max\big\{ \|\beps(\alpha\bu_0) \|_\infty, \, \|\beps(\bu_1 + \alpha\bu_0) \|_\infty \big\} =  C_* < 1. 
\end{equation}
Such a condition is discussed in detail in \cite{RN6}. To ensure that the total energy functional is finite at the initial time, we require a bound on  \( \beps(\alpha\bu_0) \), as well as \( \beps( \bu_1 + \alpha\bu_0) \) in the safety strain condition. The bound on \( \beps(\alpha\bu_0) \) is non-standard in the literature. However, the function \( \varphi^*\) is not necessarily finite outside the open unit ball, which is why we include the extra assumption. 

As with the analysis of time-dependent, strain-limiting problems without a phase field function (see, for example, the analysis in \cite{mypaperpreprint} and \cite{preprint2}), we approximate (\ref{a:equ1})--(\ref{a:equ3}) by a regularised problem. The function \( F \) is replaced by a function \( F_n \) that contains an elliptic term.  As a result, \( F_n\) is a bijection from \( \mathbb{R}^{d\times d}\) to itself and so the problem can be written in terms of only \(\bu \) and \( v\). 
 We define  \( F_n \) by
\begin{align*}
F_n(\bbT ) = \frac{\bbT}{(1 + |\bbT|^a)^{\frac{1}{a}}} + \frac{\bbT}{n}. 
\end{align*}
with \( \varphi_n \), \(\varphi^*_n\) and \( \mathcal{E}_n\) defined analogously. We  approximate the body forces \(l \) in the regularised problem to ensure that a suitable compatibility condition analogous to (\ref{a:equ4}) holds in the setting of the regularised problem. We replace the function \( \bg\) by an approximation \( \bg^n \) that we define by
\begin{equation}\label{a:equ6}
\bg^n = F^{-1}_n(\beps(\bu_1 + \alpha\bu_0))\mathbf{n}\cdot \psi(nt) + \bg(t,x) \cdot (1 - \psi(nt)),
\end{equation}
where \( \psi \in C^\infty_c([0, \infty)) \) is a non-increasing function with \( \psi = 1 \) on \( [0, \frac{1}{2}]\) and \( \psi = 0 \) on \( [1,\infty) \). 
The external force approximation is defined by \( l^n:= \boldf + \bg^n \). 
We define the total energy \( \mathcal{F}_n \) by replacing \(\mathcal{E}\) with \(\mathcal{E}_n \) and \( l \) with \( l^n \) in the definition of the functional \( \mathcal{F}\). Under the assumption that \( l\in W^{2,1}([0, T]; W^{-1,2}_D(\Omega)^d) \), the sequence \( (l^n)_n \) is bounded in this space independent of \( n \) and converges weakly to \( l \) as \( n\rightarrow\infty \). 

\begin{remark}\label{a:remark_F}
As mentioned in Section \ref{sec:intro}, the results in this section can be generalised to a more generalised class of functions \( F \). Indeed, the required assumptions on \( F\) are as follows. There exist positive constants \( c_0 \), \( c_1\) and \(a \) such that, for every \( \bbT\), \( \bbS \in \mathbb{R}^{d\times d}\), 
\begin{align*}
( F( \bbT) - F(\bbS)) : (\bbT - \bbS) &\geq 0 ,
\\
F( \bbT) \cdot \bbT &\geq c_0 |\bbT | - c_1, 
\\
|F(\bbT) |&\leq c_1,
\end{align*}
that is, monotonicity, coercivity and boundedness, respectively. Furthermore, we assume that there exists a strictly convex function \( \varphi\in C^2( \mathbb{R}_+; \mathbb{R}_+) \) such that \( \varphi(0)  = \varphi^\prime(0) = 0\) and \( |\varphi^{\prime\prime}(s)|\leq c_1(1  + s)^{-1}\) for every \( s\in \mathbb{R}_+\), with \( \varphi \) being related to \( F \) by the representation
\[
F( \bbT) = \frac{\varphi^\prime(|\bbT|)}{|\bbT|}\bbT. 
\]
Finally, if we define the inner product \((\cdot, \cdot)_{\mathcal{A}(\bbT) }\), for \( \bbT \in \mathbb{R}^{d\times d}\), by
\begin{align*}
(\bbS, \bbU)_{\mathcal{A}( \bbT) } = \sum_{i,j,k,l=1}^d \frac{\partial F_{ij}(\bbT)}{\partial \bbT_{kl}} \bbS_{ij} \bbU_{kl},
\end{align*}
then we require the lower bound
\begin{align*}
( \bbS, \bbS)_{\mathcal{A}( \bbT)} \geq \frac{c_0 |\bbS|^2}{(1 + |\bbT|)^{ 1 + a}},
\end{align*}
for every \( \bbT\), \( \bbS\in \mathbb{R}^{d\times d}\). 
\end{remark}

The structure of this section is as follows. First, we focus on proving the existence of a weak energy solution to the regularised problem, using a discretisation in time similar to that which is used in Section \ref{sec:p}. Then we look for \( n\)-independent bounds on the sequence of solutions, particularly those which involve higher regularity estimates. 
At best, we have a bound on \( (\bbT^n)_n \) in \( L^\infty(0, T; L^1(\Omega)^{d\times d}) \), a function space that does not have nice compactness properties. Hence, to obtain a pointwise convergence result for \((\bbT^n)_n \),  we use   weighted regularity estimates.  Using this pointwise convergence, if \( \Gamma_N = \emptyset\),  we can then improve the convergence  to strong convergence in \( L^1(0, T; L^1_{loc}( \Omega)^{d\times d}) \), making use of the monotonicity of \( F\). If \( a\in (0, \frac{2}{d}) \), this is further improved to strong convergence in \( L^{1 + \delta}(0, T; L^{1 + \delta}_{loc}( \Omega)^{d\times d}) \) for every \( \delta>  0 \) such that \( a + \delta < \frac{2}{d}\). 

We remark that we do not use a Galerkin approximation in space to deduce the existence of solutions to the regularised problem. This is due to the choice of regularisation term, which ensures that we are in the Hilbert space setting. This setting ensures that  appropriate bounds are available which allow us to use the Browder--Minty theorem to prove existence of solutions to the time discrete problem, without also applying a Galerkin approximation. This is not the case when the constitutive function has arbitrary \(p\)-growth,  as in Section \ref{sec:p}.  

Let us now introduce the time-discrete approximation of the regularised problem. 
Although the existence of a solution to the regularised problem follows from the proof in Section \ref{sec:p},  we need the structure of the time discrete problem to obtain higher time regularity estimates. In particular, we need to overcome the fact that the second time derivative \( \bu^n_{tt}\) is not necessarily continuous in the time variable. We do this by working with the time discrete problem and ensuring that \( \delta^2\bu^\gamma_0 \) is suitably defined. 

\begin{theorem}\label{a:thm1}
Suppose that \( \bu_0\), \( \bu_1\in W^{1,2}_D(\Omega)^d\) and  \( l \in C^1([0, T]; W^{-1,2}_D(\Omega)^d) \). Let \( l_n \) be defined by (\ref{a:equ6}). Fix \( M \) and \( n\in \mathbb{N}\), let \( \gamma = (n,M) \)   and define a time step \( h = \frac{T}{M}\).   We define initialisations of the time discrete problem by
\begin{align*}
\bu^{\gamma}_0 = \bu_0, \quad \bu^{\gamma}_{-1} = \bu_0 - h\bu_1, \quad v^{\gamma}_0 = v_0.
\end{align*}
Furthermore, we denote \( l^{\gamma}_m = l(t^{M}_m) \) where \( t^M_m = mh\) for \( 0 \leq m \leq M \). 
Defining the solution sequence recursively, for every \( 1\leq m \leq M \) there exists a unique \( \bu^{\gamma}_m \in W^{1,2}_D(\Omega)^d\) such that
\begin{equation}\label{a:equ7}
\int_\Omega \delta^2 \bu^{\gamma}_m\cdot \bw + b(v^{\gamma}_{m-1}) F^{-1}_n(\beps(\delta\bu^\gamma_m + \alpha\bu^\gamma_m)) :\beps(\bw) \,\mathrm{d}x = \langle l^\gamma_m , \bw\rangle,
\end{equation}
for every \( \bw \in W^{1,2}_D(\Omega)^d\), 
and  a unique \( v^\gamma_m \in H^k_{D+1}(\Omega) \) such that
\begin{equation}\label{a:equ8}
\begin{aligned}
&\mathcal{E}_n(\bu^\gamma_m , v^\gamma_m) + \mathcal{H}(v^\gamma_m ) + \frac{1}{2h}\mathcal{G}_k(v^\gamma_m - v^\gamma_{m-1}, v^\gamma_m - v^\gamma_{m-1})
\\
&= \inf\Big\{ \mathcal{E}_n(\bu^\gamma_m , v) + \mathcal{H}(v) + \frac{1}{2h}\mathcal{G}_k(v - v^\gamma_{m-1}, v - v^{\gamma}_{m-1})\,:\, v\in H^k_{D+1}(\Omega),\, , v\leq v^\gamma_{m-1}\Big\}.
\end{aligned}
\end{equation}
\end{theorem}

The proof follows identical reasoning to that of Theorem \ref{p:thm3} so is not repeated here. Similarly, reasoning as with Proposition \ref{p:prop2}, Corollary \ref{p:cor1} and Proposition \ref{p:prop3}, we deduce the following four results concerning the minimisation problem and energy-dissipation inequality.

\begin{lemma}\label{a:lem1}
Let the assumptions of Theorem \ref{a:thm1} hold and let  \((\bu^\gamma_m, v^\gamma_m)_{m=1}^M \)  be the solution of the time discrete problem. For every \( 1\leq m \leq M\) and every \( \tilde{\chi} \in H^k_{D+1}(\Omega) \) with \( \tilde{\chi }\leq v^\gamma_{m-1}\), we have
\begin{align*}
0 & \leq \big[ \partial_v\mathcal{E}_n(\bu^\gamma_m, v^\gamma_m) + \mathcal{H}^\prime(v^\gamma_m) + \frac{1}{2h}\partial_v\mathcal{G}_k(v^\gamma_m - v^\gamma_{m-1}, v^\gamma_m - v^{\gamma}_{m-1}) \big] (\tilde{\chi} - v^\gamma_{m-1})
\\
&= \int_\Omega \Big\{ \frac{b^\prime(v^\gamma_m)}{\alpha}(\tilde{\chi } - v^\gamma_{m-1}) \varphi^*_n(\beps(\alpha\bu^\gamma_m)) + \frac{1}{2\epsilon}(v^\gamma_m - 1) (\tilde{\chi} - v^\gamma_{m-1}) 
\\
&\quad
+ 2\epsilon\nabla v^\gamma_m \cdot \nabla (\tilde{\chi} - v^\gamma_{m-1})\Big\} \,\mathrm{d}x 
+ \frac{1}{h}\big( v^\gamma_m - v^\gamma_{m-1}, \tilde{\chi} - v^\gamma_{m-1}\big)_{k,2}.
\end{align*}
\end{lemma}

\begin{corollary}\label{a:cor1}
For every \( 1\leq m \leq M\) and  every \( \chi \in H^k_D(\Omega) \) with \( \chi \leq 0 \), we have
\begin{align*}
0 \leq \big[ \partial_v\mathcal{E}_n(\bu^\gamma_m, v^\gamma_m) + \mathcal{H}^\prime(v^\gamma_m) + \frac{1}{2h}\partial_v\mathcal{G}_k(v^\gamma_m - v^\gamma_{m-1}, v^\gamma_m - v^{\gamma}_{m-1}) \big] (\chi).
\end{align*}
\end{corollary}

\begin{corollary}\label{a:cor2}
For every \( 1\leq m \leq M \), we have 
\begin{align*}
0 &= \big[ \partial_v\mathcal{E}_n(\bu^\gamma_m, v^\gamma_m) + \mathcal{H}^\prime(v^\gamma_m) + \frac{1}{2h}\partial_v\mathcal{G}_k(v^\gamma_m - v^\gamma_{m-1}, v^\gamma_m - v^{\gamma}_{m-1}) \big] (\delta v^\gamma_h)
\\
&= \int_\Omega \Big\{ \frac{b^\prime(v^\gamma_m)}{\alpha}\delta v^\gamma_m \varphi^*_n(\beps(\alpha\bu^\gamma_m)) + \frac{1}{2\epsilon}(v^\gamma_m - 1) \delta v^\gamma_m
+ 2\epsilon\nabla v^\gamma_m \cdot \nabla \delta v^\gamma_m\Big\} \,\mathrm{d}x 
\\&\quad
+ \|\delta v^\gamma_m\|_{k,2}^2.
\end{align*}
\end{corollary}

\begin{proposition}
Let the assumptions of Theorem \ref{a:thm1} hold and let \((\bu^\gamma_m,v^\gamma_m)_{m=1}^M \) be  the solution of the time discrete problem. For every \( 1\leq m\leq M\), 
\begin{align*}
&\mathcal{F}_n(t^M_m; \bu^\gamma_m,\delta\bu^\gamma_m, v^\gamma_m) + h \sum_{j=1}^m \|\delta v^\gamma_j \|_{k,2}^2 + h\sum_{j=1}^m \langle \delta l^\gamma_m, \bu^\gamma_{m-1}\rangle
\\
&\quad
+ h \sum_{j=1}^m \int_\Omega b(v^\gamma_{j-1}) \big[ F^{-1}_n(\beps(\delta\bu^\gamma_n + \alpha\bu^\gamma_j)) - F^{-1}_n(\beps(\alpha\bu^\gamma_j)) \big] :\beps(\delta\bu^\gamma_j) \,\mathrm{d}x
\\
&\leq \mathcal{F}_n(0;\bu_0, \bu_1,v_0).
\end{align*}
\end{proposition}
Now we concentrate on obtaining various \( M \)-independent bounds on the discrete solution sequence. Then we take the limit as \( M \rightarrow\infty \) to obtain a weak energy solution of the regularised problem. The time continuous problem is more amenable to work with when it comes to higher spatial regularity estimates, so we cannot derive these estimates until we have taken the limit as \( M \rightarrow\infty \). 
To this end, using the fact that
\begin{align*}
(F^{-1}_n(\bbT) - F^{-1}_n(\bbS)): (\bbT - \bbS) \geq C(n) |\bbT - \bbS|^2,
\end{align*}
for every \( \bbT \), \(\bbS\in \mathbb{R}^{d\times d}\), where \( C= C(n) \) is a constant depending only on \( n \), we immediately obtain an \( M \)-independent bound from the energy-dissipation inequality. 

\begin{lemma}\label{a:lem2}
Let the assumptions of Theorem \ref{a:thm1} hold and let \((\bu^\gamma_m,v^\gamma_m)_{m=1}^M \) be the solution of the time discrete problem. There exists a constant \(C = C(n) \), independent of \( M \), such that 
\begin{align*}
\max_{1\leq m \leq M}\|\bu^\gamma_m \|_{1,2}  + \max_{1\leq m \leq M}\|\delta\bu^\gamma_m \|_{2}  + \max_{1\leq m \leq M}\|v^\gamma_m \|_{k,2}+ h\sum_{m=1}^M \Big( \|\delta\bu^\gamma_m \|_{1,2}^2 + \|\delta v^\gamma_m \|_{k,2}^2\Big)    \leq C.
\end{align*}
\end{lemma}

The next bound in Lemma \ref{a:lem3} gives us information on the second derivative \( \delta^2\bu^\gamma_m \). As a result, we are able to improve the bound on \( (\delta\bu^\gamma_m)_{m=1}^M \) also. In Lemma \ref{a:lem4}, we improve the bounds on the phase-field function. This gives sufficient bounds on the corresponding interpolant functions, which we collect in Corollary \ref{a:cor3}. Then we take the limit as \( M \rightarrow\infty\) in Proposition \ref{a:prop2} and identify the limiting couple as a weak energy solution of the regularised problem. 

\begin{lemma}\label{a:lem3}
Let the assumptions of Theorem \ref{a:thm1} hold and let \((\bu^\gamma_m,v^\gamma_m)_{m=1}^M \) be the solution of the time discrete problem. There exists a constant \(C = C(n) \) independent of \( M \) such that 
\begin{align*}
h \sum_{m=1}^M\|\delta^2 \bu^\gamma_m\|_{2}^2 + \max_{1\leq m \leq M}\|\delta\bu^\gamma_m \|_{1,2}\leq C.
\end{align*}
\end{lemma}
\begin{proof}
Testing in (\ref{a:equ7}) against \( h (\delta^2\bu^\gamma_m + \alpha\delta\bu^\gamma_m ) \), we rewrite the term involving the nonlinearity in the following way:
\begin{equation}\label{a:equ40}
\begin{aligned}
&\int_\Omega b(v^\gamma_{m-1}) F^{-1}_n(\beps(\delta\bu^\gamma_m + \alpha\bu^\gamma_m) ) : \beps(h (\delta^2\bu^\gamma_m + \alpha\delta\bu^\gamma_m)) \,\mathrm{d}x
\\
&= 
\int_\Omega b(v^\gamma_{m-1}) F^{-1}_n(\beps(\delta\bu^\gamma_m + \alpha\bu^\gamma_m)) : \beps( \delta\bu^\gamma_m + \alpha\bu^\gamma_m)\,\mathrm{d}x
\\
&\quad - \int_\Omega b(v^\gamma_{m-1}) F^{-1}_n(\beps(\delta\bu^\gamma_{m-1} + \alpha\bu^\gamma_{m-1})) : \beps( \delta\bu^\gamma_{m-1} + \alpha\bu^\gamma_{m-1})\,\mathrm{d}x
\\
&\quad
- \int_\Omega  \Big\{ b(v^\gamma_{m-1}) \big[ F^{-1}_n(\beps(\delta\bu^\gamma_m + \alpha\bu^\gamma_m)) - F^{-1}_n(\beps(\delta\bu^\gamma_{m-1} + \alpha\bu^\gamma_{m-1}))\big] 
\\&\quad
: \beps( \delta\bu^\gamma_{m-1} + \alpha\bu^\gamma_{m-1})  \Big\} \,\mathrm{d}x.
\end{aligned}
\end{equation}
We define the fourth order tensor \( \mathcal{B}_n ( \bbT) \), where \( \bbT \in \mathbb{R}^{d\times d}_{sym}\), by
\begin{align*}
\mathcal{B}_n^{ijkl}(\bbT) = \frac{\partial (F^{-1}_n)_{ij}}{\partial T_{kl}} ( \bbT).
\end{align*}
This defines an inner product on \( \mathbb{R}^{d\times d}_{sym}\) in the following way: 
\begin{align*}
( \bbS, \bbU)_{\mathcal{B}_n(\bbT)} = \mathcal{B}^{ijkl}_n(\bbT) S_{ij}U_{kl}. 
\end{align*}
It follows that, for every \( \bbT \), \( \bbS \in \mathbb{R}^{d\times d}_{sym}\),
\begin{equation}\label{a:equ39}
(F^{-1}_n(\bbT) - F^{-1}_n(\bbS)) : \bbT  = \int_0^1 (\bbT, \bbT - \bbS) _{\mathcal{B}_n(s\bbT + (1-s) \bbS)} \,\mathrm{d}s.
\end{equation}
On the other hand, we also have
\begin{align*}
&\varphi_n (F^{-1}_n(\bbT)) - \varphi_n(F^{-1}_n(\bbS)) 
\\
&= - \int_0^1 (1-s) (\bbT - \bbS, \bbT - \bbS)_{\mathcal{B}_n(s\bbT  + (1-s)\bbS)} \,\mathrm{d}s + \int_0^1 (\bbT, \bbT - \bbS)_{\mathcal{B}_n(s\bbT + (1-s) \bbS) }\,\mathrm{d}s.
\end{align*}
Substituting this into (\ref{a:equ39}), it follows that
\begin{align*}
&(F^{-1}_n(\bbT) - F^{-1}_n(\bbS) ):\bbT 
\\
&= \varphi_n (F^{-1}_n(\bbT)) - \varphi_n(F^{-1}_n(\bbS))  + \int_0^1 (1-s) (\bbT - \bbS, \bbT - \bbS)_{\mathcal{B}_n(s\bbT  + (1-s)\bbS)} \,\mathrm{d}s 
\\
&\geq \varphi_n (F^{-1}_n(\bbT)) - \varphi_n(F^{-1}_n(\bbS)).
\end{align*}
Returning to (\ref{a:equ40}), we deduce that
\begin{equation}\label{a:equ41}
\begin{aligned}
&\int_\Omega b(v^\gamma_{m-1}) F^{-1}_n(\beps(\delta\bu^\gamma_m + \alpha\bu^\gamma_m) ) : \beps(h (\delta^2\bu^\gamma_m + \alpha\delta\bu^\gamma_m)) \,\mathrm{d}x
\\
&\geq \int_\Omega b(v^\gamma_{m-1}) F^{-1}_n(\beps(\delta\bu^\gamma_m + \alpha\bu^\gamma_m)) : \beps( \delta\bu^\gamma_m + \alpha\bu^\gamma_m)\,\mathrm{d}x
\\
&\quad - \int_\Omega b(v^\gamma_{m-1}) F^{-1}_n(\beps(\delta\bu^\gamma_{m-1} + \alpha\bu^\gamma_{m-1})) : \beps( \delta\bu^\gamma_{m-1} + \alpha\bu^\gamma_{m-1})\,\mathrm{d}x
\\
&\quad
+ \int_\Omega b(v^\gamma_{m-1}) \big[ \varphi_n(F^{-1}_n( \beps(\delta\bu^\gamma_{m-1} + \alpha\bu^\gamma_{m-1}))) - \varphi_n(F^{-1}_n(\beps(\delta\bu^\gamma_m + \alpha\bu^\gamma_m ))) \,\mathrm{d}x
\\
&= \int_\Omega b(v^\gamma_{m-1}) \big[ \varphi^*_n(\beps(\delta\bu^\gamma_m + \alpha\bu^\gamma_m)) - \varphi^*_n (\beps(\delta\bu^\gamma_{m-1} + \alpha\bu^\gamma_{m-1})) \big] \,\mathrm{d}x
\\
&\geq \int_\Omega b(v^\gamma_{m}) \varphi^*_n(\beps(\delta\bu^\gamma_m + \alpha\bu^\gamma_m)) - b(v^\gamma_{m-1}) \varphi^*_n (\beps(\delta\bu^\gamma_{m-1} + \alpha\bu^\gamma_{m-1})) \,\mathrm{d}x.
\end{aligned}
\end{equation}
We insert (\ref{a:equ41}) into the elastodynamic equation and rewrite the terms involving \(\delta^2\bu^\gamma_m\) using (\ref{p:equ35}) to deduce that
\begin{align*}
&h \|\delta^2\bu^\gamma_m \|_2^2 + \frac{\alpha}{2} \Big( \|\delta\bu^\gamma_m \|_2^2 - \|\delta\bu^\gamma_{m-1}\|_2^2 + \|\delta\bu^\gamma_{m} - \delta\bu^\gamma_{m-1}\|_2^2 \Big) 
\\
&\quad + \int_\Omega b(v^\gamma_{m}) \varphi^*_n(\beps(\delta\bu^\gamma_m + \alpha\bu^\gamma_m)) - b(v^\gamma_{m-1}) \varphi^*_n (\beps(\delta\bu^\gamma_{m-1} + \alpha\bu^\gamma_{m-1}))  \,\mathrm{d}x
\\
&\leq \langle l^\gamma_m , \delta\bu^\gamma_m + \alpha\bu^\gamma_m \rangle - \langle l^\gamma_{m-1}, \delta\bu^\gamma_{m-1} + \alpha\bu^\gamma_{m-1}\rangle + h \langle \delta l^\gamma_m, \delta\bu^\gamma_{m-1} + \alpha\bu^\gamma_{m-1}\rangle.
\end{align*}
Using this inequality recursively and cancelling appropriate terms, we deduce that 
\begin{align*}
&h\sum_{j=1}^m \|\delta^2\bu^\gamma_m \|_2^2 +  \|\delta\bu^\gamma_m \|_2^2 + \int_\Omega b(v^\gamma_m) \varphi^*_n(\beps(\delta\bu^\gamma_m + \alpha\bu^\gamma_m)) \,\mathrm{d}x
\\
&\leq C\Big[ \|l^\gamma_m\|_{-1,2} \|\beps(\delta\bu^\gamma_m + \alpha\bu^\gamma_m ) \|_2 + \|l^\gamma_0\|_{-1,2} \|\beps(\bu_1 + \alpha\bu_0) \|_2 + h \sum_{j=1}^m \|\delta l^\gamma_j \|_{-1,2}^2 
\\
&\quad
+ h\sum_{j=0}^{m-1}\|\beps(\delta\bu^\gamma_j + \alpha\bu^\gamma_j ) \|_2^2 + \|\bu_1 \|_2^2 + \int_\Omega b(v_0) \varphi^*_n(\beps(\bu_1 + \alpha\bu_0)) \,\mathrm{d}x\Big].
\end{align*}
Using Young's inequality and the fact that \( \varphi^*_n(\bbT)\geq c(n) |\bbT|^2 \) to absorb the first term on the right-hand side into the left,  applying the bound in Lemma \ref{a:lem2}  it follows that for a constant \(C = C(n)\), independent of \( h \) and \( m \), we have 
\begin{align*}
h\sum_{j=1}^m \|\delta^2\bu^\gamma_m \|_2^2 +  \|\delta\bu^\gamma_m \|_2^2 + \int_\Omega b(v^\gamma_m) \varphi^*_n(\beps(\delta\bu^\gamma_ m + \alpha\bu^\gamma_m )) \,\mathrm{d}x
\leq C.
\end{align*}
Optimising with respect to \( m \) on the left-hand side,  we obtain the required result. 
\end{proof}


\begin{lemma}\label{a:lem4}
Let the assumptions of Theorem \ref{a:thm1} hold and let \((\bu^\gamma_m,v^\gamma_m)_{m=1}^M \) be the solution of the time discrete problem. Suppose additionally that \( k> \frac{d}{2}\).  There exists a constant \(C = C(n) \), independent of \( M \), such that 
\begin{align*}
\max_{1\leq m \leq M}\|\delta v^\gamma_m\|_{k,2} \leq C.
\end{align*}
\end{lemma}

\begin{proof}
Using Corollary \ref{a:cor2}, we have
\begin{align*}
&\|\delta v^\gamma_m \|_{k,2}^2 
\\
&= -\int_\Omega \Big\{ \frac{b^\prime(v^\gamma_m)}{\alpha}\delta v^\gamma_m  \varphi^*_n(\beps(\alpha\bu^\gamma_m)) + \frac{1}{2\epsilon}(v^\gamma_m - 1) \delta v^\gamma_m
+ 2\epsilon\nabla v^\gamma_m \cdot \nabla \delta v^\gamma_h\Big\} \,\mathrm{d}x 
\\
&\leq 
\frac{b^\prime(1)}{\alpha} \|\delta v^\gamma_m \|_\infty\|\varphi^*_n(\beps(\alpha\bu^\gamma_m )) \|_1 + \frac{1}{2\epsilon}\|v^\gamma_m - 1\|_2 \|\delta v^\gamma_m \|_2 + 2\epsilon \|\nabla v^\gamma_m \|_2 \|\nabla \delta v^\gamma_m \|_2
\\
&\leq  C\Big[ \Big( \int_\Omega\varphi^*_n(\beps(\alpha\bu^\gamma_m )) \,\mathrm{d}x\Big)^2 + \|v^\gamma_m -1\|_2^2 + \|\nabla v^\gamma_m \|_2^2 \Big] + \frac{1}{2}\|\delta v^\gamma\|_{k,2}^2.
\end{align*}
We absorb the final term into the left-hand side and use Lemma \ref{a:lem2} to bound the remaining terms.
\end{proof}


\begin{corollary}\label{a:cor3}
Let the assumptions of Theorem \ref{a:thm1} hold and let \((\bu^\gamma_m,v^\gamma_m)_{m=1}^M \) be the solution of the time discrete problem. Suppose additionally that \( k > \frac{d}{2}\).  Using the notation of Section \ref{sec:p}, there exists a constant \( C = C(n) \), independent of \( M \), such that
\begin{align*}
&\max_{t\in [0, T]}\|\overline{\bu}^\gamma(t)\|_{1,2} + \max_{t\in [0, T]}\|\bu^{\gamma,\pm}(t)\|_{1,2} + \max_{t\in [0, T]}\|\overline{\bu}^{\gamma,\prime}(t)\|_{1,2} + \max_{t\in [0, T]}\|\bu^{\gamma,\pm,\prime}(t)\|_{1,2} 
\\&\quad
+ \int_0^T \|\bu^{\gamma,+,\prime\prime}(t)\|_2^2 \,\mathrm{d}t + \max_{t\in [0, T]}\|\overline{v}^\gamma(t) \|_{k,2} + \max_{t\in [0, T]}\|v^{\gamma,\pm}(t) \|_{k,2} + \max_{t\in [0, T]}\|v^{\gamma,+,\prime}(t) \|_{k,2}
\\&
\leq C.
\end{align*}
\end{corollary}

\begin{proposition}\label{a:prop2}
Let the assumptions of Theorem \ref{a:thm1} hold. Suppose additionally that \( k > \frac{d}{2}\). There exists a weak energy solution \((\bu^n, \bbT^n, v^n) \) of the  regularised problem in the following sense. The weak elastodynamic equation holds, that is, for a.e. \( t\in (0, T) \), we have
\begin{equation}\label{a:equ9}
\int_\Omega\bu^n_{tt}(t) \cdot \bw + b(v^n(t)) \bbT^n(t) : \beps(\bw) \,\mathrm{d}x = \langle l^n(t), \bw\rangle,
\end{equation}
for every \( \bw \in W^{1,2}_D(\Omega)^d\), where the stress tensor \( \bbT^n \) is identified by the constitutive relation 
\begin{align*}
\beps(\bu^n_t + \alpha\bu^n) = F_n(\bbT^n) \quad\text{ a.e. in }Q. 
\end{align*}
For a.e. \( t\in [0, T]\), the following minimisation problem is satisfied:
\begin{equation}\label{a:equ61}
\begin{aligned}
&\mathcal{E}_n(\bu^n(t), v^n(t)) + \mathcal{H}(v^n(t)) + \mathcal{G}_k(v^n(t), v^n_t(t)) 
\\
&= \inf\Big\{  \mathcal{E}_n(\bu^n(t), v) + \mathcal{H}(v)  + \mathcal{G}_k(v, v^n_t(t)) \,:\, v\in H^k_{D+1}(\Omega),\, v\leq v^n(t) \Big\},
\end{aligned}
\end{equation}
Furthermore, for every \( t\in [0, T]\), we have the energy-dissipation balance
\begin{equation}\label{a:equ44}
\begin{aligned}
&\mathcal{F}_n(t;\bu^n(t), \bu^n_t(t), v^n(t)) + \int_0^t \langle l^n_t(s), \bu^n(s)\rangle + \|v^n_t(s) \|_{k,2}^2 \,\mathrm{d}s
\\
&\quad
+ \int_0^t \int_\Omega b(v^n) \big[ F^{-1}_n(\beps(\bu^n_t + \alpha\bu^n)) - F^{-1}_n(\beps(\alpha\bu^n)) \big] : \beps(\bu^n_t) \,\mathrm{d}x\,\mathrm{d}s
\\
&= \mathcal{F}_n(0;\bu_0, \bu_1, v_0),
\end{aligned}
\end{equation}
and the non-healing property \( v^n_t\leq 0 \). The initial conditions hold in the sense that
\begin{align*}
\lim_{t\rightarrow 0+}\Big[ \|\bu^n(t) - \bu_0\|_{1,2} + \|\bu^n_t(t) - \bu_1 \|_2 + \|v^n(t) - v_0\|_{k,2}\Big] = 0.
\end{align*}
Furthermore, there exists a subsequence (not relabelled) in \(M \),   independent of \( n \), such that the following convergence results hold:
\begin{itemize}
\item \( \overline{\bu}^{(n,M)}\overset{\ast}{\rightharpoonup}\bu^n \) weakly-* in \( W^{1,\infty}(0, T; W^{1,2}_D(\Omega)^d) \);
\item \( \bu^{(n,M),\pm}\), \( \bu^{(n,M),\pm,\prime}\overset{\ast}{\rightharpoonup}\bu^n\), \( \bu^n_t\) weakly-* in \( L^\infty(0, T; W^{1,2}_D(\Omega)^d) \), respectively;
\item \( \overline{\bu}^{(n,M),\prime}\overset{\ast}{\rightharpoonup}\bu^n_t\) weakly-* in \( L^\infty(0, T; W^{1,2}_D(\Omega)^d) \) and weakly in \( W^{1,2}(0,T; L^2(\Omega)^d) \);
\item \( \overline{v}^{(n,M)}\overset{\ast}{\rightharpoonup}v^n\) weakly-* in \( W^{1,\infty}(0, T; H^k(\Omega)) \);
\item \( v^{(n,M),\pm} \overset{\ast}{\rightharpoonup} v^n\) weakly-* in \( L^\infty(0, T; H^k_{D+1}(\Omega))\).
\end{itemize}
\end{proposition}
\begin{proof}
The stated convergence results follow immediately from standard compactness results for Bochner spaces. By the Aubin--Lions lemma and the fact that \( k > \frac{d}{2}\geq 1\), it follows that \((\overline{v}^{(n,M)})_M \) converges strongly in \( C([0, T]; H^1(\Omega)) \) to \( v^n \). As a result, \((v^{(n,M),\pm})_M \) converges strongly in \( L^\infty(0, T; H^1(\Omega)) \) to the same limit. Satisfaction of the initial condition for \( v^n \) follows   from the strong convergence in \( C([0, T]; H^1(\Omega)) \). We argue similarly for \( \bu^n(0) \) and \( \bu^n_t(0) \), applying the Aubin--Lions lemma again. 

Defining the approximate stress tensor \( \bbT^{(n,M)} = F^{-1}_n(\beps(\partial_t\overline{\bu}^{(n,M)} + \alpha\bu^{(n,M),+})) \), we see from Lemmas \ref{a:lem2} and \ref{a:lem3} that \((\bbT^{(n,M)})_{M}\) is bounded in \( L^\infty(0,T; L^2(\Omega)^{d\times d}) \). Hence it converges in this space to a limiting function \( \bbT^n \in L^\infty(0, T; L^2(\Omega)^{d\times d}) \). 

Rewriting the elastodynamic equation (\ref{a:equ7}) in terms of the interpolants, multiplying by an arbitrary \( \chi \in C([0, T]) \) and integrating over \( [0, T]\), we deduce that
\begin{align*}
\int_Q \partial_t \overline{\bu }^{\gamma, \prime} \cdot ( \chi \bw) + b(v^{\gamma,-}) \bbT^\gamma: \beps(\chi \bw) \,\mathrm{d}x \,\mathrm{d}t = \int_0^T \langle l^{\gamma,+}, \chi \bw\rangle\,\mathrm{d}t,
\end{align*}
for every \( \bw \in W^{1,2}_D(\Omega)^d\). Letting \( M \rightarrow\infty \) and using Lebesgue's differentiation theorem, it follows that  
\begin{align*}
\int_\Omega \bu_{tt}^n(t) \cdot \bw + b(v^n(t)) \bbT^n(t) :\beps(\bw) \,\mathrm{d}x = \langle l^n(t), \bw\rangle,
\end{align*}
for a.e. \( t\in (0, T) \) and every \( \bw \in W^{1,2}_D(\Omega)^d\). 

Next, noticing that \((\partial_t\overline{\bu}^{(n,M), \prime} + \alpha\bu^{(n,M),+})_M \) converges strongly in \( L^2(0, T; L^2(\Omega)^d) \), by the Aubin--Lions lemma, we deduce that
\begin{equation}\label{a:equ60}
\begin{aligned}
&\lim_{M\rightarrow\infty} \int_Q b(v^{(n,M), -}) \bbT^\gamma: \beps(\partial_t \overline{\bu}^{(n,M),\prime} + \alpha\bu^{(n,M),+}) \,\mathrm{d}x\,\mathrm{d}t 
\\&
= \lim_{M\rightarrow\infty}\Big\{ - \int_Q \partial_t \overline{\bu}^{(n,M), \prime}\cdot ( \partial_t \overline{\bu}^{(n,M),\prime} + \alpha\bu^{(n,M),+}) \,\mathrm{d}x\,\mathrm{d}t \\&\quad
+ \int_0^T \langle l^{(n,M),+}, \partial_t \overline{\bu}^{(n,M),\prime} + \alpha\bu^{(n,M),+}\rangle\,\mathrm{d}t \Big\}
\\
&= - \int_Q \bu^n_{tt}\cdot (\bu^n_t + \alpha\bu^n) \,\mathrm{d}x\,\mathrm{d}t + \int_0^T \langle l^n, \bu^n_t + \alpha\bu^n\rangle\,\mathrm{d}t
\\
&= \int_Q b(v^n) \bbT^n:\beps(\bu^n_t + \alpha\bu^n) \,\mathrm{d}x\,\mathrm{d}t.
\end{aligned}
\end{equation}
Applying Minty's method, as in the proof of Theorem \ref{p:thm6}, it follows from (\ref{a:equ60}) that we   have \( \beps(\bu^n_t + \alpha\bu^n) = F_n(\bbT^n) \)  pointwise a.e. in \( Q\). 

To prove that the minimisation problem (\ref{a:equ61}) is satisfied, we need a strong convergence result for \( (\bu^{(n,M),+})_M \) and \( ({\bu}^{(n,M), \prime, +})_M\) in \( L^2(0, T; W^{1,2}(\Omega)^d) \). 
Such convergence results allow us to deal with the presence of nonlinearities in the minimisation problem and energy-dissipation balance. Adapting the reasoning in \cite{MR2673410} (namely, Lemmas 3.9, 3.10 and 3.11) to the nonlinear problem, we deduce that this is  the case. Indeed, we have that
\begin{equation}\label{a:equ43}
\bu^{(n,M),+}\rightarrow\bu^n ,\quad \bu^{(n,M),\prime, +} \rightarrow\bu^n_t \quad\text{ strongly in }L^2(0, T; W^{1,2}(\Omega)^d) .
\end{equation}
As a result, we also have \( \bbT^{(n,M)} \rightarrow\bbT^n \) strongly in \( L^2(0, T; L^2(\Omega)^{d\times d}) \) as \( M \rightarrow\infty \). 
Now, we know that the satisfaction of the minimisation problem (\ref{a:equ61}) is equivalent to 
\begin{equation}\label{a:equ42}
0 \leq \big[ \partial_v \mathcal{E}_n(\bu^n(t), v^n(t))   + \mathcal{H}^\prime(v^n(t)) + \partial_v \mathcal{G}_k(v^n(t), v^n_t(t)) \big](\chi),
\end{equation} 
for every \( \chi \in H^k_D(\Omega) \) with \( \chi \leq 0 \), for a.e. \( t\in (0, T) \). 
However, (\ref{a:equ42}) follows  from taking \( M \rightarrow\infty \) in   Corollary \ref{a:cor1}, applying the weak convergence results for the approximations of the phase-field function and the strong convergence result (\ref{a:equ43}). 

It remains to show that the energy-dissipation equality (\ref{a:equ44}) holds. 
Using the convergence results, weak lower semi-continuity of norms and Fatou's lemma, it follows that (\ref{a:equ44}) holds with ``\(=\)'' replaced by ``\( \leq \)''. For the opposite inequality, we argue as in the proof of Proposition \ref{p:prop6}, adapting the reasoning slightly to take care of the extra rate-dependent term. 
\end{proof}


Now that we have the existence of a solution to the regularised problem, the focus is to obtain appropriate \( n \)-independent bounds that allow us to take the limit as \( n \rightarrow\infty \). Before formulating these bounds, we state a consequence of the energy-dissipation equality that  is analogous to the result in Corollary \ref{p:cor1}. 
We differentiate the energy-dissipation equality, carefully apply the chain rule for weak derivatives and test in the elastodynamic equation against \( \bu_t^n \). The remaining terms must vanish and are exactly the left-hand side of (\ref{a:equ45}). This result is vital to the proof of Lemma \ref{a:lem7}. In turn, Lemma \ref{a:lem7} is needed to obtain a bound on \((v^n_t)_n \) in \( L^\infty(0, T; H^k(\Omega)) \). Provided that \( k > \frac{d}{2}+1\), by the Sobolev embedding theorem  it follows that \((v^n)_n \) is uniformly bounded in \( W^{1,\infty}(0, T; W^{1,\infty}(\Omega)) \), which is required for the proof of higher regularity estimates for the sequence \((\bbT^n)_n \). 

\begin{corollary}\label{a:cor4}
Let the assumptions of Proposition \ref{a:prop2} hold and  let \((\bu^n, \bbT^n, v^n) \) be the solution triple constructed there. For a.e. \( t\in [0, T]\),
\begin{equation}\label{a:equ45}
\big[ \partial_v\mathcal{E}_n(\bu^n(t), v^n(t))  + \mathcal{H}^\prime(v^n(t)) + \partial_v\mathcal{G}_k(v^n(t), v_t^n(t)) \big](v^n_t(t)) = 0.
\end{equation}
\end{corollary}

Now we   derive \( n\)-independent bounds on the solution triple \((\bu^n , \bbT^n, v^n)\).  Testing in the elastodynamic equation (\ref{a:equ9}) against \( \bu^n_t + \alpha\bu^n \), we reason as in \cite{mypaperpreprint}  to obtain Lemma \ref{a:lem5}.  The  bound on the gradient of \(\bu^n \) follows from using the bound  on \(\bbT^n \), the constitutive relation, the fundamental theorem of calculus and the Korn--Poincar\'{e} inequality. Indeed, we recall the memory kernel property
\begin{equation}\label{a:equ69}
\beps(\bu^n(t)) =\mathrm{e}^{-\alpha t}\beps(\bu_0) + \int_0^t \mathrm{e}^{-\alpha(t-\tau)} F_n(\bbT^n(\tau)) \,\mathrm{d}\tau.
\end{equation}
Furthermore, we have \( |F_n(\bbT^n) |\leq 1 + n^{-1}|\bbT^n|\). The representation (\ref{a:equ69}) is vital to obtaining improved bounds on the solution triple, including using bounds on the linear combination \(\bu^n_t + \alpha\bu^n \) to yield bounds on the function \( \bu^n\). 

\begin{lemma}\label{a:lem5}
Let the assumptions of Proposition \ref{a:prop2} hold and let \((\bu^n, \bbT^n, v^n) \) be the solution triple constructed there. There exists a constant \( C\), independent of \( n \), such that
\begin{align*}
\sup_{t\in [0, T]}\|\bu^n(t) \|_{1,2} + \sup_{t\in [0, T]}\|\bu^n_t(t) \|_2+ \int_Q |\bbT^n| + \frac{|\bbT^n|^2}{n} + |\beps(\bu^n_t) |^2 \,\mathrm{d}x\,\mathrm{d}t \leq C.
\end{align*}
\end{lemma}

Combining  this with the energy-dissipation inequality immediately yields the  improved bound in Lemma \ref{a:lem6}. Indeed, from Lemma \ref{a:lem5}, we use the bound on \( (\bu^n)_n \) in \( L^\infty(0, T; W^{1,2}(\Omega)^d) \) to bound the terms in the energy-dissipation equality that result from the presence of an external force. 

At this point, we assume the safety strain condition (\ref{a:equ5}) so that the sequence \((\varphi^*_n(\beps(\alpha\bu_0)))_n \) is uniformly bounded in \( L^1(\Omega)\). We recall that convex conjugation is order reversing so, for every \( \bbS\in \mathbb{R}^{d\times d}\),  we have \( \varphi^*_n(\bbS) \leq \varphi^*(\bbS) \). It follows that
\begin{align*}
\int_\Omega\varphi^*_n( \beps(\alpha\bu_0)) \,\mathrm{d}x
\leq \int_\Omega\varphi^*(\beps(\alpha\bu_0)) \,\mathrm{d}x 
= \int_\Omega\int_0^{|\beps(\alpha\bu_0) |} \frac{s}{(1 - s^a)^{\frac{1}{a}}} \,\mathrm{d}s\,\mathrm{d}x\leq \frac{|\Omega| C_*^2}{(1 - C_*^a)^{\frac{1}{a}}},
\end{align*}
where \( C_* \) is the constant from the safety strain condition (\ref{a:equ5}). 

\begin{lemma}\label{a:lem6}
Let the assumptions of Proposition \ref{a:prop2} hold and  let \((\bu^n, \bbT^n, v^n) \) be the solution triple constructed there. Suppose additionally that the safety strain condition (\ref{a:equ5}) holds. There exists a constant \( C\), independent of \( n \), such that
\begin{align*}
\sup_{t\in [0, T]}\Big[ \int_\Omega \varphi^*_n(\beps(\alpha\bu^n(t))) \,\mathrm{d}x \Big] + \sup_{t\in [0, T]} \|v^n(t) \|_{k,2} + \int_0^T \|v^n_t(t) \|_{k,2}^2 \,\mathrm{d}t \leq C.
\end{align*}
\end{lemma}
Adapting the argument in Lemma \ref{a:lem4} and making use of Corollary \ref{a:cor4}, we can improve the bound on \((v^n_t)_n \) in \(L^2(0, T; H^k(\Omega)) \) to a bound in \( L^\infty(0, T; H^k(\Omega))\). As a result and recalling that \( k > \frac{d}{2}\), it follows that \((v^n_t)_n \) is bounded in \( L^\infty(Q)\), independent of \( n \). This  is used  to improve the time regularity of the approximations \((\bbT^n)_n \), which is the subject of Proposition \ref{a:prop3}. 


\begin{lemma}\label{a:lem7}
Let the assumptions of Proposition \ref{a:prop2} hold and let \((\bu^n, \bbT^n, v^n) \) be the solution triple constructed there. Suppose additionally that the safety strain condition (\ref{a:equ5}) holds.  There exists a constant \( C\) independent of \( n \) such that
\begin{align*}
\sup_{t\in [0, T]}\|v^n_t(t) \|_{k,2}\leq C.
\end{align*}
\end{lemma}

With Lemmas \ref{a:lem5}, \ref{a:lem6} and \ref{a:lem7} in mind, we work on improving the regularity of \((\bbT^n)_n \). Indeed, we look for weighted estimates on the derivatives \((\bbT^n_t)_n \) and \((\nabla \bbT^n)_n \) which allow us to prove a pointwise convergence result on \( Q\). For the pointwise result, we mimic some of the ideas from \cite{RN6}. However, significant adaptations are made as we work in the time-dependent setting. Also, we have the added complication of the phase-field function. To deal with this, Lemma \ref{a:lem7} is vital. 

We begin with studying the regularity in time. The proof of Proposition \ref{a:prop3} is one of the highlights of this work.  As mentioned previously, we work directly with the time discrete approximation. If we were to work with the time continuous solution, we would require \( \bu^n_{tt}\) to be continuous at \( t = 0\). This is not necessarily true. However, to  work with the time discrete solution, we must define an extra term \(\bu^{\gamma}_{-2}\) in the solution sequence in order to define \( \delta^2\bu^\gamma_0\). By our choice of \( \bu^\gamma_{-2}\), we   show that \(\delta^2\bu^\gamma_0\) is uniformly bounded in \(L^2(\Omega)^d\), independent of \( n \) and \( M \). We deduce from this that \( (\bu^n_{tt})_n \) is bounded in \(L^\infty(0, T; L^2(\Omega)^d) \), independent of \( n \). Furthermore, we see that we have \( \bbT^n\in W^{1,2}(0, T; L^2(\Omega)^{d\times d}) \) and obtain a weighted bound on the time derivative \(\bbT^n_t\) that is independent of \(n\). 

At this stage, we   impose the compatibility condition (\ref{a:equ4}). Under the compatibility condition, we know that the sequence of approximations \((l^n)_n \) is bounded in \(W^{2,1}(0, T; W^{-1,2}_D(\Omega)^d) \) and converges weakly in this space to \( l \). Furthermore, it becomes clear why the approximation choice (\ref{a:equ6})  is made. Indeed, suppose that \( \bu^n_{tt}\) is continuous at \( t = 0\) and sufficiently smooth so that the following reasoning is justified. Testing in the elastodynamic equation at \( t = 0 \) against \( \bu^n_{tt}(0) \), we get
\begin{align*}
\|\bu^n_{tt}(0) \|_2^2 &= \int_\Omega - b(v^n(0)) \bbT^n(0) \cdot \beps(\bu^n_{tt}(0)) \,\mathrm{d}x +\langle l^n(0) ,\bu^n_{tt}(0) \rangle
\\
&= \int_\Omega -b( v_0)F^{-1}_n(\beps(\bu_1 + \alpha\bu_0)) \cdot \beps(\bu^n_{tt}(0)  + \boldf(0) \cdot \bu^n_{tt}(0) \,\mathrm{d}x 
\\&\quad
+ \int_{\Gamma_N} F^{-1}_n(\beps(\bu_1 + \alpha\bu_0)) \mathbf{n} \cdot \bu^n_{tt}(0) \,\mathrm{d}S
\\
&= \int_\Omega \big( \mathrm{div}( b(v_0) F^{-1}_n(\beps(\bu_1 + \alpha\bu_0)) ) + \boldf(0) \big) \cdot \bu^n_{tt}(0) \,\mathrm{d}x,
\end{align*}
where the boundary integral vanishes due to integration by parts and the choice of \( \bg^n(0) \). It remains to show that \( \mathrm{div}( b(v_0) F^{-1}_n(\beps(\bu_1 + \alpha\bu_0)) ) \) is   bounded in \( L^2(\Omega)^d\), independent of \( n\), to deduce that \( \bu^n_{tt}(0)\) is also bounded in \( L^2(\Omega)^d\), independent of \( n \).  We use analogous reasoning in the proof of Proposition \ref{a:prop3} to show that the discrete time derivative \(\delta^2\bu^\gamma_0 \) is uniformly bounded in \(L^2(\Omega)^{d}\). Without the compatibility condition for \( \bg^n\), we get a boundary term involving \(\bu^n_{tt}(0) \), so we would need a bound on  \( (\bu^n_{tt}(0))_n \) in \( W^{1,2}(\Omega)^d\). However, this does not appear to be possible, thus justifying our choice of approximation and the compatibility condition (\ref{a:equ4}). 


\begin{proposition}\label{a:prop3}
Let the assumptions of Proposition \ref{a:prop2} hold and let  \((\bu^n, \bbT^n, v^n) \) be the solution triple constructed there. Suppose additionally that
\( l\in W^{2,1}(0, T; W^{-1,2}_D(\Omega)^d)  \), 
 the safety strain condition (\ref{a:equ5}) holds, the compatibility condition (\ref{a:equ4}) holds and that \( \bu_1 + \alpha\bu_0 \in W^{2,2}_D(\Omega)^d\). Then \( \bbT^n \in W^{1,2}(0, T; L^2(\Omega)^{d\times d}) \) and there exists a constant \( C\), independent of \( n \),  such that
\begin{align*}
\sup_{t\in [0, T]}\|\bu^n_{tt}(t) \|_2 + \int_Q \frac{|\bbT^n_t|^2}{(1 + |\bbT^n|)^{1+a}} + \frac{|\bbT^n_t|^2}{n}\,\mathrm{d}x\,\mathrm{d}t\leq C.
\end{align*}
\end{proposition}

\begin{proof}
We define \( \bbT^\gamma_m = F^{-1}_n(\beps(\delta\bu^\gamma_m + \alpha\bu^\gamma_m)) \) for \( 0 \leq m \leq M \),  where \((\bu^\gamma_m, v^\gamma_m)_{m=0}^M \) is the solution sequence for the time discrete problem in Theorem \ref{a:thm1}, with \( \beta = (n,M) \). 
We fix  \( 2\leq m \leq M \),  test in the elastodynamic equation (\ref{a:equ7}) at time levels \( m \) and \( m-1\) against \( \delta^2\bu^\gamma_m + \alpha\delta\bu^\gamma_m \) and subtract the resulting equalities. Using (\ref{p:equ35}), we rewrite the interial terms in the usual way. For the body force \(l^n \), to avoid having terms involving  spatial derivatives of \( \delta^2\bu^\gamma_m \), which we cannot necessarily bound, we write
\begin{align*}
\langle l^\gamma_m - l^\gamma_{m-1}, \delta^2 \bu^\gamma_m + \alpha\delta\bu^\gamma_m\rangle &= \langle \delta l^\gamma_m, \delta\bu^\gamma_m + \alpha\bu^\gamma_m \rangle  - \langle \delta l^\gamma_{m-1}, \delta\bu^\gamma_{m-1} + \alpha\bu^\gamma_{m-1} \rangle 
\\&\quad
+ h\langle \delta^2 l^\gamma_m , \delta\bu_{m-1}^\gamma + \alpha\bu_{m-1}^\gamma\rangle.
\end{align*}
For the nonlinear terms, we write
\begin{equation}\label{a:equ46}
\begin{aligned}
&\int_\Omega \big[ b(v^\gamma_{m-1}) \bbT^\gamma_m - b(v^\gamma_{m-2}) \bbT^\gamma_{m-1}\big] :\beps(\delta^2 \bu^\gamma_m + \alpha\delta\bu^\gamma_m ) \,\mathrm{d}x
\\
&= \int_\Omega \Big\{ b(v^\gamma_{m-1}) \big( \bbT^\gamma_m  - \bbT^\gamma_{m-1} \big) : \frac{F_n(\bbT^\gamma_m ) - F_n(\bbT^\gamma_{m-1})}{h}
\\
&\quad
 + \big( b(v^\gamma_{m-1}) - b(v^\gamma_{m-2}) \big) \bbT^\gamma_{m-1} :\frac{F_n(\bbT^\gamma_m) - F_n(\bbT^\gamma_{m-1}) }{h}\Big\}  \,\mathrm{d}x
 \\
 & = h\int_\Omega\int_0^1 b(v^\gamma_{m-1}) \Big( \frac{\bbT^\gamma_m - \bbT^\gamma_{m-1}}{h}, \frac{\bbT^\gamma_m - \bbT^\gamma_{m-1}}{h}\Big)_{\mathcal{A}_n(s\bbT^\gamma_m + (1-s) \bbT^\gamma_{m-1}) } \,\mathrm{d}s\,\mathrm{d}x
 \\
 &\quad
 + h \int_\Omega\int_0^1 \Big( \frac{b(v^\gamma_{m-1}) - b(v^\gamma_{m-2}) }{h}\Big) \cdot \Big( \bbT^\gamma_{m-1}, \frac{\bbT^\gamma_m - \bbT^\gamma_{m-1}}{h}\Big)_{\mathcal{A}_n(s\bbT^\gamma_m + (1-s) \bbT^\gamma_{m-1}) } \,\mathrm{d}s\,\mathrm{d}x
 \\
 &\geq \frac{h}{2} \int_\Omega\int_0^1 b(v^\gamma_{m-1}) \Big( \frac{\bbT^\gamma_m - \bbT^\gamma_{m-1}}{h}, \frac{\bbT^\gamma_m - \bbT^\gamma_{m-1}}{h}\Big)_{\mathcal{A}_n(s\bbT^\gamma_m + (1-s) \bbT^\gamma_{m-1}) } \,\mathrm{d}s\,\mathrm{d}x
 \\
 &\quad
 - \frac{h}{2\eta} \int_\Omega\int_0^1 \Big( \frac{b(v^\gamma_{m-1}) - b(v^\gamma_{m-2}) }{h}\Big)^2 \big( \bbT^\gamma_{m-1}, \bbT^\gamma_{m-1}\big)_{\mathcal{A}_n(s\bbT^\gamma_m + (1-s) \bbT^\gamma_{m-1}) } \,\mathrm{d}s\,\mathrm{d}x,
\end{aligned}
\end{equation}
where \( \mathcal{A}_n(\bbU ) \) is the fourth-order tensor defined by
\begin{align*}
\mathcal{A}_n(\bbU)_{ijkl} = \frac{\partial F_n(\bbU)_{ij}}{\partial U_{kl}} 
\end{align*}
with corresponding bilinear form  
\begin{align*}
(\bbT, \bbS)_{\mathcal{A}_n(\bbU) } := \sum_{i,j,k,l=1}^d \mathcal{A}_n(\bbU)_{ijkl} T_{ij}S_{kl}.
\end{align*}
We can easily check that \((\cdot, \cdot)_{\mathcal{A}_n(\bbU) }\) defines an inner product on \( \mathbb{R}^{d\times d}\), justifying the use of the Cauchy--Schwarz inequality in (\ref{a:equ46}). Summing the result  over the indices \( 2\leq j \leq m\), we deduce that
\begin{equation}\label{a:equ10}
\begin{aligned}
&\frac{1}{2}\Big( \|\delta^2\bu^\gamma_m\|_2^2 - \|\delta^2\bu^\gamma_1\|_2^2 \Big) + \alpha(\delta^2\bu^\gamma_m, \delta\bu^\gamma_m) - \alpha (\delta^2\bu^\gamma_1, \delta\bu^\gamma_1) 
\\
&\quad
+ \frac{h}{2}\sum_{j=2}^m \int_\Omega\int_0^1 b(v^\gamma_{j-1}) \Big( \delta\bbT^\gamma_j , \delta\bbT^\gamma_j \Big)_{\mathcal{A}_n(s\bbT^\gamma_j  + (1-s) \bbT^\gamma_{j-1}) } \,\mathrm{d}s\,\mathrm{d}x
\\&\quad - \alpha h \sum_{j=2}^m (\delta^2 \bu^\gamma_j , \delta^2\bu^\gamma_{j-1})
\\
&\leq \langle \delta l^\gamma_m , \delta\bu^\gamma_m + \alpha\bu^\gamma_m \rangle - \langle \delta l^\gamma_1 , \delta\bu^\gamma_1 + \alpha\bu^\gamma_1 \rangle 
+ h\sum_{j=2}^m \langle \delta^2 l^\gamma_j , \delta\bu^\gamma_{j-1} + \alpha\bu^\gamma_{j-1}\rangle
\\
&\quad
+ \frac{h}{2\eta}\sum_{j=1}^{m-1} \int_\Omega\int_0^1 \Big( \frac{b(v^\gamma_j) - b(v^\gamma_{j-1})}{h}\Big)^2 (\bbT^\gamma_j, \bbT^\gamma_j)_{\mathcal{A}_n(s\bbT^\gamma_{j+1} + (1-s) \bbT^\gamma_{j}) } \,\mathrm{d}s\,\mathrm{d}x.
\end{aligned}
\end{equation}
To bound the terms involving the index 1, in particular, the term \( \delta^2\bu^\gamma_1\), we  need to define \( \delta^2\bu^\gamma_0 \). Equivalently, we look for a \( \bu_{-2}^\gamma\in L^2(\Omega)^d\) such that, for every \( \bw \in W^{1,2}_D(\Omega)^d\), 
\begin{equation}\label{a:equ47}
\begin{aligned}
&\int_\Omega\frac{\bu^\gamma_0 - 2\bu^\gamma_{-1} + \bu^\gamma_{-2}}{h^2}\cdot \bw \,\mathrm{d}x \\
&= - \int_\Omega b(v_0) F^{-1}_n(\beps(\delta\bu^\gamma_0 + \alpha\bu^\gamma_0)):\beps(\bw) \,\mathrm{d}x + \langle l^n(0), \bw\rangle
\\
&= \int_\Omega\Big[ \diver\Big( b(v_0) F^{-1}_n(\beps(\bu_1 + \alpha\bu_0)) \Big)  + \boldf(0) \Big]:\bw\,\mathrm{d}x,
\end{aligned}
\end{equation}
using integration by parts and the compatibility condition satisfied by the approximation \((l^n)_n \). Under the regularity assumptions, namely, \( \bu_1  + \alpha\bu_0 \in W^{2,2}_D(\Omega)^d\),  we have 
\begin{align*}
\Big\|\diver\Big( b(v_0) F^{-1}_n(\beps(\bu_1 + \alpha\bu_0)) \Big)  + \boldf(0)\Big\|_2\leq C,
\end{align*}
where \( C\) is independent of \( n \). 
The bound on \( \diver( b(v_0) F^{-1}_n(\beps(\bu_1 + \alpha\bu_0)) ) \) in \( L^2(\Omega)^d\) can be determined from the following reasoning. Performing the differentiation, we have
\begin{align*}
&\diver\big( b(v_0) F^{-1}_n(\beps(\bu_1 + \alpha\bu_0))\big)_i
\\
&= \frac{\partial}{\partial x_j}\Big( b(v_0) F^{-1}_n(\beps(\bu_1 + \alpha\bu_0))_{ij}\Big)
\\
&= b^\prime(v_0) \frac{\partial v_0 }{\partial x_j} F^{-1}_n(\beps(\bu_1 + \alpha\bu_0))_{ij}
+ b(v_0 ) \mathcal{A}_n(\beps(\bu_1 + \alpha\bu_0))_{ijkl} \frac{\partial \beps(\bu_1 + \alpha\bu_0)_{kl} }{\partial x_j}.
\end{align*}
Next, we notice that
\begin{align*}
|F^{-1}_n(\beps(\bu_1 + \alpha\bu_0))| \leq |F^{-1}(\beps(\bu_1 + \alpha\bu_0))| \leq \frac{C_*}{(1 + C_*^a)^{\frac{1}{a}}},
\end{align*}
and from direct calculation, we have
\begin{align*}
\big|\mathcal{A}_n(\bbU)_{ijkl}\big| = \Big|\delta_{ik}\delta_{jl} \Big( \frac{1}{(1 + |\bbU|^a)^{\frac{1}{a}}} + \frac{1}{n}\Big) - \frac{|\bbU|^{a-2}\bbU_{ij}\bbU_{kl}}{(1 + |\bbU|^a)^{1  + \frac{1}{a}}}\Big|\leq 3. 
\end{align*}
Thus we have
\begin{align*}
\big|\diver\big( b(v_0) F^{-1}_n(\beps(\bu_1 + \alpha\bu_0))\big) \big| \leq \frac{2C_*}{(1 - C_*^a)^{\frac{1}{a}}} \Big|\frac{\partial v_0 }{\partial x_j}\Big| + 6 \Big|\frac{\partial\beps( \bu_1 + \alpha\bu_0)_{kl}}{\partial x_j}\Big|.
\end{align*}
It follows that \( \diver( b(v_0) F^{-1}_n(\beps(\bu_1 + \alpha\bu_0)))\) is an element of \( L^2(\Omega)^d\) with bound indepndent of \( n \), since \( v_0 \in H^1(\Omega) \) and \( \bu_1 + \alpha\bu_0\in W^{2,2}_D(\Omega)^d\). 
Returning to (\ref{a:equ47}), we deduce that there exists a unique \(\bu^\gamma_{-2}\in L^2(\Omega)^d\) solving (\ref{a:equ47}). 
As a result, the difference quotient \( \delta^2 \bu_0^\gamma\) is well-defined and is an element of \( L^2(\Omega)^d\). 
Testing in (\ref{a:equ47}) against \( \delta^2 \bu^\gamma_0 \) yields
\begin{align*}
\|\delta^2\bu^\gamma_0\|_2^2 &\leq \Big\| \diver\Big( b(v_0) F^{-1}_n(\beps(\bu_1 + \alpha\bu_0)) \Big)  + \boldf(0)\Big\|_2\|\delta\bu^\gamma_0 \|_2
\\& \leq C\|\delta^2\bu^\gamma_0 \|_2.
\end{align*}
Hence \( \delta^2\bu^\gamma_0\) is bounded in \( L^2(\Omega)^d\), independent of \( n \) and \( M \). 
Subtracting the equality
\begin{align*}
\int_\Omega\delta^2\bu^\gamma_0 \cdot \bw + b(v^\gamma_0) \bbT^\gamma_0 :\beps(\bw) \,\mathrm{d}x = \langle l^\gamma_0, \bw\rangle,
\end{align*}
from the elastodynamic equation (\ref{a:equ7}) at time level \( m = 1\), and choosing the test function  \( \bw = \delta^2\bu^\gamma_1 + \alpha\delta\bu^\gamma_1\), 
we can manipulate the resulting equation as in the case \( m \geq 2\). 
We add the corresponding result to  (\ref{a:equ10}) and  deduce that 
\begin{equation}\label{a:equ11}
\begin{aligned}
&\frac{1}{2}\Big( \|\delta^2\bu^\gamma_m \|_2^2 - \|\delta^2\bu^\gamma_0\|_2^2 \Big) + \alpha(\delta^2\bu^\gamma_m , \delta\bu^\gamma_m ) - \alpha (\delta^2\bu^\gamma_0, \delta\bu^\gamma_1 )
\\
&\quad
+ \frac{h}{2}\sum_{j=1}^m\int_\Omega\int_0^1 b(v^\gamma_{j-1}) \Big( \delta\bbT^\gamma_j, \delta\bbT^\gamma_j \Big)_{\mathcal{A}_n(s\bbT^\gamma_j +(1-s)\bbT^\gamma_{j-1}) }\,\mathrm{d}s\,\mathrm{d}x
\\
&\quad
- \alpha h\sum_{j=2}^m (\delta^2 \bu^\gamma_j, \delta^2\bu^\gamma_{j-1})
\\
&\leq \langle \delta l^\gamma_m ,\delta\bu^\gamma_m + \alpha\bu^\gamma_m \rangle - \langle \delta l^\gamma_1, \bu_1 + \alpha\bu_0 \rangle + h \sum_{j=2}^m \langle \delta^2 l^\gamma_j , \delta\bu^\gamma_{j-1} + \alpha\bu^\gamma_{j-1}\rangle 
\\
&\quad
+ \frac{h}{2\eta} \sum_{j=1}^{m-1} \int_\Omega\int_0^1 \Big( \frac{b(v^\gamma_j) - b(v^\gamma_{j-1})}{h}\Big)^2 (\bbT^\gamma_j, \bbT^\gamma_j)_{\mathcal{A}_n(s\bbT^\gamma_{j+1} + (1-s) \bbT^\gamma_{j}) } \,\mathrm{d}s\,\mathrm{d}x.
\end{aligned}
\end{equation}
We have 
\begin{equation}\label{a:equ48}
\alpha h\sum_{j=2}^m (\delta^2\bu^\gamma_j, \delta^2\bu^\gamma_{j-1}) \leq C(\alpha) h \sum_{j=1}^{m-1} \|\delta^2 \bu^\gamma_j \|_2^2 + \frac{\|\delta^2\bu^\gamma_m \|_2^2}{8},
\end{equation}
and 
\begin{equation}\label{a:equ49}
\alpha (\delta^2\bu^\gamma_{m},\delta\bu^\gamma_m ) \leq \frac{\|\delta^2 \bu^\gamma_m \|_2^2}{8}  + 2\alpha^2 \|\delta\bu^\gamma_m \|_2^2. 
\end{equation}
Substituting (\ref{a:equ48}) and (\ref{a:equ49}) into (\ref{a:equ11}) and applying the discrete Gronwall inequality, it follows that, for every \( 1\leq m \leq M\), 
\begin{equation}\label{a:equ12}
\begin{aligned}
&\|\delta^2\bu^\gamma_m \|_2^2 + h \sum_{j=1}^m \int_\Omega\int_0^1 b(v^\gamma_{j-1}) \Big( \delta\bbT^\gamma_j, \delta\bbT^\gamma_j \Big)_{\mathcal{A}_n(s\bbT^\gamma_j + (1-s) \bbT^\gamma_{j-1}) } \,\mathrm{d}s\,\mathrm{d}x
\\
&\leq 
C(\delta)\Big[ \|\delta^2\bu^\gamma_0\|_2^2 + \max_{0 \leq j \leq m }\|\delta\bu^\gamma_j \|_2^2 + \max_{1\leq j \leq m }\|\delta l^\gamma_j \|_{-1,2}^2 + \Big( h \sum_{j=2}^m \|\delta^2 l^\gamma_j \|_{-1,2} \Big)^2
\\
&\quad
 + h \sum_{j=1}^{m-1} \int_\Omega\int_0^1 \Big( \frac{b(v^\gamma_j) - b(v^\gamma_{j-1})}{h}\Big)^2 (\bbT^\gamma_j , \bbT^\gamma_j )_{\mathcal{A}_n(s\bbT^\gamma_{j+1} +(1-s) \bbT^\gamma_j ) } \,\mathrm{d}s\,\mathrm{d}x + 1\Big]  
 \\&\quad
 + \delta \max_{0\leq j \leq m }\|\beps(\delta\bu^\gamma_j + \alpha\bu^\gamma_j ) \|_{2}^2 ,
 \end{aligned}
\end{equation}
where \( C= C(\delta)\) is a constant that is independent of \( M\) and \( n \), but does depend on \( \delta\), which is fixed sufficiently small so that the final term on the right can be absorbed into the left-hand side.
To prove that \( \delta\) can be chosen such that this is possible, we first notice that
\begin{align*}
\bbT^\gamma_m = h \sum_{j=1}^m \delta\bbT^\gamma_j + \bbT^\gamma_0 = h \sum_{j=1}^m \delta\bbT^\gamma_j + F^{-1}_n(\beps(\bu_1 + \alpha\bu_0)) .
\end{align*}
It follows that
\begin{align*}
\max_{1\leq j \leq m}\frac{\|\bbT^\gamma_j\|_2^2}{n} &\leq  C\Big(  h\sum_{j=1}^m \frac{\|\delta\bbT^\gamma_j \|_2^2}{n} + \frac{\|F^{-1}_n(\beps(\bu_1 + \alpha\bu_0)) \|_2^2}{n}\Big) 
\\
& \leq C \Big(  h\sum_{j=1}^m \int_\Omega\int_0^1 b(v^\gamma_{j-1}) \Big( \delta\bbT^\gamma_j, \delta\bbT^\gamma_j \Big)_{\mathcal{A}_n(s\bbT^\gamma_j +(1-s)\bbT^\gamma_{j-1}) }\,\mathrm{d}s\,\mathrm{d}x + 1\Big) .
\end{align*}
By definition, we have
\begin{align*}
|\beps(\delta\bu^\gamma_m + \alpha\bu^\gamma_m ) | \leq 1 + \frac{|\bbT^\gamma_m|}{n}.
\end{align*}
This yields
\begin{align*}
&\max_{1\leq j \leq m}\|\beps(\delta\bu^\gamma_j + \alpha\bu^\gamma_j) \|_2^2 
\\&
\leq C\Big( 1 +  h  \sum_{j=1}^m \int_\Omega\int_0^1 b(v^\gamma_{j-1})   (\delta \bbT^\gamma_j , \delta \bbT^\gamma_j )_{\mathcal{A}_n(s\bbT^\gamma_{j} +(1-s) \bbT^\gamma_{j-1} ) } \,\mathrm{d}s\,\mathrm{d}x \Big) .
\end{align*}
Thus we see that \( \delta\) can be chosen sufficiently small so that the final term on the right-hand side of (\ref{a:equ12}) can be absorbed into the left. Optimising over the indices \( 1\leq m \leq M\), we deduce that
\begin{align*}
&\max_{1\leq m \leq M} \|\delta^2\bu^\gamma_m \|_2^2 + h \sum_{m=1}^M \int_\Omega\int_0^1 b(v^\gamma_m) \big( \delta\bbT^\gamma_m , \delta\bbT^\gamma_m \big)_{\mathcal{A}_n( s\bbT^\gamma_m + (1-s) \bbT^\gamma_{m-1}) } \,\mathrm{d}s\,\mathrm{d}x
\\
&\leq
C\Big[ 1 + \max_{1\leq m \leq M} \|\delta\bu^\gamma_m\|_2^2 + \max_{1\leq m \leq M} \|\delta l^\gamma_m \|_{-1,2}^2 + \Big( h \sum_{m=2}^M\|\delta^2 l^\gamma_m \|_{-1,2}\Big)^2
\\
&\quad
+ h \sum_{m=1}^{M-1} \int_\Omega\int_0^1 \Big( \frac{b(v^\gamma_m) - b(v^\gamma_{m-1}) }{h}\Big)^2 (\bbT^\gamma_{m}, \bbT^\gamma_m )_{\mathcal{A}_n( s\bbT^\gamma_{m+1}+ (1-s) \bbT^\gamma_m)} \,\mathrm{d}s\,\mathrm{d}x \Big] . 
\end{align*}
Rewriting this in terms of interpolants of the discrete solution sequence, we obtain
\begin{equation}\label{a:equ13}
\begin{aligned}
&\sup_{t\in [0, T]} \|\overline{\bu}^{\gamma,\prime}_t (t) \|_2^2 + 
\sup_{t\in [0, T]}\|\beps(\overline{\bu}^\gamma_t + \alpha\bu^{\gamma,+}) \|_2^2 
\\&\quad 
+\int_Q 
\int_0^1 b(v^{\gamma,-}) \big( \overline{\bbT}^\gamma_t, \overline{\bbT}^\gamma_t \big)_{\mathcal{A}_n(s\bbT^{\gamma,+} + (1-s) \bbT^{\gamma,-}) } \,\mathrm{d}s\,\mathrm{d}x\,\mathrm{d}t
\\
&\leq 
C\Big[1 +  \max_{t\in [0, T]} \|\overline{\bu}^\gamma (t) \|_2^2 + \max_{t \in [0, T] } \|\overline{l}^\gamma_t (t) \|_{-1,2}^2 
+ \Big( \int_0^T \|l^{\gamma,+,\prime\prime}(t)\|_{-1,2}\,\mathrm{d}t \Big)^2
\\
&\quad 
+ h \int_Q\int_0^1   \Big( \frac{b(v^{\gamma,+}) - b(v^{\gamma,-})}{h}\Big)^2  \cdot 
 (\bbT^{\gamma,+}, \bbT^{\gamma,+})_{\mathcal{A}_n(s\tau^h_t\bbT^{\gamma,+} + (1-s) \bbT^{n
\gamma,+})}   \,\mathrm{d}s\,\mathrm{d}x\,\mathrm{d}t\Big],
\end{aligned}
\end{equation}
where \(C \) is independent of \( M \) and \( n\) and \( \tau^h_t \) denotes the translation operator  in the time variable of length \( h \).

Now we show that \( \bbT^n \in W^{1,2}(0, T; L^2(\Omega)^{d\times d}) \). 
To bound the right-hand side of (\ref{a:equ13}) independent of \( M \), we notice that
\begin{align*}
\frac{b(v^{\gamma,+}) - b(v^{\gamma,-})}{h} \leq \|\overline{v}^\gamma_t \|_\infty\|b^\prime(\overline{v}^\gamma) \|_\infty \leq C(n).
\end{align*} 
Furthermore, for every \( \bbT\), \( \bbS\in \mathbb{R}^{d\times d}\), we have
\begin{align*}
\frac{|\bbT|^2}{n} \leq (\bbT, \bbT)_{\mathcal{A}_n(\bbS) }\leq 3|\bbT|^2.
\end{align*}
Using these two statements with the previously determined \(M\)-independent bounds and the fact that \( (l^n)_n\) is bounded in \( W^{2,1}(0, T; W^{-1,2}_D(\Omega)^d) \), it follows from  (\ref{a:equ13}) that
\begin{equation}\label{a:equ14}
\begin{aligned}
&\sup_{t\in [0, T]} \big( \|\overline{\bu}^{\gamma,\prime}_t (t) \|_2^2 \big) +
\sup_{t\in [0, T]}\big( \|\beps(\overline{\bu}^\gamma_t + \alpha\bu^{\gamma,+}) \|_2^2\big)
+ \int_Q 
\frac{|\overline{\bbT}^\gamma_t|^2}{n}\,\mathrm{d}x\,\mathrm{d}t
\\
&\leq C\Big\{ 1 + \|l^\gamma_t\|_{L^\infty(W^{-1,2})}  + \|l^\gamma_{tt}\|_{L^1(W^{-1,2})}^2 + \int_Q |\bbT^{\gamma,+}|^2\,\mathrm{d}x\,\mathrm{d}s\Big\}
\\
&\leq C,
\end{aligned}
\end{equation}
where \( C = C(n) \) is independent of \( M \). Hence \((\overline{\bbT}^{(n,M)})_M \) is bounded in \( W^{1,2}(0, T; L^2(\Omega)^{d\times d}) \) independent of \( M \) and so converges weakly in this space to \( \bbT^n  = F_n(\beps(\bu^n_t + \alpha\bu^n)) \). 
However, since \((\overline{\bu}^{(n,M)}_t + \alpha\bu^{(n,M),+})_M \) converges strongly in \( L^2(0, T; W^{1,2}_D(\Omega)^d) \) by (\ref{a:equ43}),  it follows from the constitutive relation that  \((\bbT^{(n,M),+})_M \) converges strongly in \( L^2(Q)^{d\times d}\). Using Lebesgue's dominated convergence theorem, we get
\begin{equation}\label{a:equ15}
\begin{aligned}
&\lim_{M\rightarrow\infty }\int_Q \int_0^1 (\bbT^{(n,M),+}, \bbT^{(n,M),+})_{\mathcal{A}_n(s\tau^h_t\bbT^{(n,M),+} + (1-s) \bbT^{(n,M),+})} \,\mathrm{d}s\,\mathrm{d}x\,\mathrm{d}t \\&\quad
= \int_Q (\bbT^n,\bbT^n)_{\mathcal{A}_n(\bbT^n)}\,\mathrm{d}x\,\mathrm{d}t.
\end{aligned}
\end{equation}
Making further use of the previously determined strong convergence results for the interpolants, we see that
\begin{equation}\label{a:equ16}
\begin{aligned}
&\lim_{M\rightarrow\infty}\Big[ \max_{t\in [0, T]} \|\overline{\bu}^{(n,M)}(t) \|_2^2 + \max_{t\in [0, T]} \|\overline{l}^{(n,M)}_t(t) \|_{-1,2}^2 + \Big( \int_0^T \|l^{(n,M),+,\prime\prime}(t)\|_{-1,2}\,\mathrm{d}t \Big)^2 \Big]
\\
&= 
\sup_{t\in [0, T]} \|\bu^n_t(t) \|_2^2 + \sup_{t\in [0, T]}\|l^n_t(t) \|_{-1,2}^2 +\Big(  \int_0^T \|l^n_{tt}(t) \|_{-1,2}\,\mathrm{d}t\Big)^2
\\
&\leq C
\end{aligned}
\end{equation}
where \( C\) is a constant that is independent of \( n\).
Ignoring the final term on the left-hand side of (\ref{a:equ13}), and using (\ref{a:equ15}) and (\ref{a:equ16}), it follows that 
\begin{align*}
\limsup_{M\rightarrow\infty}\Big[ \sup_{t\in [0, T]} \|\overline{\bu}^{(n,M),\prime}_t (t) \|_2^2 + 
\sup_{t\in [0, T]}\|\beps(\overline{\bu}^{(n,M)}_t + \alpha\bu^{(n,M),+}) (t) \|_2^2 \Big] \leq C,
\end{align*}
where \( C\) is independent of \( n\). By weak lower semi-continuity, we have
\begin{equation}\label{a:equ17}
\sup_{t\in [0, T]}\|\bu^n_{tt}(t) \|_2 + \sup_{t\in [0, T]}\|\beps(\bu^n_t + \alpha\bu^n)(t) \|_2 \leq C.
\end{equation}
Next we use (\ref{a:equ17}) with the extra regularity of \( \bbT^n\) in order to obtain the  bound on \( \bbT^n_t \) in the statement of the proposition. 
Using  that \( \bbT^n \in W^{1,2}(0, T; L^2(\Omega)^{d\times d})\), it follows that we  have \( F_n(\bbT^n) \in W^{1,2}(0, T; L^2(\Omega)^{d\times d}) \), with weak derivative given by  the chain rule, that is, \( \mathcal{A}_n(\bbT^n) \bbT^n_t \). By standard properties of differentiable Bochner functions, the following convergence results hold  for a.e. \( t\in (0, T) \) with respect to strong convergence in \( L^2(\Omega)^{d\times d}\):
\begin{equation}\label{a:equ18}
\lim_{h\rightarrow 0+  }\frac{\bbT^n(t+h) - \bbT^n(t) }{h} = \bbT^n_t,\quad
\lim_{h\rightarrow 0+} \frac{F_n(\bbT^n(t + h)) - F_n(\bbT^n(t))}{h} = \mathcal{A}_n(\bbT^n) \bbT^n_t.
\end{equation}
Let  \( \Delta^h_t \) denote the undivided difference quotient of length \( h \) in the time variable and let  \( t\in (0,T) \). Considering the elastodynamic equation at times \( t\) and \( t +h \), for a  \( h>0 \) sufficiently small such that \( t  + h < T\),  choosing \( \Delta^h_t ( \bu^n_t + \alpha\bu^n)(t)\) as the test function  we see that 
\begin{equation}\label{a:equ21}
\begin{aligned}
&\langle \Delta^h_t l^n(t), \Delta^h_t (\bu^n_t  + \alpha\bu^n)(t)\rangle
\\
&=\int_\Omega\Delta^h_t \bu^n_{tt}(t) \cdot \Delta^h_t (\bu^n_t  + \alpha\bu^n)(t) + \Delta^h_t \Big( b(v^n)\bbT^n\Big) (t):\beps\big( \Delta^h_t (\bu^n_t  + \alpha\bu^n)\big)(t) \,\mathrm{d}x 
\\&
=  \frac{\mathrm{d}}{\mathrm{d}t}\Big( \frac{\|\Delta^h_t\bu^n_t\|_2^2}{2} + \alpha\int_\Omega\Delta^h_t \bu^n_t\cdot \Delta^h_t \bu^n \,\mathrm{d}x\Big) - \alpha\int_\Omega|\Delta^h_t \bu_t^n|^2\,\mathrm{d}x
\\
&\quad
+ \int_\Omega\Delta^h_t b(v^n) \tau^h_t \bbT^n : \Delta^h_t F_n(\bbT^n) + b(v^n) \Delta^h_t \bbT^n : \Delta^h_t F_n(\bbT^n) \,\mathrm{d}x.
\end{aligned}
\end{equation}
From Lemmas \ref{a:lem6} and \ref{a:lem7}, we have
\begin{align*}
|\Delta^h_t b(v^n) | =  \Big| \int_0^h b^\prime(v^n(t+s)) v^n_t(t+s)\,\mathrm{d}s\Big|
\leq h \|b^\prime(v^n) \|_{L^\infty(Q)}\|v^n_t\|_{L^\infty(Q)} \leq Ch,
\end{align*}
where \( C\) is independent of \( n \).  Using this and recalling  that \((\cdot, \cdot)_{\mathcal{A}_n(\bbU) }\) is an inner product on \( \mathbb{R}^{d\times d}\) for every \( \bbU \in \mathbb{R}^{d\times d}\), it follows that
\begin{equation}\label{a:equ19}
\begin{aligned}
\Big|\int_\Omega\Delta^h_t b(v^n) \tau^h_t \bbT^n : \Delta^h_t F_n(\bbT^n)\,\mathrm{d}x\Big|
&\leq Ch \int_\Omega\int_0^1 \big|\big(\tau^h_t \bbT^n, \Delta^h_t \bbT^n\big)_{
\mathcal{A}_n(s\tau^h_t \bbT^n + (1-s) \bbT^n) }\big| \,\mathrm{d}s\,\mathrm{d}x 
\\
&\leq 
\frac{h^2}{2\eta} \int_\Omega\int_0^1 (\tau^h_t \bbT^n, \tau^h_t \bbT^n)_{\mathcal{A}_n(s\tau^h_t \bbT^n + (1-s) \bbT^n)}\,\mathrm{d}s\,\mathrm{d}x
\\
&\quad
+ \frac{\eta}{2} \int_\Omega\int_0^1 \big( \Delta^h_t \bbT^n, \Delta^h_t \bbT^n\big)_{
\mathcal{A}_n(s\tau^h_t \bbT^n + (1-s) \bbT^n) }\,\mathrm{d}s\,\mathrm{d}x.
\end{aligned}
\end{equation}
Now we notice that there exists a constant \( C\), independent of \( n \), such that
\begin{align*}
|F_n(\bbT) - F_n(\bbS)|^2 \leq C(\bbT - \bbS) :(F_n(\bbT) - F_n(\bbS)) .
\end{align*}
Thus we have \( |\Delta^h_tF_n(\bbT^n)|^2 \leq C \Delta^h_t\bbT^n:\Delta^h_tF_n(\bbT^n) \). As a result, we can bound the term involving the external force in the following way:
\begin{equation}\label{a:equ20}
\begin{aligned}
|\langle\Delta^h_t l^n, \Delta^h_t(\bu^n_t + \alpha\bu^n) \rangle|&\leq C\|\Delta^h_tl^n\|_{-1,2}\|\Delta^h_t\beps(\bu^n_t + \alpha\bu^n) \|_2
\\
&\leq C\|\Delta^h_tl^n\|_{-1,2}^2 + \frac{\eta}{4}\int_\Omega \Delta^h_t\bbT^n:\Delta^h_t F_n(\bbT^n) \,\mathrm{d}x,
\end{aligned}
\end{equation}
where \( C\) is independent of \( n \). 
Substituting  (\ref{a:equ19}) and (\ref{a:equ20}) into (\ref{a:equ21}) and integrating  over \((t_1,t_2) \subset [0, T]\), we see that
\begin{align*}
&{\|\Delta^h_t \bu^n_t(t_2) \|_2^2} + \int_{t_1}^{t_2} \int_\Omega {b(v^n)} \Delta^h_t\bbT^n: \Delta^h_t F_n(\bbT^n) \,\mathrm{d}x\,\mathrm{d}t
\\
&
\leq C\Big[ \int_{t_1}^{t_2} \|\Delta^h_t l^n \|_{-1,2}^2 \,\mathrm{d}t + h^2 \int_{t_1}^{t_2}\int_\Omega\int_0^1 (\tau^h_t \bbT^n, \tau^h_t \bbT^n)_{\mathcal{A}_n(s\tau^h_t\bbT^n +(1-s)\bbT^n)} \,\mathrm{d}s\,\mathrm{d}x\,\mathrm{d}t
\\
&\quad
+ \int_{t_1}^{t_2} \|\Delta^h_t \bu^n_t(t) \|_2^2 \,\mathrm{d}t + \|\Delta^h_t \bu^n_t(t_1) \|_2^2 + \|\Delta^h_t\bu^n(t_1) \|_2^2 \Big].
\end{align*}
Dividing through by \( h^2\) and using standard properties of difference quotients, we get 
\begin{equation}\label{a:equ23}
\begin{aligned}
&\frac{\|\Delta^h_t\bu^n_t(t_2) \|_2^2}{h^2} + \int_{t_1}^{t_2} \int_\Omega b(v^n) \frac{\Delta^h_t \bbT^n }{h}: \frac{\Delta^h_t F_n(\bbT^n)}{h}\,\mathrm{d}x\,\mathrm{d}t
\\
&\leq 
C\Big[ \int_0^T \|l^n_t \|_{-1,2}^2 \,\mathrm{d}t +\int_{t_1}^{t_2}\int_\Omega\int_0^1 \big(\tau^h_t \bbT^n, \tau^n_t \bbT^n\big)_{\mathcal{A}_n(s\tau^h_t \bbT^n +(1-s) \bbT^n)}\,\mathrm{d}s\,\mathrm{d}x\,\mathrm{d}t
\\
&\quad
+ \sup_{t\in [0, T]}\|\bu^n_{tt}(t) \|_2^2 + \sup_{t\in [0, T]}\|\bu^n_t(t) \|_2^2 \Big]
\\
&\leq C\Big[ 1 + \int_{t_1}^{t_2}\int_\Omega\int_0^1 \big(\tau^h_t \bbT^n, \tau^n_t \bbT^n\big)_{\mathcal{A}_n(s\tau^h_t \bbT^n +(1-s) \bbT^n)}\,\mathrm{d}s\,\mathrm{d}x\,\mathrm{d}t\Big] ,
\end{aligned}
\end{equation}
where \( C\) is a constant, independent of \( n \) and \( h \). Next,  we take the limit as \( h \rightarrow 0+\). The strong convergence result \( \tau^h_t \bbT^n \rightarrow\bbT^n \) in \( L^2((t_1, t_2)\times\Omega)^{d\times d}\) and the dominated convergence theorem yield
\begin{equation}\label{a:equ22}
\begin{aligned}
\lim_{h\rightarrow 0+ }\int_{t_1}^{t_2}\int_\Omega\int_0^1 \big(\tau^h_t \bbT^n, \tau^n_t \bbT^n\big)_{\mathcal{A}_n(s\tau^h_t \bbT^n +(1-s) \bbT^n)}\,\mathrm{d}s\,\mathrm{d}x\,\mathrm{d}t &= \int_{t_1}^{t_2}\int_\Omega \big( \bbT^n, \bbT^n)_{\mathcal{A}_n(\bbT^n)} \,\mathrm{d}x\,\mathrm{d}t
\\
&\leq C\int_{t_1}^{t_2}\int_\Omega |\bbT^n| + \frac{|\bbT^n|^2}{n}\,\mathrm{d}x\,\mathrm{d}t
\\
&\leq C.
\end{aligned}
\end{equation}
Using (\ref{a:equ18}), Fatou's lemma and the monotonicity of \( F_n \), we   have
\begin{equation}\label{a:equ50}
\begin{aligned}
\liminf_{h\rightarrow 0+}\int_{t_1}^{t_2} \int_\Omega b(v^n) \frac{\Delta^h_t \bbT^n }{h}: \frac{\Delta^h_t F_n(\bbT^n)}{h}\,\mathrm{d}x\,\mathrm{d}t &\geq \int_{t_1}^{t_2}\int_\Omega b(v) \bbT^n_t :F_n(\bbT^n)_t\,\mathrm{d}x\,\mathrm{d}t
\\
& \geq \int_{t_1}^{t_2}\int_\Omega\eta \big( \bbT^n_t,\bbT^n_t)_{\mathcal{A}_n(\bbT^n) }\,\mathrm{d}x\,\mathrm{d}t
\\
&= \eta \int_{t_1}^{t_2}\int_\Omega \frac{|\bbT^n_t|^2}{(1 + |\bbT^n|^a)^{1 + \frac{1}{a}}} + \frac{|\bbT^n_t|^2 }{n}\,\mathrm{d}x\,\mathrm{d}t.
\end{aligned}
\end{equation}
Substituting  (\ref{a:equ22}) and (\ref{a:equ50}) into (\ref{a:equ23}) gives the required result.
\end{proof}

Noticing that we have \( |\Delta^h_t \beps(\bu^n_t + \alpha\bu^n) |\leq C\Delta^h_t \bbT^n :\Delta^h_t F_n(\bbT^n) \), as well as  the fact that \( \beps(\bu^n_t ) \in L^\infty(0, T; L^2(\Omega)^{d\times d}) \), we immediately obtain the bound in Corollary \ref{a:cor5}. Then in Lemma \ref{a:lem8}, we use the new bounds to improve the integrability of \(( \bbT^n)_n \) to \( L^\infty(0, T; L^1(\Omega)^{d\times d}) \). 

\begin{corollary}\label{a:cor5}
There exists a constant \( C\) independent of \( n \) such that
\begin{align*}
\int_0^T \|\bu^n_{tt}(t) \|_{1,2}^2 \,\mathrm{d}t \leq C. 
\end{align*}
\end{corollary}


\begin{lemma}\label{a:lem8}
Let the assumptions of Proposition \ref{a:prop2} hold and let \((\bu^n, \bbT^n, v^n) \) be the solution triple constructed there. Suppose additionally that
\( l\in W^{2,1}(0, T; W^{-1,2}_D(\Omega)^d) \), 
 the safety strain condition (\ref{a:equ5}) holds, the compatibility condition (\ref{a:equ4}) holds, and that \( \bu_1 + \alpha\bu_0 \in W^{2,2}_D(\Omega)^d\). There exists a constant \( C\) independent of \( n \) such that
 \begin{align*}
 \sup_{t\in [0, T]}\int_\Omega|\bbT^n(t) |\,\mathrm{d}x \leq C.
 \end{align*}
\end{lemma}

\begin{proof}
Let  \( t\in (0, T) \) and pick \( h > 0 \) sufficiently small such that \( t + h < T\).  We test in (\ref{a:equ9}) at times \( t\) and \( t+h \) against \( (\bu^n_t + \alpha\bu^n)(t)\) and subtract the resulting equations. Dividing through by \( h \), integrating over \((t_1,t_2) \subset [0, T]\) and letting \(h \rightarrow 0+\),  we deduce that
\begin{align*}
\int_{t_1}^{t_2} \langle l^n_t, \bu^n_t + \alpha\bu^n\rangle\,\mathrm{d}t &=\int_\Omega \bu^n_{tt}(t_2) \cdot(\bu^n_t + \alpha\bu^n)(t_2) - \bu^n_{tt}(t_1) \cdot (\bu^n_t + \alpha\bu^n)(t_1) \,\mathrm{d}x
\\
&\quad
+ \int_{t_1}^{t_2} \int_\Omega b^\prime(v^n) v^n_t \bbT^n : F_n(\bbT^n) + b(v^n) \bbT^n_t : F_n(\bbT^n) \,\mathrm{d}x\,\mathrm{d}t
\\
&\quad
- \int_{t_1}^{t_2}\int_\Omega\bu^n_{tt}\cdot (\bu^n_{tt}+ \alpha\bu^n_t) \,\mathrm{d}x\,\mathrm{d}t.
\end{align*}
From the chain rule and product rule for weak derivatives,  we have  \( \varphi_n(\bbT^n) \in W^{1,1}(0, T; L^1(\Omega)) \) with weak derivative given by \( \bbT^n_t :F_n(\bbT^n) \). Furthermore, since \(\bu^n_{tt}\in L^2(0, T; W^{1,2}(\Omega)^d) \), we must have  \( \bu^n_t + \alpha\bu^n \in C([0, T]; W^{1,2}_D(\Omega)^d) \) and so \( F_n(\bbT^n) \in C([0, T]; L^2(\Omega)^{d\times d}) \). By continuity of \( F^{-1}_n\), it follows that  \( \bbT^n \in C([0, T]; L^2(\Omega)^{d\times d}) \) with \(\bbT^n(0) = F^{-1}_n(\beps(\bu_1 + \alpha\bu_0)) \). Putting together these facts, we get
\begin{equation}\label{a:equ24}
\begin{aligned}
&\lim_{t_1 \rightarrow 0+}\int_{t_1}^{t_2}\int_\Omega b(v^n) \bbT^n_t:F_n(\bbT^n) \,\mathrm{d}x\,\mathrm{d}t
\\
&= \lim_{t_1 \rightarrow 0+} \Big\{ \Big[ \int_\Omega b(v^n(t)) \varphi_n(\bbT^n(t)) \,\mathrm{d}x\Big]_{t = t_1}^{t = t_2} - \int_{t_1}^{t_2} \int_\Omega b^\prime(v^n) v^n_t \varphi_n(\bbT^n) \,\mathrm{d}x\,\mathrm{d}t\Big\} 
\\
&= \int_\Omega  \big( b(v^n) \varphi_n(\bbT^n) \big)(t_2) - b(v_0)\varphi_n(F_n^{-1}(\beps(\bu_1 + \alpha\bu_0))) \,\mathrm{d}x - \int_{0}^{t_2} \int_\Omega b^\prime(v^n) v^n_t \varphi_n(\bbT^n) \,\mathrm{d}x\,\mathrm{d}t.
\end{aligned}
\end{equation}
Furthermore, we have
\begin{equation}\label{a:equ51}
\begin{aligned}
&\Big[ \int_\Omega b(v^n(t)) \varphi_n(\bbT^n(t)) \,\mathrm{d}x\Big]_{t = t_1}^{t = t_2}
\\
&\leq \Big[ \int_\Omega b(v^n(t)) \varphi_n(\bbT^n(t)) \,\mathrm{d}x\Big]_{t = t_1}^{t = t_2} - \int_{t_1}^{t_2} \int_\Omega b^\prime(v^n) v^n_t \varphi_n(\bbT^n) \,\mathrm{d}x\,\mathrm{d}t
\\
&= \int_{t_1}^{t_2} \langle l^n_t, \bu^n_t + \alpha\bu^n\rangle\,\mathrm{d}t -  \int_\Omega \bu^n_{tt}(t_2) \cdot(\bu^n_t + \alpha\bu^n)(t_2) - \bu^n_{tt}(t_1) \cdot (\bu^n_t + \alpha\bu^n)(t_1) \,\mathrm{d}x
\\&\quad
- \int_{t_1}^{t_2} \int_\Omega b^\prime(v^n) v^n_t \bbT^n : F_n(\bbT^n)   \,\mathrm{d}x\,\mathrm{d}t
+ \int_{t_1}^{t_2}\int_\Omega\bu^n_{tt}\cdot (\bu^n_{tt}+ \alpha\bu^n_t) \,\mathrm{d}x\,\mathrm{d}t
\\
& \leq C\Big( \int_0^T  \|l^n_t\|_{-1,2}^2 \,\mathrm{d}t + \sup_{t\in [0, T]}\|\bu^n_{tt}(t) \|_2^2  + \sup_{t\in [0, T]}\|(\bu^n_t + \alpha\bu^n)(t)\|_{1,2}^2  + \sup_{t\in [0, T]} \|\bu^n_t(t) \|_2^2
\\
&\quad+ \|v^n_t\|_\infty \int_Q \Big[ \frac{|\bbT^n|^2}{n} + |\bbT^n|\Big] \,\mathrm{d}x\,\mathrm{d}t\Big)
\\
&\leq C,
\end{aligned}
\end{equation}
where \( C\) is a constant that is independent of \( n \), \( t_1\) and \( t_2\). 
Letting \( t_1 \rightarrow 0+ \) in (\ref{a:equ51}) and using (\ref{a:equ24}) again, it follows that 
\begin{equation}\label{a:equ52}
\int_\Omega\varphi_n(\bbT^n(t_2)) \,\mathrm{d}x \leq C\Big( 1  + \int_\Omega b(v_0) \varphi_n(F^{-1}_n(\beps(\bu_1 + \alpha\bu_0)))\,\mathrm{d}x\Big)  \leq C.
\end{equation}
The second inequality  follows from the safety strain condition (\ref{a:equ5}). Using the definition of \( \varphi_n \), there exists positive constants \( c_1\) and \( c_2\), independent of \( n \), such that   \( c_1|\bbT | - c_2 \leq \varphi_n(\bbT) \),   for every \( \bbT \in \mathbb{R}^{d\times d}\). Substituting this into  (\ref{a:equ52}), the asserted bound follows. 
\end{proof}

Next, we improve the spatial regularity of the stress tensor. We  use arguments mimicking those in \cite{RN6} and \cite{preprint2}. However, now we have the extra issue of the phase-field function being present. It is here that we  require \(\nabla v^n \) to be uniformly bounded in \( L^\infty(Q) \), independent of \( n \). Thus we assume that \( k > \frac{d}{2}+1\) to ensure that we have   \( H^k(\Omega) \subset C^1(\overline{\Omega}) \). 

\begin{proposition}\label{a:prop4}
Let the assumptions of Proposition \ref{a:prop2} hold and let \((\bu^n, \bbT^n, v^n) \) be the solution triple constructed there. Suppose also  that
\( l\in W^{2,1}(0, T; W^{-1,2}_D(\Omega)^d) \),  \( \boldf \in L^2(0, T; W^{1,2}_{loc}(\Omega)^d) \),
 the safety strain condition (\ref{a:equ5}) holds, the compatibility condition (\ref{a:equ4}) holds and  \( \bu_1 + \alpha\bu_0 \in W^{2,2}_D(\Omega)^d\) with \( \bu_0 \in W^{2,2}_{loc}(\Omega)^d\). Furthermore, assume \( k > \frac{d}{2}+1\). For every compact  subset \( \Omega_0 \subset\Omega\), there exists a constant \( C = C(\Omega_0) \), independent of \( n \), such that
 \begin{align*}
\sup_{t\in [0, T]} \| \nabla^2 \bu^n (t) \|_{L^2(\Omega_0)} +   \int_0^T \int_{\Omega_0} \frac{|\nabla\bbT^n|^2}{(1 + |\bbT^n|)^{1+a}} + \frac{|\nabla \bbT^n|^2}{n} + |\nabla^2\bu^n_t(t) |^2 \,\mathrm{d}x\,\mathrm{d}t \leq C.
 \end{align*}
\end{proposition}

\begin{proof}
Let \( \zeta \in C^\infty_c(\Omega) \) be a cut-off function between \( \Omega\) and \( \Omega_0\). We  denote by \( \Omega_1\), a compact  subset of \( \Omega\)  that  contains \( \mathrm{supp}(\zeta) \).  Suppose that \( h > 0 \) is sufficiently small so that \( \zeta\Delta^h_i \bw\) is well-defined on \( \Omega\) for an arbitrary \( \bw\in L^2(\Omega)^d \), where \( \Delta^h_i \) is the undivided  difference quotient of length \( h \) in the \(i\)-th spatial co-ordinate direction. For a.e. \( t\in (0, T) \), noticing that  the Neumann boundary term vanishes due to  the compact support of \( \zeta\),  using Proposition \ref{a:prop3} we get
\begin{align*}
&\int_\Omega\zeta^2 \Delta^h_i\bu^n_{tt}\cdot \Delta^h_i (\bu^n_t + \alpha\bu^n) + \Delta^h_i \big( b(v^n) \bbT^n\big) :\beps\big(\zeta^2 \Delta^h_i (\bu^n_t + \alpha\bu^n) \big) \,\mathrm{d}x  
\\
&= \int_\Omega\zeta^2\Delta^h_i \boldf(t) \cdot \Delta^h_i (\bu^n_t  +\alpha\bu^n) \,\mathrm{d}x
\\
&\leq C(\zeta) h^2 \|\nabla \boldf(t) \|_{L^2(\Omega_1)} \|\nabla (\bu^n_t + \alpha\bu^n)(t) \|_{L^2(\Omega_1)}
\\
&\leq Ch^2g(t),
\end{align*}
where \( g\in L^2(0, T; \mathbb{R}_+) \) with bound in this space independent of \( n \) but depending on \( \Omega_1\). 
 Similarly, we have
\begin{align*}
\int_\Omega\zeta^2 \Delta^h_i \bu^n_{tt} \cdot \Delta^h_i (\bu^n_t + \alpha\bu^n)\,\mathrm{d}x &\leq C(\zeta) h^2\|\nabla \bu^n_{tt}(t) \|_{L^2(\Omega_1)}\|\nabla(\bu^n_t + \alpha\bu^n) (t) \|_{L^2(\Omega_1)}
\\
&\leq C(\zeta) h^2\|\nabla \bu^n_{tt}(t) \|_{L^2(\Omega_1)}
\\
&\leq C(\zeta) h^2 g(t), 
\end{align*}
making \( g \) larger if necessary. 
The right-hand side is in \(L^2(0, T) \) by Corollary \ref{a:cor5}, with bound independent of \( n \). For the nonlinear terms, we expand as follows:
\begin{align*}
&\int_\Omega \Delta^h_i\big( b(v^n) \bbT^n\big) :\beps\big( \zeta^2 \Delta^h_i (\bu^n_t + \alpha\bu^n) \big) \,\mathrm{d}x
\\
&= \int_\Omega \Big( \Delta^h_i b(v^n) \tau^h_i \bbT^n + b(v^n) \Delta^h_i \bbT^n \Big)  :\Big( 2\zeta\nabla \zeta \otimes \Delta^h_i (\bu^n_t + \alpha\bu^n) + \zeta^2 \beps(\Delta^h_i (\bu^n_t + \alpha\bu^n)) \Big) \,\mathrm{d}x,
\end{align*}
where \( \tau^h_i \) is a translation of length \( h \) in the \(i\)-th spatial co-ordinate direction. 
Integration over \((0, T) \) yields
\begin{align*}
&\int_Q \eta \zeta^2 \frac{|\Delta^h_i \bbT^n|^2}{n}\,\mathrm{d}x\,\mathrm{d}t
\\
&\leq \int_Q \zeta^2 b(v^n)  \Delta^h_i \bbT^n : \Delta^h_i F_n(\bbT^n ) \,\mathrm{d}x\,\mathrm{d}t
\\
&\leq 
Ch^2  \Big[ 1 + \int_0^T g(t)  \,\mathrm{d}t\Big] + \int_Q \Big| 2\zeta \Delta^h_i b(v^n) \tau^h_i \bbT^n :\big( \nabla \zeta \otimes \Delta^h_i (\bu^n_t + \alpha\bu^n) \big) \Big|\,\mathrm{d}x\,\mathrm{d}t
\\
&\quad
+ \int_Q \Big| \zeta^2\Delta^h_i  b(v^n) \tau^h_i \bbT^n : \Delta^h_i  F_n(\bbT^n) \Big|\,\mathrm{d}x\,\mathrm{d}t +\int_Q \Big| 2\zeta b(v^n) \Delta^h_i \bbT^n :\big( \nabla \zeta \otimes \Delta^h_i (\bu^n_t +\alpha\bu^n) \big)\Big|\,\mathrm{d}x\,\mathrm{d}t .
\end{align*}
Now we use this to show that \( \bbT^n \in L^2(0, T; W^{1,2}_{loc}(\Omega)^{d\times d}) \) with bound that does depend on \( n\). Using Young's inequality and the fact that, for every \( \bbT \), \( \bbS\in \mathbb{R}^{d\times d}\),    \(|F_n(\bbT) - F_n(\bbS) |\leq 3|\bbT - \bbS | \), we see that 
\begin{align*}
\frac{\eta}{2}\int_Q \zeta^2\frac{|\Delta^h_i \bbT^n|^2}{n}\,\mathrm{d}x\,\mathrm{d}t
&\leq 
C(n,\zeta) \Big[ h^2 + \int_0^T\Big\{ \||\nabla \zeta| \Delta^h_i  (\bu^n_t + \alpha\bu^n) \|_2^2 + h^2 \|\nabla v^n \|_{L^\infty(\Omega_1)}^2 \|\zeta\tau^h_i \bbT^n\|_2^2 
\\
&\quad
+ h \|\nabla v^n \|_{L^\infty(\Omega_1)}\|\zeta \tau^h_i \bbT^n \|_2\||\nabla\zeta|\Delta^h_i (\bu^n_t + \alpha\bu^n) \|_2\Big\} \,\mathrm{d}t\Big]
\\
& \leq C(n,\zeta) h^2,
\end{align*}
using the \( L^\infty(0, T; W^{1,\infty}(\Omega)) \) bound on \( v^n\), which is a consequence of the Sobolev embedding \( H^k(\Omega) \subset W^{1,\infty}(\Omega) \). We divide through by \( h^2\) and let \( h \rightarrow 0+ \) in order to deduce that \( \bbT^n \) is an element of \( L^2(0, T; W^{1,2}_{loc}(\Omega)^{d\times d}) \). 
As a result, we have  \( \beps(\bu^n_t + \alpha\bu^n) \in L^2(0, T; W^{1,2}_{loc}(\Omega)^{d\times d}) \), from which we deduce that    \( \bu^n_t + \alpha\bu^n \in L^2(0, T; W^{2,2}_{loc}(\Omega)^{d}) \), using the Korn--Poincar\'{e} inequality.

From the elastodynamic equation (\ref{a:equ9}) and this extra regularity, it follows that
\begin{equation}\label{a:equ25}
\bu^n_{tt} = \diver\big( b(v^n) \bbT^n\big) + \boldf\quad\text{ pointwise a.e. in }Q,
\end{equation}
with each term  an element of \( L^2(0, T; L^2_{loc}(\Omega)^d) \). We take the scalar product of (\ref{a:equ25}) with the function \( \zeta^2\nabla \cdot \nabla(\bu^n_t + \alpha\bu^n) \) and integrate over \( \Omega\) to obtain
\begin{equation}\label{a:equ27}
\begin{aligned}
&\int_\Omega \zeta^2\bu^n_{tt}\cdot \big( \nabla\cdot \nabla (\bu^n_t + \alpha\bu^n ) \big) \,\mathrm{d}x
\\
&= \int_\Omega\zeta^2 \diver\big( b(v^n) \bbT^n\big) \cdot \big( \nabla\cdot \nabla (\bu^n_t + \alpha\bu^n ) \big) + \zeta^2 \boldf \cdot \big( \nabla\cdot \nabla (\bu^n_t + \alpha\bu^n ) \big)\,\mathrm{d}x.
\end{aligned}
\end{equation}
First we focus on the term containing the nonlinearity \( \bbT^n \). We have
\begin{align*}
&\int_\Omega\zeta^2 \diver\big( b(v^n) \bbT^n\big) \cdot \big( \nabla\cdot \nabla (\bu^n_t + \alpha\bu^n ) \big) \,\mathrm{d}x
\\
&= \int_\Omega\zeta^2b(v^n) \frac{\partial T^n_{ij}}{\partial x_j }\frac{\partial^2}{\partial x_k^2} \big( \bu^n_t + \alpha\bu^n\big)_i + \zeta^2 b^\prime(v^n) T^n_{ij} \frac{\partial v^n}{\partial x_j }\frac{\partial^2}{\partial x_k^2 }\big( \bu^n_t + \alpha\bu^n\big)_i\,\mathrm{d}x.
\end{align*}
Using an integration by parts argument, we write 
\begin{align*}
&\int_\Omega\zeta^2b(v^n) \frac{\partial T^n_{ij}}{\partial x_j }\frac{\partial^2}{\partial x_k^2} \big( \bu^n_t + \alpha\bu^n\big)_i \,\mathrm{d}x
\\
&= \int_\Omega \Big\{ \zeta^2 b(v^n) \frac{\partial T^n_{ij}}{\partial x_k} \frac{\partial}{\partial x_k}\beps(\bu^n_t + \alpha\bu^n)_{ij} 
- T^n_{ij} \frac{\partial^2}{\partial x_k^2} \big(\bu^n_t + \alpha\bu^n\big)_i \Big( b^\prime(v^n) \frac{\partial v^n}{\partial x_j}\zeta^2 + 2\zeta \frac{\partial\zeta}{\partial x_j} b(v^n) \Big)
\\&\quad
+ T^n_{ij} \frac{\partial}{\partial x_k} \beps(\bu^n_t + \alpha\bu^n)_{ij} \Big( b^\prime(v^n) \frac{\partial v^n}{\partial x_k}\zeta^2 + 2\zeta \frac{\partial\zeta}{\partial x_k }b(v^n) \Big)\Big\} \,\mathrm{d}x.
\end{align*}
This yields 
\begin{equation}\label{a:equ26}
\begin{aligned}
&\int_\Omega \zeta^2 \diver\big( b(v^n)\bbT^n \big) \cdot \big( \nabla \cdot (\nabla (\bu^n_t + \alpha\bu^n)) \big) \,\mathrm{d}x
\\
&= \int_\Omega\zeta^2 b(v^n)\nabla\bbT^n\vdots \nabla F_n(\bbT^n) + T^n_{ij}\frac{\partial^2}{\partial x_k^2}(\bu^n_t + \alpha\bu^n)_i \Big( \zeta^2 b^\prime(v^n)\frac{\partial v^n }{\partial x_j }\Big) 
\\
&\quad
- T^n_{ij}\frac{\partial^2 }{\partial x_k^2} (\bu^n_t + \alpha\bu^n)_i \Big( b^\prime(v^n) \frac{\partial v^n}{\partial x_j}\zeta^2 + 2\zeta \frac{\partial\zeta}{\partial x_j } b(v^n) \Big) 
\\
&\quad
 + T^n_{ij}\frac{\partial}{\partial x_k} \beps(\bu^n_t + \alpha\bu^n)_{ij} \Big( b^\prime(v^n) \frac{\partial v^n}{\partial x_k}\zeta^2 + 2\zeta \frac{\partial\zeta}{\partial x_k }b(v^n)\Big) \,\mathrm{d}x
 \\
 &= \int_\Omega\zeta^2 b(v^n) \nabla \bbT^n \vdots \nabla F_n(\bbT^n) + \zeta \frac{\partial}{\partial x_k}\beps(\bu^n_t + \alpha\bu^n)_{ij}B^{n,k}_{ij}\,\mathrm{d}x,
\end{aligned}
\end{equation}
where \( \mathbf{B}^{n,k} = (B^{n,k}_{ij})_{i,j=1}^d \) is defined, for  \( 1\leq k \leq d\),  by
\begin{align*}
B^{n,k}_{ij} &= T^n_{ij}\big( 2b(v^n) \partial_k \zeta + \zeta b^\prime(v^n) \partial_k v^n \Big) - 2\delta_{jk}T^n_{im}\Big( 2b(v^n) \partial_m \zeta + \tau b^\prime(v^n) \partial_m v^n \Big)
\\
&\quad
+ \delta_{ij} T^n_{km} \Big( 2b(v^n) \partial_m \zeta + \tau b^\prime(v^n) \partial_m v^n \Big) + 2\delta_{jk}T^n_{im}\zeta b^\prime(v^n) \partial_m v^n - \delta_{ij}T^n_{km} \zeta b^\prime(v^n) \partial_m v^n ,
\end{align*}
making use of the identity
\begin{align*}
\frac{\partial^2}{\partial x_k^2} \bw_i = 2\frac{\partial}{\partial x_k}\beps(\bw)_{ik} - \frac{\partial}{\partial x_i }\beps(\bw)_{kk}.
\end{align*}
Using the bound on \((v^n)_n \) in \(L^\infty(0, T; W^{1,\infty}(\Omega)) \) and the definition of \( \mathbf{B}^{n,k}\),  there exists a constant \( C\), independent of \( n \), such that \(|\mathbf{B}^{n,k}|\leq C|\bbT^n| \), for every \( 1\leq k \leq d\).  
Furthermore, we have 
\begin{align*}
&\Big|\int_\Omega\zeta \frac{\partial}{\partial x_k }\beps(\bu^n_t + \alpha\bu^n) :\mathbf{B}^{n,k}\,\mathrm{d}x\Big| 
\\
&= \Big|\int_\Omega \zeta \frac{\partial}{\partial x_k } F_n(\bbT^n) : \mathbf{B}^{n,k}\,\mathrm{d}x\Big|
\\
&= \Big|\int_\Omega\zeta \big( \partial_k \bbT^n, \mathbf{B}^{n,k}\big)_{\mathcal{A}_n(\bbT^n)} \,\mathrm{d}x\Big|
\\
& \leq \int_\Omega\frac{\zeta^2 \eta}{2} \big( \partial_k \bbT^n, \partial_k\bbT^n\big)_{\mathcal{A}_n(\bbT^n)} + C(\eta) \mathbb{I}_{\mathrm{supp}(\zeta)}\big( \mathbf{B}^{n,k},\mathbf{B}^{n,k}\big)_{\mathcal{A}_n(\bbT^n)}\,\mathrm{d}x
\\
&\leq \int_\Omega\frac{\zeta^2 \eta}{2} \big( \partial_k \bbT^n, \partial_k\bbT^n\big)_{\mathcal{A}_n(\bbT^n)} + C \Big( \frac{|\mathbf{B}^{n,k}|^2}{n} + \frac{|\mathbf{B}^{n,k}|^2}{1 + |\bbT^n|}\Big) \,\mathrm{d}x
\\
&\leq \int_\Omega\frac{\zeta^2 \eta}{2} \nabla \bbT^n \vdots \nabla F_n(\bbT^n) + C\Big( \frac{|\bbT^n|^2}{n} + |\bbT^n|\Big) \,\mathrm{d}x.
\end{align*}
Using this  in (\ref{a:equ26}) and substituting the result into  (\ref{a:equ27}),  we obtain
\begin{equation}\label{a:equ54}
\begin{aligned}
&\int_\Omega\zeta^2 \Big( \bu^n_{tt} - \boldf \Big) \cdot \Big( \nabla \cdot \big( \nabla(\bu^n_t + \alpha\bu^n) \big) \Big) \,\mathrm{d}x + C\int_\Omega\frac{|\bbT^n|^2}{n} + |\bbT^n |\,\mathrm{d}x
\\
&\geq \int_\Omega\frac{\zeta^2 b(v^n) }{2}\nabla \bbT^n \vdots \nabla F_n(\bbT^n ) \,\mathrm{d}x.
\end{aligned}
\end{equation}
For  the first term on the right-hand side, we integrate by parts to deduce that 
\begin{equation}\label{a:equ53}
\begin{aligned}
&\int_\Omega \zeta^2 \Big( \bu^n_{tt} - \boldf \Big) \cdot \Big( \nabla \cdot \big( \nabla(\bu^n_t + \alpha\bu^n) \big) \Big) \,\mathrm{d}x
\\
&= 
\int_\Omega \zeta^2 \nabla \big( \boldf - \bu^n_{tt}\big) :\nabla \big( \bu^n_t + \alpha\bu^n \big) + 2\zeta \big[ \nabla \zeta \otimes \big( \boldf- \bu^n_{tt}\big) \big] :\nabla \big( \bu^n_t + \alpha\bu^n\big) \,\mathrm{d}x
\\
& = 
-\frac{\mathrm{d}}{\mathrm{d}t}\Big( \frac{\|\zeta \nabla \bu^n_t\|_2^2}{2} + \alpha\int_\Omega\zeta^2 \nabla \bu^n_t : \nabla \bu^n \,\mathrm{d}x\Big)  + \alpha\int_\Omega\zeta^2|\nabla \bu^n_t |^2 \,\mathrm{d}x 
\\&\quad
+ \| \zeta \nabla \boldf \|_2 \|\zeta \nabla (\bu^n_t + \alpha\bu^n)\|_2 
+ C(\zeta) \|\boldf\|_2 \|\nabla (\bu^n_t + \alpha\bu^n) \|_2
\\&\quad
+ C(\zeta) \|\bu^n_{tt}\|_2 \|\nabla (\bu^n_t + \alpha\bu^n) \|_2.
\end{aligned}
\end{equation}
Integrating in time, using the continuity   of \( \nabla \bu^n\) and \( \nabla \bu^n_t \) as maps from \([0, T]\) to \( L^2(\Omega)^{d\times d}\), with the assumptions on the initial data, and applying Gronwall's inequality, it follows from (\ref{a:equ54}) and (\ref{a:equ53}) that
\begin{equation}\label{a:equ62}
\sup_{t\in [0, T]}\|\zeta\nabla \bu^n_t(t) \|_2^2 + \int_Q \zeta^2 \Big( \frac{|\nabla \bbT^n|^2}{( 1 + |\bbT^n|)^{1 + a}} + \frac{|\nabla \bbT^n|^2}{n}\Big) \,\mathrm{d}x\,\mathrm{d}t \leq C(\zeta). 
\end{equation}
Recalling that for every \( \bbT\), \( \bbS \in \mathbb{R}^{d\times d}\) we have
\begin{align*}
|F_n(\bbT) - F_n(\bbS) |^2 \leq C (\bbT - \bbS) : (F_n(\bbT) - F_n(\bbS)),
\end{align*}
where \( C\) is a constant independent of \( n \), it follows from (\ref{a:equ62}) that
\begin{align*}
\int_0^T\int_{\Omega_0} |\beps(\partial_i (\bu^n_t + \alpha\bu^n)) |^2\,\mathrm{d}x\,\mathrm{d}t \leq \int_0^T\int_{\Omega_0}|\nabla (\beps(\bu^n_t + \alpha\bu^n)) |^2  \leq C(\Omega_0),
\end{align*}
for every \( 1\leq i \leq d\). Applying   Korn's inequality for functions in the space \( W^{1,2}(\Omega_0)^2\), we deduce that
\begin{align*}
\int_0^T \|\partial_i (\bu^n_t + \alpha\bu^n)\|_{W^{1,2}(\Omega_0)}^2\,\mathrm{d}t 
&\leq 
C\int_0^T \|\partial_i (\bu^n_t + \alpha\bu^n)\|_{L^2(\Omega_0)}^2 + \|\beps(\partial_i(\bu^n_t + \alpha\bu^n)) \|_{L^2(\Omega)}^2\,\mathrm{d}t
\\
& \leq 
C\int_0^T \|\nabla \bu^n_t \|_2^2 + \|\nabla \bu^n \|_2^2 + \|\nabla(\beps(\bu^n_t + \alpha\bu^n) ) \|_{L^2(\Omega_0)}^2 \,\mathrm{d}t
\\
&\leq C(\Omega_0).
\end{align*}
Hence   \( \bu^n_t + \alpha\bu^n \in L^2(0, T; W^{2,2}_{loc}(\Omega)^d) \) with bound  independent of \( n \). It remains to prove that \(\bu^n\in L^\infty(0, T; W^{2,2}_{loc}(\Omega)^d) \). From the memory kernel property, we have 
\begin{align*}
\bu^n(t) = \mathrm{e}^{-\alpha t}\bu_0 + \int_0^t \mathrm{e}^{\alpha(s-t)} (\bu^n_t + \alpha\bu^n)\,\mathrm{d}s,
\end{align*}
with equality meant in \( L^2(\Omega)^d\). Each term on the right-hand side is an element of \( W^{2,2}_{loc}(\Omega)^d\). Hence, for a.e. \( t\in (0, T) \), we must have \( \bu^n(t) \in W^{2,2}_{loc}(\Omega)^d\). Furthermore, we have
\begin{align*}
\nabla^2\bu^n(t) = \mathrm{e}^{-\alpha t}\nabla^2\bu_0 + \int_0^t \mathrm{e}^{\alpha(s-t)} \nabla^2(\bu^n_t + \alpha\bu^n)\,\mathrm{d}s,
\end{align*}
with equality in \(L^2_{loc}( \Omega)^{d\times d}\).  
Applying the \( L^2(\Omega_0)\) norm to this   and using the previously determined result for \( \bu^n_t + \alpha\bu^n\), we conclude the stated result.
\end{proof}

Now we have sufficient \(n\)-independent estimates to take the limit as \( n \rightarrow\infty\). In Theorem \ref{a:thm2}, we focus on proving a pointwise convergence result for \((\bbT^n)_n \) and show that the limiting triple satisfies the weak elastodynamic equation, up to the Neumann boundary condition. We also prove the satisfaction of the initial conditions. We delay the proof of the minimisation problem   and the energy-dissipation equality  until Propositions \ref{a:prop5} and \ref{a:prop6}, respectively, for simplicity of the presentation.

\begin{theorem}\label{a:thm2}
Assume that \( k > \frac{d}{2} + 1\) with data \(v_0 \in H^{k}_{D+1}(\Omega) \), and \(\bu_0 \), \( \bu_1 \in W^{1,2}_D(\Omega)^d\) such that \( \bu_0 \in W^{2,2}_{loc}(\Omega)^d\), \(\bu_1 + \alpha\bu_0 \in W^{2,2}_D(\Omega)^d\) and the safety strain condition (\ref{a:equ5}) holds. Furthermore, assume that  \( l \in W^{2,1}(0, T; W^{-1,2}_D(\Omega)^d) \) such that \( \boldf\in L^2(0, T; W^{1,2}_{loc}(\Omega)^d) \) and   the compatibility condition (\ref{a:equ4}) holds. Let \(((\bu^n, \bbT^n, v^n))_n \) be the sequence of solutions to the regularised problem constructed in Proposition \ref{a:prop2}. There exists a subsequence in \( n \), not relabelled, and a limiting triple \((\bu,\bbT, v) \) such that
\begin{itemize}
\item \( \bu^n \overset{\ast}{\rightharpoonup}\bu\) weakly-* in \( W^{2,\infty}(0, T; L^2(\Omega)^d) \cap W^{1,\infty}(0, T; W^{1,2}_D(\Omega)^d) \),
\item \( \bu^n \rightharpoonup \bu \) weakly in  \(W^{2,2}(0, T; W^{1,2}_D(\Omega)^d) \cap W^{1,2}(0, T; W^{2,2}_{loc}(\Omega)^d) \),
\item \(\beps(\bu^n_t + \alpha\bu^n) \rightarrow\beps(\bu_t + \alpha\bu) \) pointwise a.e. in \( Q\),
\item \( v^n \overset{\ast}{\rightharpoonup} v\) weakly-* in \( W^{1,\infty}(0, T; H^k(\Omega)) \) with \( v(t) \in H^k_{D+1}(\Omega) \) for every \( t\in [0, T]\),
\item \( \bbT^n \rightarrow\bbT \) pointwise a.e. in \( Q\) where \( \bbT \in L^\infty(0, T; L^1(\Omega)^{d\times d}) \).
\end{itemize}
Furthermore, there exists an error term \( \tilde{\bg}\in L^\infty_{w^*}(0, T; (C^1_0(\Gamma_N)^d)^*) \) such that 
\begin{equation}\label{a:equ30}
\int_\Omega\bu_{tt}(t)\cdot \bw  + b(v(t)) \bbT(t) :\beps(\bw) \,\mathrm{d}x = \int_\Omega\boldf(t) \cdot \bw \,\mathrm{d}x + \langle \bg(t) - \tilde{\bg}(t), \bw\rangle_{C_0^1(\Gamma_N)},
\end{equation}
for a.e. \( t\in (0, T) \) and every \( \bw \in C^1_D(\overline{\Omega})^d\), where the stress tensor is identified by the relation \( \beps(\bu_t+ \alpha\bu) = F(\bbT ) \), pointwise a.e. in \( Q\).   Furthermore, the initial conditions hold in the sense that
\begin{equation}\label{a:equ31}
\lim_{t\rightarrow 0+} \Big[ \|\bu(t) - \bu_0\|_{1,2} + \|\bu_t(t) - \bu_1 \|_{1,2} + \|v(t) - v_0\|_{k,2} \Big] = 0.
\end{equation}
\end{theorem}

\begin{proof}
The first, second and fourth convergence results follow immediately from the {\it a priori} bounds in Lemma \ref{a:lem5}, Lemma \ref{a:lem6}, Lemma \ref{a:lem7}, Proposition \ref{a:prop3} and Corollary \ref{a:cor5}. Using Propositions \ref{a:prop3} and \ref{a:prop4}, applying the reasoning in \cite{mypaperpreprint} yields the pointwise convergence result \(\bbT^n \rightarrow\bbT \) a.e. in \( Q\). In particular, we consider sequences 
\[
\bbS^n = \frac{\bbT^n}{(1 + |\bbT^n|)^{1+a}}, \quad s^n = \frac{1 }{(1 + |\bbT^n|)^{1 + a}}.
\]
By a careful application of the chain rule for weak derivatives, it follows that \((\bbS^n)_n \) is bounded in \( W^{1,2}(0, T; L^2(\Omega)^{d\times d}) \) and \( L^2(0, T; W^{1,2}_{loc}(\Omega)^{d\times d}) \), independent of \( n \). An analogous result holds for \((s^n)_n\). By the Aubin--Lions lemma, it follows that \( \bbS^n \rightarrow\bbS \) and \( s^n \rightarrow s\) strongly in \(L^2(0,  T; L^2_{loc}(\Omega)) \) and hence pointwise a.e. in \( Q\), up to a subsequence that we do not relabel. By Fatou's lemma and the bound on \((\bbT^n)_n \) in \( L^1(Q)^{d\times d}\), we deduce that \( s^{-\frac{1}{1+a}}\in L^1(Q) \). Hence \( s> 0 \) a.e. in \( Q\) and so \( \bbT^n = (s^n)^{-1}\bbS^n \) converges pointwise a.e. in \( Q\) to a limit which we denote by \( \bbT \). 
By Fatou's lemma and Lemma \ref{a:lem8},  \( \bbT\in L^\infty(0, T; L^1(\Omega)^{d\times d}) \). The fifth bullet point now follows. 

By Lemma \ref{a:lem5},  we have that \(n^{-1}\bbT^n \rightarrow\mathbf{0}\) strongly in \( L^2(Q)^{d\times d}\). With the constitutive relation for the regularised problem and the pointwise convergence of \((\bbT^n)_n \), this implies the pointwise convergence result for \((\beps(\bu^n_t + \alpha\bu^n))_n \), noting that weak limits in \( L^2(Q)\) and pointwise limits   coincide. Hence all of the stated convergence results  hold.
Furthermore,   the  strain-limiting  constitutive relation between \( \bbT \) and \( \bu \) must hold by uniqueness of pointwise limits. 

The attainment of the initial conditions (\ref{a:equ31}) follows from standard arguments, using  the aforementioned regularity, application of the Aubin--Lions lemma  and the initial conditions satisfied by the sequence of approximate solutions. 

Adapting the reasoning in \cite{preprint2}, noticing that the Neumann boundary term vanishes when considering compactly supported test functions, and using the fact that \((b(v^n))_n \) is uniformly bounded on \( Q\) and converges strongly in \( L^p(Q)\) to \( b(v) \) for every \(  p \in [1,\infty) \), we see that, for every \( \bw\in W^{1,2}_0(\Omega)^d\) and a.e. \( t\in (0, T) \), 
\begin{equation}\label{a:equ28}
\int_\Omega\bu_{tt}(t) \cdot \bw  + b(v(t)) \bbT(t) :\beps(\bw)\,\mathrm{d}x  = \langle l(t), \bw\rangle = \int_\Omega\boldf(t)\cdot \bw\,\mathrm{d}x.
\end{equation}

Next, we want to be able to test against functions that only vanish on the Dirichlet part of the boundary, and so obtain some information regarding the satisfaction of the  Neumann boundary condition. 
First we notice that \( (\bbT^n)_n \) is bounded in \(L^\infty(0, T; L^1(\Omega)^{d\times d}) \) and hence in \( L^\infty_{w^*}(0, T; \mathcal{M}(\overline{\Omega})^{d\times d}) \),   the dual space of \( L^1(0, T; C(\overline{\Omega})^{d\times d}) \). Up to a further subsequence that we do not relabel, it follows from the Banach--Alaoglu theorem that we have \( \bbT^n \overset{\ast}{\rightharpoonup}\bbT \) weakly-* in \( L^\infty_{w^*}(0, T; \mathcal{M}(\overline{\Omega})^{d\times d}) \). Furthermore, this limit satisfies
\begin{equation}\label{a:equ29}
\int_\Omega \bu_{tt}(t) \cdot \bw \,\mathrm{d}x + \langle \overline{\bbT}(t), \beps(\bw) \rangle_{(\mathcal{M}(\overline{\Omega}),  C(\overline{\Omega}))} = \langle l(t), \bw\rangle,
\end{equation}
for every \( \bw \in C^1_D(\overline{\Omega})^d\) and a.e. \( t\in (0, T) \), where \( \langle \cdot, \cdot\rangle_{(\mathcal{M}(\overline{\Omega}),  C(\overline{\Omega}))} \) denotes the duality pairing between \( \mathcal{M}(\overline{\Omega}) \) and \( C(\overline{\Omega})\). In the spirit of \cite{RN6},  to define the error on \( \Gamma_N \), we define the normal component the stresses \( \bbT(t) \) and \( \overline{\bbT}(t)\) on the Neumann boundary, respectively, by
\begin{align*}
\langle \bbT(t)\mathbf{n}, \bw \rangle|_{C^1_0(\Gamma_N)} = \int_\Omega\bu_{tt}(t) \cdot \tilde{\bw} + \bbT(t) :\beps(\tilde{\bw}) \,\mathrm{d}x,
\end{align*}
and 
\begin{align*}
\langle\overline{\bbT}(t) \mathbf{n},\bw \rangle_{C^1_0(\Gamma_N)} = \int_\Omega\bu_{tt}(t) \cdot \tilde{\bw} \,\mathrm{d}x + \langle \overline{\bbT}(t), \tilde{\bw}\rangle_{(\mathcal{M}(\overline{\Omega}),  C(\overline{\Omega}))},
\end{align*}
for a.e. \( t\in (0, T) \) and every \( \bw \in C^1_0(\Gamma_N)^d\), where  \( \tilde{\bw}\in C^1_D(\overline{\Omega})^d\) is an extension of \( \bw \) to \( \overline{\Omega}\) such that it vanishes on \( \Gamma_D\). 
Such an extension exists by the definition of \( C^1_0(\Gamma_N)^d\). 
Using (\ref{a:equ28}) and  (\ref{a:equ29}), the definitions of \( \bbT(t)\mathbf{n}\) and \( \overline{\bbT}(t) \mathbf{n}\) as above are independent of the choice of extension \( \tilde{\bw}\). 
Furthermore, by the discussion in Section \ref{sec:naux}, we can choose the extension such that \( \|\tilde{\bw}\|_{1,\infty}\leq C \|\bw \|_{C^1_0(\Gamma_N) }\), where \( C\) depends only on \( \Omega\). Hence, \(\bbT(t)\mathbf{n}\), \( \overline{\bbT}(t) \mathbf{n}\) are well-defined elements of the dual space \( (C^1_0(\Gamma_N)^d)^*\). 
Moreover, by the stated regularity results, we have \( \bbT(t) \mathbf{n}\), \(\overline{\bbT}(t) \mathbf{n}\in L^\infty_{w^*}(0, T; (C^1_0(\Gamma_N)^d)^*) \). Then we define the penalisation \( \tilde{\bg}\) on \( (0,T) \times \Gamma_N \) by
\begin{equation}\label{a:equ55}
\langle \tilde{\bg}(t), \bw \rangle  = \langle\overline{\bbT}(t) \mathbf{n} - \bbT(t) \mathbf{n}, \bw\rangle,
\end{equation}
for a.e. \( t\in (0, T) \), for every \( \bw\in C^1_0(\Gamma_N)^d\). 
Using (\ref{a:equ28}), (\ref{a:equ29}) and (\ref{a:equ55}), it follows that (\ref{a:equ30}) holds. 
Hence we conclude the existence of a solution to the elastodynamic equation with strain-limiting constitutive relation, up to the Neumann boundary condition. 
\end{proof}


\begin{proposition}\label{a:prop5}
Let the assumptions of Theorem \ref{a:thm2} hold and let \((\bu, \bbT,v, \tilde{\bg}) \) be the solution quadruple constructed there. For a.e. \( t\in (0, T) \), \( v(t) \) is the unique solution of the minimisation problem 
\begin{equation}\label{a:equ32}
\mathcal{E}(\bu(t), v(t)) + \mathcal{H}(v(t)) + \mathcal{G}_k(v(t), v_t(t)) = \inf_{v\in H^k_{D+1}(\Omega), \, v\leq v(t)} \Big\{ \mathcal{E}(\bu(t), v) + \mathcal{H}(v) + \mathcal{G}_k(v, v_t(t))\Big\}.
\end{equation}
Furthermore, \( v_t \leq 0 \) a.e. in \(Q\).
\end{proposition}
\begin{proof}
Since \( v^n_t\leq 0 \) in \( Q\) and weak convergence preserves ordering, it is immediate that \( v_t \leq 0 \) a.e. in \( Q\). To prove (\ref{a:equ32}), it is enough to show, for a.e. \( t\in (0, T) \) and  every \( \chi \in H^k_D(\Omega) \) such that \( \chi \leq 0 \), that the following holds:
\begin{equation}\label{a:equ63}
0 \leq \big[ \partial_v\mathcal{E}(\bu(t), v(t)) + \mathcal{H}^\prime(v(t)) + \partial_v\mathcal{G}_k(v(t), v_t(t)) \big](\chi).
\end{equation}
For every \( n \) and every test function \( \chi \) as above, we know that the solution of the regularised problem satisfies
\begin{align*}
0 \leq \big[ \partial_v\mathcal{E}_n(\bu^n(t), v^n(t)) + \mathcal{H}^\prime(v^n(t)) + \partial_v\mathcal{G}_k(v^n(t), v_t^n(t)) \big](\chi).
\end{align*}
Using the convergence results of Theorem \ref{a:thm2}, we immediately deduce that
\begin{equation}\label{a:equ33}
\lim_{n\rightarrow\infty}\int_0^T \psi \big[ \mathcal{H}^\prime(v^n(t)) + \partial_v\mathcal{G}_k(v^n(t), v_t^n(t)) \big](\chi)\,\mathrm{d}t
=  \int_0^T \psi \big[\mathcal{H}^\prime(v(t)) + \partial_v\mathcal{G}_k(v(t), v_t(t)) \big](\chi)\,\mathrm{d}t,
\end{equation}
for a fixed but arbitrary  \( \psi \in C([0, T]) \) such that \( \psi \geq 0 \). 

For  the nonlinear term, first we claim that \( \beps(\bu^n) \rightarrow\beps(\bu ) \) pointwise a.e. in \( Q\). Recalling  that \( (\bu^n)_n \) converges weakly in \( W^{1,2}(0, T; W^{2,2}_{loc}(\Omega)^d) \) and applying the Aubin--Lions lemma, it follows  that \( \bu^n \rightarrow\bu \) strongly in \( C([0, T]; W^{1,2}_{loc}(\Omega)^d) \). Considering an exhaustion of \(\Omega\) yields the required pointwise convergence result. 

Next, we claim that \((\varphi^*(\beps(\alpha \bu^n)))_n \) converges pointwise a.e. on \( Q\) to \( \varphi^*(\beps(\alpha\bu)) \). 
Using that \( \beps(\bu_t + \alpha\bu) = F(\bbT ) \) with the memory kernel property and the safety strain condition (\ref{a:equ5}), we have
\begin{align*}
|\beps(\bu(t)) | &= \Big| \mathrm{e}^{-\alpha t}\beps(\bu_0) + \int_0^t \mathrm{e}^{\alpha(s-t) } F(\bbT(s)) |\,\mathrm{d}s\Big|
\\
&\leq \frac{\mathrm{e}^{-\alpha t}}{\alpha}C_* + \int_0^t \mathrm{e}^{\alpha (s-t) }\,\mathrm{d}t
\\
& = \frac{1}{\alpha}\Big( 1 + \mathrm{e}^{-\alpha T}(C_* - 1) \Big)
\\
&= \frac{C_{**}}{\alpha},
\end{align*}
where \( C_{**} \in (0, 1) \). In particular,  \( \|\beps(\alpha\bu) \|_{L^\infty(Q)} \leq C_{**}< 1\). For a.e. \( (t,x) \in Q\), for  \( n \) sufficiently large (depending on  \((t,x) \)) we have that  \( |\beps(\alpha\bu^n(t,x))| \) is  bounded away from \( 1\) as a consequence of the pointwise convergence result. The function \( \varphi^*\) is well-defined and finite on the open unit ball. Also \( \varphi^*_n(\bbT) \rightarrow\varphi^*(\bbT) \) as \( n \rightarrow\infty\) for every \(|\bbT|<1 \). Thus, we deduce that
\begin{align*}
\varphi^*_n(\beps(\alpha\bu^n)) \rightarrow\varphi^*(\beps(\alpha\bu)) \quad\text{ pointwise a.e. in }Q.
\end{align*}
Combining this with the pointwise convergence of \( (v^n)_n \), we get that
\[
\psi \frac{b^\prime(v^n)}{\alpha}\chi \varphi^*_n(\beps(\alpha\bu^n)) \rightarrow \psi \frac{b^\prime(v)}{\alpha}\chi \varphi^*(\beps(\alpha\bu))\quad\text{ pointwise a.e. in }Q.
\]
The sequence on the left is non-positive a.e. in \( Q\), using the non-positivity of \( \chi \). Hence, by the reverse Fatou lemma, we deduce that
\begin{equation}\label{a:equ56}
\begin{aligned}
\int_0^T \psi \partial_v\mathcal{E}(\bu(t), v(t)) (\chi)\,\mathrm{d}t &= \int_Q \psi \frac{b^\prime(v)}{\alpha}\chi \varphi^*(\beps(\alpha\bu))\,\mathrm{d}x\,\mathrm{d}t
\\
&\geq \lim_{n\rightarrow\infty} \int_Q\psi \frac{b^\prime(v^n)}{\alpha}\chi \varphi^*_n(\beps(\alpha\bu^n)) \,\mathrm{d}x\,\mathrm{d}t.
\end{aligned}
\end{equation}
Combining (\ref{a:equ56}) with (\ref{a:equ33}) and recalling that \( \psi \) is arbitrary, we see that (\ref{a:equ63}) holds and so we conclude the stated result. 
\end{proof}

Using the pointwise convergence results,  Fatou's lemma and the weak lower semi-continuity of norms, taking the limit as \( n \rightarrow\infty \) in  the energy-dissipation equality for the regularised solution immediately  yields the following result.

\begin{proposition}\label{a:prop6}
Let the assumptions of Theorem \ref{a:thm2} hold and let \((\bu, \bbT,v, \tilde{\bg}) \) be the solution quadruple constructed there. For every \( t\in [0, T]\), the following energy-dissipation inequality holds:
\begin{align*}
&\mathcal{F}(t;\bu(t), \bu_t(t), v(t)) + \int_0^t \|v_t(s) \|_{k,2}^2 \,\mathrm{d}s + \int_0^t \langle l_t(s), \bu(s) \rangle\,\mathrm{d}s
\\
&\quad
+ \int_0^t \int_\Omega b(v) \big[ F^{-1}(\beps(\bu_t + \alpha\bu)) - F^{-1}(\beps(\alpha\bu)) \big]:\beps(\bu_t) \,\mathrm{d}x\,\mathrm{d}s
\\
&\leq 
\mathcal{F}(0; \bu_0, \bu_1, v_0).
\end{align*}
\end{proposition}

Before trying to obtain the opposite inequality, we  discuss equivalent formulations of the  energy-dissipation equality to motivate our proof in the case of fully Dirichlet boundary conditions, as well as to see why it fails in the case of mixed boundary conditions. We note that we cannot adapt the proof of Proposition \ref{p:prop6} to this setting due to  the lack of satisfaction of the elastodynamic equation without a penalisation term. 

Differentiating and carefully applying the chain rule for weak derivatives, an equivalent formulation of the energy-dissipation equality is
\begin{align*}
0 &= \Big\{ \int_\Omega \bu_{tt}(t) \cdot \bu_t(t) + b(v(t)) \bbT(t) :\beps(\bu_t(t)) \,\mathrm{d}x - \langle l(t), \bu_t(t) \rangle\Big\}
\\
&\quad
+ \Big[ \int_\Omega \frac{b^\prime(v(t))v_t(t)}{\alpha} \varphi^*(\beps(\alpha\bu(t)))  + \frac{1}{2\epsilon} (v(t) - 1) v_t(t) + 2\epsilon\nabla v(t) \cdot \nabla v_t(t) \,\mathrm{d}x + \|v_t(t) \|_{k,2}^2\Big].
\end{align*}
The  terms in the curly brackets correspond to the elastodynamic equation, and the second set of brackets is exactly 
\begin{equation}\nonumber
\big[ \partial_v\mathcal{E}(\bu(t), v(t))  + \mathcal{H}^\prime(v(t)) + \partial_v\mathcal{G}_k(v(t), v_t(t)) \big](v_t(t))
\end{equation}
 The latter collection of terms in bounded below by zero. Hence, to show that the sum of terms in the square brackets vanishes, we must prove that
\begin{equation}\label{a:equ34}
\begin{aligned}
\int_\Omega \frac{b^\prime(v(t))v_t(t)}{\alpha} \varphi^*(\beps(\alpha\bu(t)))  + \frac{1}{2\epsilon} (v(t) - 1) v_t(t) + 2\epsilon\nabla v(t) \cdot \nabla v_t(t) \,\mathrm{d}x + \|v_t(t) \|_{k,2}^2 \leq 0 .
\end{aligned}
\end{equation}
If (\ref{a:equ34}) holds, the energy-dissipation equality reduces to showing that 
\begin{equation}\label{a:equ35}
\int_\Omega \bu_{tt}(t) \cdot \bu_t(t) + b(v(t)) \bbT(t) :\beps(\bu_t(t)) \,\mathrm{d}x - \langle l(t), \bu_t(t) \rangle = 0,
\end{equation}
for a.e. \( t\in (0, T) \). 
Clearly (\ref{a:equ35}) holds if \( \Gamma_N \neq \emptyset\). Otherwise, it is sufficient to prove that
\begin{align*}
\lim_{n\rightarrow\infty}\int_Q b(v^n) \bbT^n : \beps(\bu^n_t) \,\mathrm{d}x \,\mathrm{d}t = \int_Q b(v) \bbT : \beps(\bu_t) \,\mathrm{d}x\,\mathrm{d}t.
\end{align*}
However, due to a lack of convergence results for \((\bbT^n)_n \) that are global in space, it is not clear how to prove such a result in the mixed boundary case. Hence, from now on, we only consider the case \( \Gamma_D = \partial\Omega\) and so, to obtain the energy-dissipation equality it suffices to prove that (\ref{a:equ34}) holds. 

Using weak lower semi-continuity and the fact that \( k > 1\), we see that, for every \( \psi \in C([0, T]) \) such that \( \psi \geq 0 \), we have
\begin{align*}
0 &= \lim_{n\rightarrow\infty} \int_0^T \psi(t) \Big\{\int_\Omega \Big[ \frac{b^\prime(v^n(t))v_t^n(t)}{\alpha} \varphi^*_n(\beps(\alpha\bu^n(t)))  + \frac{1}{2\epsilon} (v^n(t) - 1) v_t^n(t) 
\\&\quad
+ 2\epsilon\nabla v^n(t) \cdot \nabla v_t^n(t) \,\mathrm{d}x \Big]
+ \|v_t^n(t) \|_{k,2}^2\Big\} \,\mathrm{d}t
\\
&\geq \lim_{n\rightarrow\infty}\Big[ \int_Q \psi \frac{b^\prime(v^n)v_t^n}{\alpha} \varphi^*_n(\beps(\alpha\bu^n))\,\mathrm{d}x\,\mathrm{d}t\Big] 
\\&\quad
+ \int_0^T \psi \Big\{\int_\Omega\frac{1}{2\epsilon} (v(t) - 1) v_t(t) + 2\epsilon\nabla v(t) \cdot \nabla v_t(t) \,\mathrm{d}x 
 + \|v_t(t) \|_{k,2}^2\Big\} \,\mathrm{d}t.
\end{align*}
If we can show that 
\begin{equation}\label{a:equ64}
\lim_{n\rightarrow\infty}\Big[ \int_Q \psi \frac{b^\prime(v^n)v_t^n}{\alpha} \varphi^*_n(\beps(\alpha\bu^n))\,\mathrm{d}x\,\mathrm{d}t\Big]
\geq \int_Q \psi \frac{b^\prime(v)v_t }{\alpha}\varphi^*(\beps(\alpha\bu)) \,\mathrm{d}x\,\mathrm{d}t,
\end{equation}
then (\ref{a:equ34}) follows. 
At best, we have a weak convergence result for \((v^n_t)_n \). Hence we need a strong convergence result for \((\varphi^*_n(\beps(\alpha\bu^n)))_n \) in order to compute the limit in (\ref{a:equ64}). At present, we only have  pointwise convergence. We   show that this pointwise result for \((\varphi^*_n(\beps(\alpha\bu^n)))_n \) can be improved to strong convergence in   \( L^1(0, T; L^1_{loc}(\Omega))\). 

We first show that  \( \bbT^n \rightarrow\bbT \) strongly in \( L^1(0, T; L^1_{loc}(\Omega)^{d\times d}) \). 
If \( a\in (0, \frac{2}{d}) \), we later improve the result on \((\bbT^n)_n \) to strong convergence in \( L^{1 + \delta}(0, T; L^{1 + \delta}_{loc}(\Omega)^{d\times d})\) for a \( \delta  = \delta(a) > 0 \). 
Using that \((v^n_t)_n \) and \( v_t \) vanish on the boundary of \( \Omega\), we then prove that
\begin{equation}\label{a:equ57}
\lim_{n\rightarrow\infty}\Big[ \int_Q \psi \frac{b^\prime(v^n)v_t^n}{\alpha} \varphi^*_n(\beps(\alpha\bu^n))\,\mathrm{d}x\,\mathrm{d}t\Big] = \int_Q \psi \frac{b^\prime(v)v_t }{\alpha}\varphi^*(\beps(\alpha\bu)) \,\mathrm{d}x\,\mathrm{d}t.
\end{equation}
The energy-dissipation equality follows from (\ref{a:equ57}) using the aforementioned arguments. As a result, we conclude the existence of a weak energy solution to the strain-limiting problem with Dirichlet boundary conditions. 

\begin{lemma}\label{a:lem12}
Suppose that the hypotheses of Theorem \ref{a:thm2} hold and additionally assume that \( \Gamma_N = \emptyset\). 
Let \((\bu, \bbT, v) \) be the solution triple from Theorem \ref{a:thm2} and \(((\bu^n, \bbT^n, v^n))_n \) the sequence of solutions to the regularised problem constructed in  Proposition \ref{a:prop2}. Then  
\begin{align*}
\bbT^n \rightarrow\bbT \quad\text{ strongly in }L^1(0, T; L^1_{loc}(\Omega)^{d\times d}) . 
\end{align*}
\end{lemma}

\begin{proof}
Using the pointwise convergence result, it is enough to show that  \(|\bbT^n|\rightarrow|\bbT |\) strongly  in \( L^1(0, T; L^1_{loc}(\Omega)) \). 
There exists a constant \( C\), independent of \( n \), such that 
\begin{align*}
\sup_{t\in [0, T]}\int_\Omega \frac{|\bbT^n(t) |^2}{n} \,\mathrm{d}x + \int_Q \frac{|\bbT^n_t|^2}{n} \,\mathrm{d}x\,\mathrm{d}t + \frac{1}{c(\Omega_0)}\int_0^T \int_{\Omega_0} \frac{|\nabla \bbT^n|^2}{n}\,\mathrm{d}x\,\mathrm{d}t
\leq C,
\end{align*}
for every compact  subset \( \Omega_0\subset \Omega \).  Applying the Aubin--Lions lemma to  \((n^{-\frac{1}{2}}\bbT^n)_n \), it follows that  \( n^{-\frac{1}{2}}\bbT^n \rightarrow \mathbf{0}\) strongly in \( L^2(0,T; L^2_{loc}(\Omega)^{d\times d}) \). 

On the other hand, using the elastodynamic equation and the  convergence results for \((\bu^n)_n \), we have that
\begin{align*}
\lim_{n\rightarrow\infty} \int_Q b(v^n) \bbT^n : F_n(\bbT^n) \,\mathrm{d}x\,\mathrm{d}t = \int_Q b(v) \bbT :F(\bbT) \,\mathrm{d}x\,\mathrm{d}t.
\end{align*}
By the pointwise convergence result and the non-negativity of the integrands, it follows that 
\[
b(v^n) \bbT^n:F_n(\bbT^n) \rightarrow b(v) \bbT :F(\bbT) \quad\text{ strongly in }L^1(0, T; L^1(\Omega)) .
\]
It follows that 
\begin{align*}
\bbT^n:F_n(\bbT^n) \rightarrow \bbT : F(\bbT) \quad \text{ strongly in } L^1(0, T; L^1(\Omega)).
\end{align*}
Using the local convergence result for \((n^{-\frac{1}{2}}\bbT^n)_n \) in \(L^2(0, T; L^2_{loc}(\Omega)^{d\times d})\), we see that
\begin{equation}\label{a:equ38}
\bbT^n :F(\bbT^n) \rightarrow \bbT : F(\bbT) \quad\text{ strongly in } L^1(0, T; L^1_{loc}(\Omega)).
\end{equation}
We write
\begin{equation}\label{a:equ37}
|\bbT^n| = \Big( \frac{|\bbT^n|^2}{(1 + |\bbT^n|^a)^{\frac{1}{a}}}\Big) \cdot \Big( \frac{(1 + |\bbT^n|^a)^{\frac{1}{a}}}{|\bbT^n|} \Big) \cdot \mathbb{I}_{\{ |\bbT^n| > 1\} }  + |\bbT^n|\cdot  \mathbb{I}_{\{ |\bbT^n| \leq 1\} } .
\end{equation}
The indicator functions \( \mathbb{I}_{\{ |\bbT^n| > 1\} } \), \( \mathbb{I}_{\{ |\bbT^n| \leq 1\} } \) converge pointwise a.e. in \( Q\) to \( \mathbb{I}_{\{ |\bbT| > 1\} } \) and \( \mathbb{I}_{\{ |\bbT| \leq  1\} } \), respectively. 
Applying the dominated convergence theorem and the pointwise convergence result, we get
\begin{align*}
|\bbT^n|\cdot  \mathbb{I}_{\{ |\bbT^n| \leq 1\} }\rightarrow |\bbT |\cdot  \mathbb{I}_{\{ |\bbT| \leq 1\} } \quad\text{ strongly in }L^1(0, T; L^1(\Omega)).
\end{align*}
For the other term on the right-hand side of (\ref{a:equ37}), the second factor is uniformly bounded by \( 2^{\frac{1}{a}}\) and converges pointwise a.e. on \(Q\). The first factor converges strongly in \( L^1(0, T; L^1_{loc}(\Omega)) \) by (\ref{a:equ38}). Thus we have 
\begin{align*}
\Big( \frac{|\bbT^n|^2}{(1 + |\bbT^n|^a)^{\frac{1}{a}}}\Big) \cdot \Big( \frac{(1 + |\bbT^n|^a)^{\frac{1}{a}}}{|\bbT^n|} \Big) \cdot \mathbb{I}_{\{ |\bbT^n| > 1\} } &\rightarrow \Big( \frac{|\bbT|^2}{(1 + |\bbT|^a)^{\frac{1}{a}}} \Big) \cdot  \Big( \frac{(1 + |\bbT|^a)^{\frac{1}{a}}}{|\bbT|} \Big) \cdot \mathbb{I}_{\{ |\bbT| > 1\} } 
\\&= |\bbT|\cdot  \mathbb{I}_{\{ |\bbT| > 1\} },
\end{align*}
with respect to strong convergence in \( L^1(0, T; L^1_{loc}(\Omega))\). Thus the convergence asserted in the statement of the lemma follows. 
\end{proof}


\begin{lemma}\label{a:lem9}
Suppose that the hypotheses of Lemma \ref{a:lem12} hold. Furthermore, we also assume that \( a\in (0, \frac{2}{d})\). 
Let \((\bu, \bbT, v) \) be the solution triple from Theorem \ref{a:thm2} and \(((\bu^n, \bbT^n, v^n))_n \) the sequence of solutions to the regularised problem constructed in  Proposition \ref{a:prop2}. Let \( \Omega_0 \subset\Omega\) be a compact  subset.  For every \( \delta>0 \) such that \( a + \delta < \frac{2}{d}\), there exists  a constant \( C = C(\Omega_0)  \), independent of \( n \), such that
\begin{align*}
\int_0^T \int_{\Omega_0} |\bbT^n|^{1 + \delta} + |\bbT|^{1 + \delta}\,\mathrm{d}x\,\mathrm{d}t \leq C.
\end{align*}
Furthermore, \( \bbT^n \rightarrow\bbT \) strongly in \( L^{1+ {\delta}}(0, T; L^{1+ {\delta}}_{loc}(\Omega)^{d\times d}) \) for every \( \delta\) as above. 
\end{lemma}
\begin{proof}
First we recall that there exists a constant \( C = C(\Omega_0)\) independent of \( n\) such that 
\begin{align*}
C&\geq \sup_{t\in [0, T]}\int_\Omega|\bbT^n(t) |\,\mathrm{d}x + \int_0^T \int_{\Omega_0} \frac{|\nabla \bbT^n|^2}{(1 + |\bbT^n|^a)^{\frac{1}{a}}}\,\mathrm{d}x\,\mathrm{d}t
\\
& \geq  c_a \Big( c_p \int_0^T \|( 1 + |\bbT^n(t)|)^{\frac{1-a}{2}}\|_{L^p(\Omega_0)} \,\mathrm{d}t - 1\Big) 
\\
&\geq c_a c_p \Big( \int_0^T \|\bbT^n\|_{L^{\frac{p(1-a)}{2}}(\Omega_0)}^{1-a} \,\mathrm{d}t\Big),
\end{align*}
where \( p \in (2,\infty) \) if \( d = 2\) and \( p \in (2,\frac{2d}{d-2}] \) if \( d> 2\), by  the Sobolev embedding theorem applied to the function \((1 + |\bbT^n(t) |)^{\frac{1-a}{2}}\). We consider the case \( d> 2\).  The argument for \( d = 2\) is similar. Letting \( p = \frac{2d}{d-2}\), it follows that
\begin{equation}\label{a:equ36}
\sup_{t\in [0, T]}\int_\Omega |\bbT^n(t) |\,\mathrm{d}x + \int_0^T \Big( \int_{\Omega_0} |\bbT^n(t) |^{\frac{d(1-a)}{d-2}}\,\mathrm{d}x\Big)^{\frac{d-2}{d}}\,\mathrm{d}t \leq C,
\end{equation}
where \( C\) is independent of \( n \). 
 We fix \( \delta  = \delta(a,d) > 0 \) such that
\( a + \delta\in (0, \frac{2}{d}) \), which is possible due to the assumption on \( a\). A standard manipulation shows  that
\begin{align*}
\frac{d(1 + \delta) - 2}{d-2} < \frac{d(1-a)}{d-2}.
\end{align*}
Applying H\"{o}lder's inequality with parameters \( \frac{d}{2}\) and \( \frac{d}{d-2}\), we get
\begin{align*}
\int_{\Omega_0} |\bbT^n(t) |^{1 + \delta}\,\mathrm{d}x & \leq \Big( \int_{\Omega_0}|\bbT^n(t) |\,\mathrm{d}x \Big)^{\frac{2}{d}}\Big( \int_{\Omega_0} |\bbT^n(t) |^{ 1 + \frac{\delta d}{d-2}}\,\mathrm{d}x\Big)^{\frac{d-2}{d}}
\\
&\leq C\Big( \int_{\Omega_0} |\bbT^n(t) |^{\frac{d(1 + \delta) - 2}{d-2} }\,\mathrm{d}x \Big)^{\frac{d-2}{d}}
\\
&\leq C \Big( \int_{\Omega_0} |\bbT^n(t) |^{\frac{d(1-a)}{d-2}}\,\mathrm{d}x \Big)^{\frac{d-2}{d}}
\end{align*}
Integrating over \((0, T) \) and using (\ref{a:equ36}), it follows that
\begin{align*}
\int_0^T \int_{\Omega_0} |\bbT^n(t) |^{ 1 + \delta}\,\mathrm{d}x\,\mathrm{d}t\leq C.
\end{align*}
By reflexivity, we deduce that \(\bbT^n \rightharpoonup \bbT \) weakly in \( L^{1 + \delta}(0, T; L^{1 + \delta}_{loc}(\Omega)^{d\times d}) \). 
Combining the weak and pointwise convergence results  yields the   strong convergence result. 
\end{proof}

With these strong convergence results in mind, we return to focusing on proving (\ref{a:equ64}). We use the strong convergence of \((\bbT^n)_n \) to first show that \((\varphi^*_n(\beps(\bu^n_t + \alpha\bu^n)))_n \) converges strongly in \( L^1(0, T; L^1_{loc}( \Omega)^{d\times d}) \). Then we make use of the memory kernel property to show that \( (\varphi^*_n(\beps( \alpha\bu^n)))_n \) converges strongly in \( L^1(0, T; L^1_{loc}( \Omega)) \) to \( \varphi^*(\beps(\alpha\bu)) \). We combine this with the weak-* convergence of \((v^n_t)_n \) in \( L^\infty(0, T; L^\infty(\Omega))\) to  deal with the limit in (\ref{a:equ64}) away from the boundary of \( \Omega\). We use the auxiliary result, Corollary \ref{a:cor6}, to control the limit near the boundary. This is another place where we make use of the fully Dirichlet boundary condition. 

\begin{lemma}\label{a:lem13}
Let the assumptions of Lemma \ref{a:lem12} hold. Then we have that
\begin{align*}
\varphi^*_n(\beps(\alpha\bu^n)) \rightarrow\varphi^*(\beps(\alpha\bu)) \quad\text{ strongly in }L^1(0, T; L^1_{loc}(\Omega)) .
\end{align*}
\end{lemma}

\begin{proof}
First, we  show that \( \varphi^*_n(\beps(\bu^n_t + \alpha\bu^n))\rightarrow\varphi^*(\beps(\bu_t + \alpha\bu)) \) strongly in \( L^1(0, T; L^1_{loc}(\Omega)) \). From the proof of Lemma \ref{a:lem12}, we know that \( (n^{-1}|\bbT^n|^2 )_n \) converges strongly to \( 0 \) in \( L^1(0, T; L^1_{loc}(\Omega) ) \). 
Using the dominated convergence theorem and Lemma \ref{a:lem12}, it follows that \( \varphi_n(\bbT^n) \rightarrow\varphi(\bbT) \) strongly in \( L^1(0, T; L^1_{loc}(\Omega))\). Using this with the strong convergence of \((\bbT^n:F_n(\bbT^n))_n\) and (\ref{intro:equ13}),  we deduce that
\begin{align*}
\varphi^*_n(\beps(\bu^n_t + \alpha\bu^n)) \rightarrow\varphi^*(\beps(\bu_t + \alpha\bu)) \quad\text{ strongly in }L^1(0, T; L^1_{loc}(\Omega)) .
\end{align*}
Now we   prove the   convergence result for \((\varphi^*_n(\beps(\alpha\bu^n)))_n\). By the memory kernel property and  Jensen's inequality, it follows that
\begin{equation}\label{a:equ65}
\varphi^*_n(\beps(\alpha\bu^n(t))) \leq \varphi^*_n(\beps(\alpha\bu_0)) + \int_0^t \alpha\mathrm{e}^{\alpha(s-t)} \varphi^*_n(\beps(\bu^n_t + \alpha\bu^n)) \,\mathrm{d}s.
\end{equation}
By the   pointwise convergence result for \((\varphi^*_n(\beps(\alpha\bu^n)))_n \) and the strong convergence result for \((\varphi^*_n(\beps(\bu^n_t + \alpha\bu^n)))_n \), we have
\begin{equation}\label{a:equ66}
\begin{aligned}
0 &\leq \limsup_{n\rightarrow\infty}\Big\{ \varphi^*_n(\beps(\alpha\bu_0)) + \int_0^t \alpha\mathrm{e}^{\alpha(s-t)} \varphi^*_n(\beps(\bu^n_t + \alpha\bu^n)) \,\mathrm{d}s - \varphi^*_n(\beps(\alpha\bu^n(t))) \Big\}
\\
&= \varphi^*(\beps(\alpha\bu_0)) + \int_0^t \alpha\mathrm{e}^{\alpha(s-t)} \varphi^*(\beps(\bu_t + \alpha\bu)) \,\mathrm{d}s - \varphi^*(\beps(\alpha\bu(t))),
\end{aligned}
\end{equation}
a.e. in \( \Omega\), for a.e. \( t\in (0, T) \). However, using the non-negativity from (\ref{a:equ65}) and (\ref{a:equ66}) with Fatou's lemma and the  strong convergence of  \((\varphi^*_n(\beps(\bu^n_t + \alpha\bu^n)))_n\) in \( L^1(0, T; L^1_{loc}( \Omega))\), it follows that
\begin{align*}
0 \leq &
\int_0^T \int_{\Omega_0} \varphi^*(\beps(\alpha\bu_0))  - \varphi^*(\beps(\alpha\bu(t))) + \Big[ \int_0^t \alpha\mathrm{e}^{\alpha(s-t)} \varphi^*(\beps(\bu_t + \alpha\bu)) \,\mathrm{d}s \Big]\,\mathrm{d}x\,\mathrm{d}t
\\
&\leq \liminf_{n\rightarrow\infty} \int_0^T \int_{\Omega_0} \varphi^*_n(\beps(\alpha\bu_0)) - \varphi^*_n(\beps(\alpha\bu^n(t))) +\Big[  \int_0^t \alpha\mathrm{e}^{\alpha(s-t)} \varphi^*_n(\beps(\bu^n_t + \alpha\bu^n)) \,\mathrm{d}s \Big] \,\mathrm{d}x\,\mathrm{d}t
\\
&= \int_0^T \int_{\Omega_0} \varphi^*(\beps(\alpha\bu_0)) + \Big[ \int_0^t \alpha\mathrm{e}^{\alpha(s-t)} \varphi^*(\beps(\bu_t + \alpha\bu)) \,\mathrm{d}s \Big]\,\mathrm{d}x\,\mathrm{d}t 
\\&\quad
- \limsup_{n\rightarrow\infty} \Big(   \int_0^T \int_{\Omega_0} \varphi^*_n(\beps(\alpha\bu^n(t))) \,\mathrm{d}x\,\mathrm{d}t\Big).
\end{align*}
Cancelling out  terms from the left- and right-hand sides, and rearranging the remaining two terms, it follows that
\begin{align*}
\limsup_{n\rightarrow\infty} \int_0^T \int_{\Omega_0} \varphi^*_n(\beps(\alpha\bu^n(t))) \,\mathrm{d}x\,\mathrm{d}t& \leq \int_0^T \int_{\Omega_0}\varphi^*(\beps(\alpha\bu(t)))\,\mathrm{d}x\,\mathrm{d}t
\\
& \leq \liminf_{n\rightarrow\infty} \int_0^T \int_{\Omega_0} \varphi^*_n(\beps(\alpha\bu^n(t))) \,\mathrm{d}x\,\mathrm{d}t ,
\end{align*}
using Fatou's lemma again for the second inequality. The convergence of norms and the pointwise convergence yields the required result. 
\end{proof}

The final piece of information required before we prove (\ref{a:equ57}) concerns the sequence \( (v^n_t)_n \). We   show that  the functions \((v^n_t)_n \) vanish uniformly on the boundary. 
This allows us to overcome the fact that the convergence result from Lemma \ref{a:lem13} is only local in space. 
We require the following auxiliary result, the proof of which can be found in the Appendix. 
 
\begin{lemma}\label{a:lem11}
Let \( w\in W^{1,\infty}(\Omega) \) be such that \( w = 0 \) on \( \partial\Omega\) in the trace sense. There exists positive constants \( C\)  and \( \delta_0  \), depending only on \( \Omega\), such that
\begin{align*}
|w(x) |\leq C\|\nabla w\|_\infty d(x,\partial\Omega) \quad \text{ for a.e. }x\in \Omega\text{ such that }d(x,\partial\Omega)< \delta_0.
\end{align*}
\end{lemma}

\begin{corollary}\label{a:cor6}
Let the assumptions of Theorem \ref{a:thm2} hold and assume additionally that \(\Gamma_N  = \emptyset\). Let  \( \delta_0 \) be the constant from Lemma \ref{a:lem11}. For every \( \delta\in (0, \delta_0) \), define \( \Omega_\delta\) to be the set of \( x\in \Omega\) such that \( d(x, \partial\Omega) < \delta\). There exists a constant \( C\), independent of \( n \),  such that, for every \( \delta\in (0, \delta_0) \), 
\begin{equation}\label{a:equ58}
\|b^\prime(v^n) v^n_t \|_{L^\infty((0, T)\times\Omega_\delta)} \leq C\delta. 
\end{equation}
\end{corollary}
\begin{proof}
For \( x\in \Omega_\delta\) and a.e. \( t\in (0, T) \), we have
\begin{align*}
|b^\prime(v^n(t,x)) v^n_t(t,x)| &\leq \|b^\prime(v^n) \|_\infty|v^n_t(t,x) |
\\
&\leq C\|v^n \|_\infty \|\nabla v^n_t \|_\infty  d(x,\partial\Omega) 
\\
&\leq C\delta. 
\end{align*}
Taking the essential supremum over the left-hand side, we conclude the required result. 
\end{proof}

Now we have enough information to prove (\ref{a:equ57}). As a consequence, we see that the energy-dissipation equality holds in the case of fully Dirichlet boundary conditions, concluding the proof of the  existence result for the strain-limiting Dirichlet problem. 

\begin{lemma}
Let the assumptions of Lemma \ref{a:lem12} hold and let \( \psi \in C([0, T]) \) be fixed but arbitrary.  We have
\begin{align*}
\lim_{n\rightarrow\infty} \int_Q \psi \frac{b^\prime(v^n)v^n_t}{\alpha} \varphi^*_n(\beps(\alpha\bu^n)) \,\mathrm{d}x\,\mathrm{d}t = \int_Q \psi \frac{b^\prime(v)v_t}{\alpha}\varphi^*(\beps(\alpha\bu)) \,\mathrm{d}x\,\mathrm{d}t.
\end{align*}
\end{lemma}
\begin{proof}
Let  \( \delta\in (0, \delta_0 ) \) where we define \( \delta_0\) and  \(\Omega_\delta\)  as in Lemma \ref{a:lem11} and Corollary \ref{a:cor6}, respectively. Then, we have
\begin{equation}\label{a:equ59}
\begin{aligned}
&\Big|\int_Q  \psi \frac{b^\prime(v^n)v^n_t}{\alpha} \varphi^*_n(\beps(\alpha\bu^n))  -  \psi \frac{b^\prime(v)v_t}{\alpha}\varphi^*(\beps(\alpha\bu)) \,\mathrm{d}x\,\mathrm{d}t\Big|
\\&\leq
\Big|\int_0^T \int_{\Omega\setminus \Omega_\delta}  \psi \frac{b^\prime(v^n)v^n_t}{\alpha} \varphi^*_n(\beps(\alpha\bu^n))  -  \psi \frac{b^\prime(v)v_t}{\alpha}\varphi^*(\beps(\alpha\bu)) \,\mathrm{d}x\,\mathrm{d}t\Big|
\\
&\quad
+ \Big|\int_0^T \int_{ \Omega_\delta}  \psi \frac{b^\prime(v^n)v^n_t}{\alpha} \varphi^*_n(\beps(\alpha\bu^n))  -  \psi \frac{b^\prime(v)v_t}{\alpha}\varphi^*(\beps(\alpha\bu)) \,\mathrm{d}x\,\mathrm{d}t\Big|
\\
&\leq 
\Big|\int_0^T \int_{\Omega\setminus \Omega_\delta} \psi \varphi^*(\beps(\alpha\bu))\Big(\frac{b^\prime(v)v_t}{\alpha} - \frac{b^\prime(v^n)v^n_t}{\alpha}\Big) \,\mathrm{d}x\,\mathrm{d}t\Big|
\\
&\quad
+ \Big| \int_{\Omega\setminus \Omega_\delta} \psi \frac{b^\prime(v^n)v^n_t}{\alpha} \Big( \varphi^*(\beps(\alpha\bu)) - \varphi^*_n(\beps(\alpha\bu^n)) \big) \,\mathrm{d}x\,\mathrm{d}t \Big|
\\
&\quad
 \int_0^T \int_{\Omega_\delta} |\psi |\Big( \frac{|b^\prime(v^n)||v^n_t|}{\alpha}\varphi^*_n(\beps(\alpha\bu^n)) + \frac{|b^\prime(v)||v_t|}{\alpha}\varphi^*(\beps(\alpha\bu)) \Big) \,\mathrm{d}x\,\mathrm{d}t
 \\
 &\leq \Big|\int_0^T \int_{\Omega\setminus \Omega_\delta} \psi \varphi^*(\beps(\alpha\bu))\Big(\frac{b^\prime(v)v_t}{\alpha} - \frac{b^\prime(v^n)v^n_t}{\alpha}\Big) \,\mathrm{d}x\,\mathrm{d}t\Big|
 \\
 &\quad
 +C \int_0^T \int_{\Omega\setminus \Omega_\delta}\Big|\varphi^*_n(\beps(\alpha\bu^n)) - \varphi^*(\beps(\alpha\bu)) \Big| \,\mathrm{d}x\,\mathrm{d}t + C\delta,
\end{aligned}
\end{equation}
where \( C\) is a constant, independent of \( n \), using the \( L^\infty(Q)\) bound on \((v^n_t)_n \) from Lemma \ref{a:lem7} and (\ref{a:equ58}) to bound the integral over \( \Omega_\delta\) on the right-hand side. 

Using the local convergence results and noting  that \( \Omega\setminus \Omega_\delta\) is compactly contained in \( \Omega\), the second term on the right-hand side of (\ref{a:equ59}) vanishes in the limit as \( n\rightarrow\infty \). For the first term on the right, we use that \( (b^\prime(v^n)v^n_t )_n\) converges weakly-* in \( L^\infty(Q) \) as \( n \rightarrow\infty \)  with the fact that \( \psi \varphi^*(\beps(\alpha\bu)) \in L^1(Q) \). It follows that this integral also vanishes in the limit. This yields 
\begin{align*}
\lim_{n\rightarrow\infty} \Big|\int_Q  \psi \frac{b^\prime(v^n)v^n_t}{\alpha} \varphi^*_n(\beps(\alpha\bu^n))  -  \psi \frac{b^\prime(v)v_t}{\alpha}\varphi^*(\beps(\alpha\bu)) \,\mathrm{d}x\,\mathrm{d}t\Big| \leq C\delta,
\end{align*}
for every \( \delta\in (0, \delta_0)\), where \( C\) is  independent of \( \delta\). Letting \( \delta\rightarrow 0\), we conclude the stated result.
\end{proof}

\begin{corollary}
Let the hypotheses of Lemma \ref{a:lem9} hold and let \((\bu, \bbT, v) \) be the solution triple constructed in Theorem \ref{a:thm2}. Then \((\bu, \bbT, v) \) is a weak energy solution of the strain-limiting dynamic fracture problem in the following sense. For every \( \bw \in W^{1,2}_D(\Omega)^d\cap W^{1,\infty}(\Omega)^d\), we have
\begin{align*}
\int_\Omega\bu_{tt}(t) \cdot \bw + b(v(t)) \bbT(t) : \beps(\bw ) \,\mathrm{d}x = \int_\Omega\boldf(t) \cdot \bw \,\mathrm{d}x,
\end{align*}
for a.e. \( t\in (0, T) \), with constitutive relation
\begin{align*}
\beps(\bu_t + \alpha\bu) = \frac{\bbT}{(1  + |\bbT|^a)^{\frac{1}{a}}}\quad \text{ pointwise a.e. in }Q.
\end{align*} 
The minimisation problem (\ref{a:equ2}) holds for a.e. \( t\in [0, T]\) and the energy-dissipation equality (\ref{a:equ3}) is satisfied for every \( t\in [0, T]\). The initial conditions hold in the sense that
\begin{align*}
\lim_{t\rightarrow 0+}\big[ \|\bu(t) - \bu_0\|_{1,2} + \|\bu_t(t) - \bu_1 \|_{1,2} + \|v(t) - v_0 \|_{k,2} \big] = 0. 
\end{align*}
\end{corollary}

\appendix
\section{Proof of Lemma \ref{a:lem11}}
Let \( w\in W^{1,\infty}(\Omega) \) be such that \( w = 0 \) on \( \partial\Omega\) in the trace sense. 
The domain \( \Omega\) is Lipschitz, so for every \( x_0 \in \partial\Omega\) there exists an \( r> 0 \) and a Lipschitz function \( \gamma: \mathbb{R}^{d - 1}\rightarrow\mathbb{R}\) such that
\begin{align*}
\Omega\cap B(x_0, r) &= \{ x\in B(x_0, r) \,:\, \mathbf{Q}x = y = (y^\prime, y_d), \, y_d > \gamma(y^\prime) \},
\\
\partial\Omega\cap B(x_0, r) &= \{ x\in B(x_0, r) \,:\, \mathbf{Q}x = y = (y^\prime, y_d), \, y_d =  \gamma(y^\prime) \},
\\
\Omega^c\cap B(x_0, r) &= \{ x\in B(x_0, r) \,:\, \mathbf{Q}x = y = (y^\prime, y_d), \, y_d < \gamma(y^\prime) \},
\end{align*}
where \( \mathbf{Q}\) is an orthogonal change of coordinates. Defining the function \( \tilde{w} = w(\mathbf{Q}^{-1}\cdot )\), we have
\begin{equation}\label{a:equ67}
\begin{aligned}
|w(x)| = |\tilde{w}(y)| &= |\tilde{w}(y^\prime, y_d) - \tilde{w}(y^\prime, \gamma(y^\prime)) |
\\
&=\Big| \int_0^1 \partial_{y_d}\tilde{w}(y^\prime, sy_d + (1-s) \gamma(y^\prime)) \cdot ( y_d - \gamma(y^\prime))  \,\mathrm{d}s\Big|
\\
&\leq \|\nabla_y \tilde{w}\|_\infty |y_d - \gamma(y^\prime) |
\\
&\leq C(\mathbf{Q}) \|\nabla w\|_\infty |y - (y^\prime, \gamma(y^\prime)) |
\\
&\leq C \|\nabla w\|_\infty d(x,\partial\Omega) ,
\end{aligned}
\end{equation}
where the constant \( C\) depends on \(x_0\), \( r\) and \( \mathbf{Q}\), but is independent of \( w\). 

Next, we cover the boundary by finitely many balls of this form, that is, there exists \( K \in \mathbb{N}\), points \( x_1, \dots, x_K\in \partial\Omega\) and positive real values \( r_1,\dots , r_K\) such that \((B(x_i, r_i))_{i=1}^K \) is an open cover of \( \partial\Omega\), with each ball \( B(x_i, r_i) \) of the above form. 
In this case, there exists \( \delta_0 > 0 \) such that \( \Omega_{\delta_0}\subset \cup_{i=1}^K B(x_i, r_i) \). 
Using (\ref{a:equ67}), there exists a constant \( C\) depending only on \( \Omega\) such that,  for every \( x\in \cup_{i=1}^K B(x_i, r_i) \), 
\begin{equation}\label{a:equ68}
|w(x)|\leq C\|\nabla w\|_{\infty} d(x, \partial\Omega). 
\end{equation}
Hence, (\ref{a:equ68}) holds for every \( x\in \Omega_{\delta_0}\) and so we conclude the required result.

 \bibliographystyle{abbrv}

 \bibliography{references} 
 
 \end{document}